\documentclass[12pt]{amsart}
\usepackage[latin1]{inputenc}
\usepackage{amsmath, latexsym, amsfonts, amssymb, amsthm, amscd}
\usepackage[left=1.4cm, right=1.4cm, top=2.5cm, bottom=2.5cm]{geometry}
\usepackage[dvipsnames]{xcolor}
\usepackage{comment}

\usepackage{tikz}

\definecolor{textcol}{rgb}{0.37,0,0.57}
\usepackage{tikz}
\usepackage{forest}
\forestset{
  default preamble={
   for tree={circle, draw, 
            minimum size=2.1em, 
            inner sep=1pt} 
  }
}

\usetikzlibrary{positioning}
\definecolor{c1}{rgb}{0.84, 0.96, 0.88}
\usepackage{caption}
\usepackage{subcaption}
\numberwithin{equation}{section}

\newtheorem{corollary}{Corollary}

\newtheorem{lemma}{Lemma}
\newtheorem{rem}{Remark}
\newtheorem{defi}{Definition}
\newtheorem{theorem}{Theorem}

\newcommand{\sam}{\textcolor{Mahogany}}

\renewcommand{\P}{\mathbb{P}}
\newcommand{\e}{\mathbb{E}}
\newcommand{\ud}{\mathrm{d}}

\newcommand{\N}{\mathbb{N}}
\newcommand{\Z}{\mathbb{Z}}

\newcommand{\E}{\mathbb{E}}

\begin{document}

 \title[Coalescent structure of heavy-tailed GW trees]{Universality classes for the coalescent structure of heavy-tailed Galton-Watson trees}

\author{Simon C. Harris}
\address{University of Auckland, Private Bag 92019, Auckland 1142, New Zealand}
\email{simon.harris@auckland.ac.nz}

\author{Samuel G.G. Johnston}
\address{King's College London, Strand Building, London, WC2R 2LS, United Kingdom}
\email{sggjohnston@gmail.com}

\author{Juan Carlos Pardo}
\address{CIMAT A.C. Calle Jalisco s/n. C.P. 36240, Guanajuato, Mexico}
\email{jcpardo@cimat.mx}

\begin{abstract} 
Consider a population evolving as a critical continuous-time Galton-Watson (GW) tree.
Conditional on the population surviving until a large time $T$, sample $k$ 
individuals uniformly at random (without replacement) from amongst those alive at time $T$. What is the genealogy of this sample of individuals?
In cases where the offspring distribution has finite variance, the probabilistic properties of the joint ancestry of these $k$ particles are well understood, as seen in \cite{HJR20, J19}. 
In the present article, we study the joint ancestry of a sample of $k$ particles under the following  regime: 
the offspring distribution has mean $1$ (critical) and the tails of the offspring distribution are \emph{heavy} 
in that $\alpha \in (1,2]$ is the supremum over indices $\beta$ such that the $\beta^{\text{th}}$ moment is finite. 
We show that for each $\alpha$, after rescaling time by $1/T$, there is a universal stochastic process describing the joint coalescent structure of the $k$ distinct particles. 
The special case $\alpha = 2$ generalises the known case of sampling from critical GW trees with finite variance 
where only pairwise mergers are observed and the genealogical tree is, roughly speaking, some kind of mixture of time-changed Kingman coalescents.
The cases $\alpha \in (1,2)$ introduce new universal limiting partition-valued stochastic processes with interesting probabilistic structures
which, in particular,  have representations connected to the Lauricella function and the Dirichlet distribution,
and whose coalescent structures exhibit multiple-mergers of family lines.
Moreover, in the case $\alpha \in (1,2)$, we show that the coalescent events of the ancestry of the $k$ particles are associated with birth events that produce giant numbers of offspring of the same order of magnitude as the entire population size, and we compute the joint law of the ancestry together with the sizes of these giant births.
\end{abstract}
 \maketitle

 \vspace{0.2in}


\bigskip


 \vspace{0.5cm}

\section{Introduction}

\subsection{Continuous-time Galton-Watson trees and their coalescent processes}

Let $r > 0$ and let $\overline{p} = (p_i)_{i \in \mathbb{N}_0}$ be a probability mass function on the non-negative integers. Consider a continuous-time Galton-Watson tree with branching rate $r$ and offspring distribution $\overline{p}$, where we start from a single initial particle at time zero. The initial particle has an exponential lifetime with parameter $r$ (i.e.\ expected length $1/r$), and upon death is replaced by a random number $L$ of offspring particles, where $\mathbb{P}(L=i) = p_i$. 
Similarly, each offspring particle independently repeats the behaviour of their parent, and so on for all subsequent generations: each particle dies at rate $r$, and upon death is replaced by a random number of offspring distributed like $\overline{p}$. In this process we write $Z_t$ for the number of particles alive at time $t$.

Continuous-time Galton-Watson trees are endowed with a natural notion of genealogy: each particle living at some time $t$ had a unique ancestor particle living at each earlier time $s < t$. It is then natural to ask questions about the shared genealogy of different particles alive in the population alive at a certain time. 
Specifically, conditioning on the event $\{ Z_T \geq k\}$ that there are at least $k$ particles alive at a time $T > 0$, consider picking $k$ 
particles uniformly at random without replacement
from the population alive at time $T$. 
Label these $k$ sampled particles with the integers $1$ through $k$.  
Recalling some standard terminology, a collection of disjoint non-empty subsets of $\{1,\ldots,k\}$ whose union is $\{1,\ldots,k\}$ is known as a set partition of $\{1,\ldots,k\}$.
We may associate with our sample of $k$ labelled particles a stochastic process $\pi^{(k,T)} := (\pi^{(k,T)}_t)_{t \in [0,T]}$ taking values in the collection of set partitions of $\{1,\ldots,k\}$ by declaring
\begin{align} \label{eq:partition}
\text{$i$ and $j$ in the same block of $\pi^{(k,T)}_t$} \iff \text{$i$ and $j$ are descended from the same time $t$ ancestor},
\end{align}
where, more precisely, in \eqref{eq:partition} we mean that the time $T$ particle labelled with $i \in \{1,\ldots,k\}$ and the time $T$ particle labelled with $j \in \{1,\ldots,k\}$ share the same unique ancestor in the time $t$ population. 
This set partition process construction is also seen in, for example, \cite{BLG} and \cite{J19}.

Since the entire process begins with a single particle at time $0$, it follows that each of the $k$ particles share the same 
initial ancestor, and accordingly $\pi^{(k,T)}_0 = \{\{ 1,\ldots,k\}\}$, that is, $\pi^{(k,T)}_0$ is the partition of $\{1,\ldots,k\}$ into a single block. 
Conversely, since we choose uniformly without replacement, each of the particles are distinct at time $T$, hence $\pi^{(k,T)}_T = \{\{1\},\ldots,\{k\} \}$ is the partition of $\{1,\ldots,k\}$ into singletons. More generally, as $t$ increases across $[0,T]$, the stochastic process $\pi^{(k,T)}$ takes a range of values in the partitions of $\{1,\ldots,k\}$, with the property that the constituent blocks of the process break apart as time passes. With this picture in mind, we define the \emph{split times} 
\begin{align*}
\tau_1 < \ldots < \tau_m
\end{align*}
to be the times of discontinuity of $\pi^{(k,T)}$. That is, at each time $\tau_i$, a block of $\pi^{(k,T)}_{\tau_i -}$ breaks into several smaller blocks in $\pi^{(k,T)}_{\tau_i}$. We note that $\pi^{(k,T)} $ is almost-surely right continuous.

Numerous authors have studied the process $\pi^{(k,T)}$ in its various incarnations and in the setting of various continuous-time Galton-Watson trees (see Section \ref{sec:literature} for further discussion). Harris, Johnston and Roberts \cite{HJR20} studied the large $T$ asymptotics of the process $\pi^{(k,T)}$ in the setting where the offspring distribution is critical (i.e. $\sum_{i \geq 0 } ip_i = 1$) with finite variance (i.e. $\sum_{i \geq 0} i(i-1)p_i < \infty$). Under these conditions they established the convergence in distribution of the renormalised process $(\pi^{(k,T)}_{sT})_{s \in [0,1]}$ to a universal stochastic process $\nu^{(k,2)} := (\nu^{(k,2)}_s)_{s \in [0,1]}$ taking values in the set of partitions of $\{1,\ldots,k\}$. This limiting process $\nu^{(k,2)}$ is universal in the sense that it does not depend on the precise form of the offspring distribution, only that the distribution is (near) critical and has finite variance. Harris et al.\ \cite{HJR20} show that $\nu^{(k,2)}$ only exhibits binary splits (i.e.\ every discontinuity amounts to one block breaking into exactly two sub-blocks) where: 
if there are currently $i$ blocks of sizes $a_1,a_2,\dots,a_i$ the probability the next block to split is block $j$ is $(a_j-1)/(k-i)$;
for any block of size $a$ that splits, the size of its first sub-block is uniformly distributed on $\{1,2,\dots,a-1\}$; 
and, independently of the block topology, the joint distribution of the $k-1$ splitting times $0 < \tau_1 < \ldots < \tau_{k-1} < 1$ is given by 
\begin{align} \label{eq:binarysplits}
f_k(t_1,\ldots,t_{k-1}) = k\int_0^\infty\prod_{i=1}^{k-1}\left(\frac{\varphi}{(1 +\varphi(1-t_i))^2}\right)\frac{1}{(1+\varphi)^2}\ud \varphi~\ud t_1\cdots\ud t_{k-1}. 
\end{align}
Further, when
viewed backwards in time, this partition process $\nu^{(k,2)}$ has the same topology as Kingman's coalescent \cite{Kingman82}, that is, any two blocks are equally likely to be the next to merge. 

\subsection{Main results}
In the present article, we will consider Galton-Watson trees whose offspring distribution is critical but with heavy tails, ultimately discovering a new collection of  \emph{universal stochastic processes} $\{ \nu^{(k,\alpha)} :=  (\nu^{(k,\alpha)}_s)_{s \in [0,1]} : \alpha \in (1,2] \}$ that describe their limiting genealogical structures. 
We define the probability generating function (PGF) of the offspring distribution $\overline{p}=(p_i)_{i \in \mathbb{N}_0}$ by
\[
f(s):=\E[s^L]=\sum_{j\ge 0} p_js^j, \qquad \textrm{for} \quad s\in [0,1].
\]
Throughout, we will assume $p_0>0$.
For some $\alpha\in(1,2]$, suppose that the PGF of $\overline{p}$ can be written as
\begin{equation}\label{eq:hyp1}\tag{\bf{H1}}
f(s)=s+(1-s)^\alpha\ell\left(\frac{1}{1-s}\right)
\end{equation}
 where $\ell$ is a slowly varying function at infinity, that is  for any $\lambda>0$,
\[
\lim_{x\to \infty}\frac{\ell(\lambda x)}{\ell(x)}=1.
\]
Note, such an $\alpha$ will be unique and it can be verified that for $f$ of the form in \eqref{eq:hyp1}, we have $f'(1) = \sum_{ j \geq 0} jp_j  = 1$, that is, $f$ is the PGF of a critical offspring distribution. The higher moments are slightly more delicate. 
If $\alpha \in (1,2)$, we will see later 
that if $L$ is an integer-valued random variable whose moment generating function takes the form in \eqref{eq:hyp1}, then $\mathbb{E}[L^\kappa] < \infty$ whenever $\kappa < \alpha$, and $\mathbb{E}[L^\kappa] = \infty$ whenever $\kappa > \alpha$. The moment $\mathbb{E}[L^\alpha]$ itself may be either finite or infinite. 
In the setting where $\alpha = 2$, however, it turns out that moment generating functions of the form in \eqref{eq:hyp1} encompass those of all unit-mean probability distributions on the non-negative integers with finite variance. 
On the other hand, as we will observe at \eqref{a2infinite}, it is possible to find infinite variance random variables whose moment generating functions take the form \eqref{eq:hyp1} with $\alpha = 2$.

We are now ready to state our first main result on universality classes for the coalescent structure of  Galton-Watson trees 
with heavy-tailed offspring distributions.

\begin{theorem} \label{thm:A}
Consider a continuous-time Galton-Watson tree with branching rate $r$ and a critical offspring distribution $\overline{p}$ whose moment generating function satisfies \eqref{eq:hyp1} for some $\alpha\in(1,2]$. 
Conditional on $\{Z_T \geq k \}$, let $(\pi^{(k,T)}_t)_{t \in [0,T]}$ denote the ancestral process associated with $k$ particles sampled uniformly at random without replacement from the population alive at time $T$. 
Then, the ancestral process converges in distribution for large times with
\begin{align*}
\left( \pi^{(k,T)}_{sT} \right)_{s \in [0,1]} \implies (\nu^{(k,\alpha)}_s)_{s \in [0,1]}  \qquad \text{as} \quad T \to \infty,
\end{align*}
where, for each $\alpha \in (1,2]$, $ (\nu^{(k,\alpha)}_s)_{s \in [0,1]}$ is a stochastic process taking values in the set of partitions of $\{1,\ldots,k\}$ which is \emph{universal} in the sense that its law depends only on the moment index $\alpha$, but not on the specific offspring distribution.
In other words, for each $\alpha\in(1,2]$, the set of Galton-Watson trees with offspring generating function of the form \eqref{eq:hyp1} form a \emph{universality class} in terms of their asymptotic sample coalescent structure.
\end{theorem}

We remark that, due to the time scaling in $T$, the branching rate parameter $r$ plays no role in Theorem \ref{thm:A}.
We will describe the complete structure of these universal partition processes $(\nu^{(k,\alpha)}_s)_{s \in [0,1]}$ (which can also be thought of as coalescent trees) in Theorem \ref{thm:B}.

Before discussing distributional properties of the stochastic process $\nu^{(k,\alpha)} := (\nu_s^{(k,\alpha)})_{s \in [0,1]}$ for general $\alpha$, we take a moment to elucidate further on the case $\alpha = 2$. We will see shortly (c.f.\ the case $\alpha = 2$ of Theorem \ref{thm:B}) that the stochastic process $\nu^{(k,2)}$ coincides with the one mentioned above 
from the work by Harris et al.\ \cite{HJR20}, in particular, that $\nu^{(k,2)}$ has binary splits, and the joint distribution of the $k-1$ splits is given by \eqref{eq:binarysplits}. 
As such, the case $\alpha=2$ of Theorem \ref{thm:A} implies the critical case in Harris et al. \cite{HJR20} (Theorem 3), namely that $\nu^{(k,2)}$ is a limiting object for the coalescent structure of critical trees with finite variance.

In fact, as touched on above, the class of offspring distributions whose moment generating functions of the form \eqref{eq:hyp1} 
with $\alpha=2$ is broader than critical distributions with finite variance. Indeed, there are cases when $\alpha=2$ but the variance is infinite; an explicit example is the offspring distribution with moment generation function
\begin{equation}\label{a2infinite}
f(s)=s+(1-s)^2\left(\frac{1}{2}+\frac{1}{4}\log\left(\frac{1}{1-s}\right)\right).
\end{equation}
See Slack \cite{Slack68} for further details.  As such, our special case $\alpha = 2$ of Theorems \ref{thm:A} 
\& \ref{thm:B} represents a slight generalisation of the critical finite variance case found in \cite{HJR20}.
Essentially, we are only able to extend to infinite variance offspring cases in the present paper by introducing discounting of the total population size into various spine changes of measure whilst, perhaps somewhat surprisingly, simultaneously being able to preserve key properties as well as understanding their more complex structures. 
This novel approach was not featured in \cite{HJR20}, where instead some $k^{th}$ moment assumptions were needed combined with truncation approximation argument. 
An analogous method of spine changes of measure with discounting has been concurrently developed by Harris, Palau, and Pardo \cite{HPP} for the (significantly different) setting of a discrete time critical Galton-Watson in varying environment with finite variances.

We now turn to describing the processes $\nu^{(k,\alpha)}$ for $\alpha \in (1,2)$, which are more complicated than their $\alpha = 2$ counterpart. In the $\alpha \in (1,2)$ setting, the blocks of the stochastic process $(\nu^{(k,\alpha)}_s)_{s \in [0,1]}$ may break into three or more sub-blocks at any splitting event, and consequently we need to take care to describe the topology of the process.

Let $\nu := (\nu_s)_{s \in [0,1]}$
be a stochastic process taking values in the set of partitions of $\{1,\ldots,k\}$ with the property that $\nu_0$ is one block, $\nu_{1}$ is singletons, and each discontinuity of $\nu$ is a block breaking into several sub-blocks. 
Write $0=:\tau_0 < \tau_1 < \ldots < \tau_m < 1$ for the splitting times (i.e.\ times of discontinuity) of $\nu$. The \emph{topology} $\mathcal{T}(\nu)$ of $\nu$ is the sequence of partitions
\begin{align*}
\mathcal{T}(\nu) := (\mathcal{T}_0,\ldots,\mathcal{T}_m) \qquad \textrm{with}\quad\mathcal{T}_i := \nu_{\tau_i}.
\end{align*}
The resulting sequence $(\mathcal{T}_0,\ldots,\mathcal{T}_m)$ is what we call a \emph{splitting sequence} of $\{1,\ldots,k\}$. A splitting sequence is a collection of partitions $(\beta_0,\ldots,\beta_m)$ of $\{1,\ldots,k\}$ such that $\beta_0$ is one block, $\beta_m$ is the singletons, and each $\beta_{i+1}$ is obtained from $\beta_i$ by breaking a single block of $\beta_i$ into two or more sub-blocks. 

Given a splitting sequence $(\beta_0,\ldots,\beta_m)$, we define the $i^{\text{th}}$ split size, $g_i$, to be the number of new blocks created at the $i^{\text{th}}$ split time, that is,
\begin{align*}
g_i := \# \beta_i - \# \beta_{i-1} +1  \qquad i =1,\ldots,m
\end{align*}
where $\# \beta_i$ is the number of blocks in the partition $\beta_i$. Since $\beta_0$ contains $1$ block, and $\beta_m$ contains $k$ blocks, $k-1$ blocks are created across the entire sequence and as such we have
\begin{align*}
\sum_{i =1}^m (g_i - 1) = k-1.
\end{align*}
With this notation at hand, we are now ready to state our second main result, characterising the law of the limit processes $(\nu^{(k,\alpha)}_s)_{s \in [0,1]}$ occuring in Theorem \ref{thm:A}.

\begin{theorem} \label{thm:B}
(a) For $\alpha \in (1,2)$, the probability law $\mathbb{P}^{(k,\alpha)}$ of $\nu^{(k,\alpha)}$ is given by the formula
\begin{align} \label{eq:vigil00}
\mathbb{P}^{(k,\alpha)}&\left( \mathcal{T}(\nu^{(k,\alpha)}) = \overline{\beta} , \tau_1 \in \mathrm{d}t_1,\ldots,\tau_m \in \mathrm{d}t_m\right) \nonumber \\
&\hspace{2cm}=  \frac{1}{ (\alpha-1) (k-1)!} .\,
\prod_{i=1}^m  \frac{\alpha \Gamma(g_i-\alpha)}{\Gamma(2-\alpha)}.\,
 \int_0^1 (1-w)^{k-1} w^{m + \frac{2-\alpha}{\alpha-1} } \prod_{i=1}^m  (1 - w t_i )^{-g_i} \mathrm{d}t_i\,  \mathrm{d}w,
\end{align}
where $m \geq 1$, $\overline{\beta} = (\beta_0,\ldots,\beta_m)$ is any splitting sequence for $\{1,\ldots,k\}$ with split sizes $g_1,\ldots,g_m$, and $0 < t_1 < \ldots < t_m < 1$ are splitting times. 

\noindent
(b) For $\alpha = 2$, here we have $\mathbb{P}^{(k,2)}( m = k-1, g_i = 2 ~\forall i ) = 1$, and we have
\begin{align} \label{eq:vigilbin}
\mathbb{P}^{(k,2)} &\left( \mathcal{T}(\nu^{(k,2)}) = \bar{P} , \tau_1 \in \mathrm{d}t_1,\ldots,\tau_{k-1} \in \mathrm{d}t_{k-1}\right) \nonumber \\
&\hspace{3cm}=  \frac{2^{k-1}}{ (k-1)!} \int_0^1 (1-w)^{k-1} w^{k-1} \prod_{i=1}^{k-1}  (1 - w t_i )^{-2} \mathrm{d}t_i\,   \mathrm{d}w.
\end{align}
where $\bar{P}$ is any splitting sequence of $\{1,\ldots,k\}$ consisting only of binary splits.
\end{theorem}

Let us make a few remarks. First we note that in the setting of Theorem \ref{thm:B}, after fixing the split sizes $g_1,\ldots,g_m$, the topology $\mathcal{T}(\nu^{(k,\alpha)})$ has no effect on the formula in \eqref{eq:vigil00}. As such, conditional on the event that $\nu^{(k,\alpha)}$ has $m$ splits at times $t_1,\ldots,t_m$ of sizes $g_1,\ldots,g_m$, the topology of the process is uniformly distributed on the set of possible splitting sequences with $m$ splits of sizes $g_1,\ldots,g_m$. 
This property may be regarded as a generalisation of the Kingman topology found in the case $\alpha =2$: 
going backwards in time, whenever a merger of $g$ ancestral lines is about to occur, it is equally likely to consist of any sub-collection of $g$ existing ancestral lines.

Next let us note that the $\alpha=2$ formula \eqref{eq:vigilbin} 
is consistent with the formula for $\alpha \in (1,2)$ in a limit when $\alpha \uparrow 2$. 
The fact that the formula is asymptotically supported on $\{g_1=2,\ldots,g_m=2\}$ when $\alpha = 2$ corresponds to the fact that ${\alpha \Gamma(g-\alpha)}/{\Gamma(2-\alpha)} \rightarrow \mathbf{1}_{\{g = 2\}}$ as $\alpha \uparrow 2$.
 
Additionally, we observe that after taking the change of variable $w = (1+\varphi)^{-2}$ and accounting for the $k!(k-1)!/2^{k-1}$ different binary trees with $k$ labelled leaves and $k-1$ ranked internal nodes, we can verify that \eqref{eq:vigilbin} is equivalent to the formula from \eqref{eq:binarysplits}. 

Recall that a splitting sequence $(\beta_0,\ldots,\beta_m)$ is a sequence of partitions of $\{1,\ldots,k\}$ such that $\beta_0$ is one block, $\beta_m$ is the singletons, and for each $i \geq 1$, $\beta_i$ is obtained from $\beta_{i-1}$ by breaking a single block of $\beta_{i-1}$ into $g_i \geq 2$ sub-blocks. 
Note that $\beta_i$ has $g_i - 1$ more blocks than $\beta_{i-1}$, so that 
\[\# \beta_j = 1 + \sum_{i=1}^j (g_i-1) =: k_j.\] 
In particular, $k_0 = 1$ and $k_m = k$. We now observe that
\begin{align} \label{eq:numbersplitting}
\# \{ \text{Splitting sequences of $\{1,\ldots,k\}$ with ordered split sizes $(g_1,\ldots,g_m)$} \} =k! \prod_{i=1}^m \frac{k_i}{g_i!}.
\end{align}
To see \eqref{eq:numbersplitting}, it is best to work backwards. Let $(g_1',\ldots,g_m')$ denote the splitting numbers listed in reverse order, and let $(k_0',\ldots,k_m')$ denote $(k_0,\ldots,k_m)$ listed in reverse order. Then going from $\beta_{m-i+1}$ to $\beta_{m-i}$, there are $\binom{k_i'}{g_i'}$ ways of choosing $g_i'$ blocks of the $k_i'$ blocks in $\beta_{m-i+1}$ and merging them to form a single block. Since every splitting process $(\beta_0,\ldots,\beta_m)$ arises in exactly this way, it follows that 
\begin{align} \label{eq:numbersplitting2}
\# \{ \text{Splitting sequences of $\{1,\ldots,k\}$} \} = \prod_{i=1}^m \binom{k_i'}{g_i'},
\end{align}
and then equation \eqref{eq:numbersplitting} follows by noting that $k_{i+1}' = k'_i - g'_i +1$ and rearranging.

Finally, we remark that equation \eqref{eq:numbersplitting} now supplies a way of describing the marginal distribution of the split times and split sizes, without any other regard for the topology of the process. Namely, by multiplying \eqref{eq:vigil00} by the number $k! \prod_{j=1}^m \frac{k_j}{g_j!}$ of splitting sequences of partitions with order split sizes $(g_1,\ldots,g_m)$, we see that the probability density of splits of ordered sizes $(g_1,\ldots,g_m)$ in times $\mathrm{d}t_1,\ldots,\mathrm{d}t_m$ is given by 
\begin{align*}
&\mathbb{P}^{(k,\alpha)}\left( \text{Splits of sizes $(g_1,\ldots,g_m)$ at times $\mathrm{d}t_1,\ldots,\mathrm{d}t_m$} \right) \nonumber \\
&\hspace{.5cm}= \frac{k }{ (\alpha-1) } .\,
\prod_{i=1}^m  \frac{\alpha  k_i\Gamma(g_i-\alpha)}{ g_i! \Gamma(2-\alpha) }  .\,
 \int_0^1 (1-w)^{k-1} w^{m + \frac{2-\alpha}{\alpha-1} } \prod_{i=1}^m  (1 - w t_i )^{-g_i}  \mathrm{d}t_i\,  \mathrm{d}w,
\end{align*}
for all $0 < t_1 < \ldots < t_m < 1$, where, as above, for $j = 1,\ldots,m$, $k_j :=1  + \sum_{i=1}^j (g_i-1)$ is the number of blocks after the $j^{\text{th}}$ split. 

\subsection{The Lauricella representation}
It turns out that it is possible to describe the distribution of the partition process $\nu^{(k,\alpha)}$ in terms of a certain special function known as the \emph{Lauricella} hypergeometric function \cite{CW}. The Lauricella function is given by 
\begin{equation} \label{eq:ocosto0}
F^{(n)}_D (c,a;b;t)=\int_{\mathbb{T}_n}\Big(1-\langle t,x\rangle\Big)^{-c} D_n(a; b;x)\ud x \qquad a \in \mathbb{R}_{>0}^n, b >0 , c >0,
\end{equation}
where $\langle t,x \rangle := \sum_{i=1}^n t_i x_i$, and
\[
D_n(a;b;x)=\frac{\Gamma(b+\sum_{i=1}^n a_i)}{\Gamma(b)\Gamma(\sum_{i=1}^n a_i)}. \left(1-\sum_{i=1}^n x_i\right)^{b-1}\prod_{i=1}^n x_{i}^{a_i-1}
\]
is the density of the Dirichlet distribution with parameters $(a,b) = (a_1,\ldots,a_n,b)$. In Section \ref{sec:lauricellaproof} we show that  equation \eqref{eq:vigil00} describing the law of $\nu^{(k,\alpha)}$ may alternatively be written
\begin{align} \label{eq:vigil_laur}
&\mathbb{P}^{(k,\alpha)}\left( \mathcal{T}(\nu) = \overline{\beta} , \tau_1 \in \mathrm{d}t_1,\ldots,\tau_m \in \mathrm{d}t_m\right) \nonumber \\
&=  \frac{1}{ (\alpha-1) }  \prod_{i=1}^m  \frac{\alpha \Gamma(g_i-\alpha)}{\Gamma(2-\alpha)}  \frac{\Gamma(m+\frac{1}{\alpha-1})}{\Gamma(k+m+\frac{1}{\alpha-1})}F^{(m)}_D\left[m+\frac{1}{\alpha-1}, g_1, \ldots, g_m; k+m+\frac{1}{\alpha-1}; t_1, \ldots, t_m\right].
\end{align}

\subsection{Population size and giant birth events}
In addition to describing the joint ancestral structure of $k$ uniformly sampled particles chosen at a large time $T$ from the population of a branching process in the \eqref{eq:hyp1} universality class, in the setting $\alpha \in (1,2)$ it transpires that there are also giant birth events occurring in conjunction with the split times. 

To explain this connection, let us begin by noting that in Section \ref{sec:gen}, using tools from Pakes \cite{Pakes10}, we undertake a careful analysis of the generating functions of branching trees whose offspring generating functions lies in the universality class \eqref{eq:hyp1}. Ultimately, we show that the survival probability takes the form
\begin{align*}
\overline{F}(T) := \mathbb{P}( Z_T > 0 ) \sim T^{-1/(\alpha-1)} \tilde{\ell}(T),
\end{align*}
where $\tilde{\ell}(\cdot)$ is another function slowly varying at infinity that may be described explicitly in terms of the slowly varying function $\ell(\cdot)$ occuring in \eqref{eq:hyp1}. See \eqref{eq:tchaik} below.

Consider now that the criticality $\mathbb{E}[L]=1$ of the offspring distribution entails that the expected number of particles alive at time $T$ satisfies $\mathbb{E}[Z_T]=1$ for all $T$. This implies that $\mathbb{E}[Z_T | Z_T > 0 ] = \overline{F}(T)^{-1}$, and as such we may expect that conditional on the event $\{Z_T > 0\}$, the number of particles alive has the order $\overline{F}(T)^{-1}$. Indeed, the following limit due to Pakes \cite{Pakes10} states that conditionally on $\{Z_T > 0 \}$, $\overline{F}(T) Z_T$ converges to a limit as $T$ gets large, with
\begin{align} \label{eq:painter}
\lim_{T \to \infty} \mathbb{E}[e^{ - \theta \overline{F}(T)Z_T } | Z_T > 0 ] = 1 - (1 + \theta^{1-\alpha})^{ - 1/(\alpha-1)}.
\end{align}

Recall the process $(\pi^{(k,T)}_{sT})_{s \in [0,1]}$ characterising the joint ancestral structure of a sample of $k$ particles from the population at time $T$, and that we denote by $\tau_1<\ldots<\tau_m$ the discontinuities, or splitting times, of this process. According to Theorem \ref{thm:C}, which we state shortly, when $T$ is large, the splitting times $\tau_1 < \ldots < \tau_m$ coincide with seismic birth events of size of the same order of magnitude as the entire population. To formulate this observation precisely, let
\begin{align*}
L_i := \text{Number of individuals born at time $\tau_i$,} \qquad i=1,\ldots,m.
\end{align*}
Let $(\pi^{k,T}_{s T} )_{s \in [0,1]}$ be the rescaled partition process characterising the joint ancestry of $k$ particles sampled from the population at time $T$. Our next result is an extension of Theorem \ref{thm:B}.

\begin{theorem} \label{thm:C}
Let $\alpha \in (1,2)$. Conditioned on the event $\{Z_T \geq k \}$, as $T \to \infty$,  we have the convergence in distribution
\begin{align*}
((\pi^{k,T}_{s T} )_{s \in [0,1]}, \overline{F}(T)L_1,\ldots,\overline{F}(T)L_m) \to ((\nu^{(k,\alpha)}_s)_{s \in [0,1]}, X_1,\ldots,X_m),
\end{align*}
of the time-rescaled partition process $(\pi^{k,T}_{sT})_{s \in [0,1]}$ together with the $\overline{F}(T)$-rescaled offspring sizes at the split times, where the limiting object $((\nu^{(k,\alpha)}_s)_{s \in [0,1]}, X_1,\ldots,X_m)$ is a partition process $(\nu^{(k,\alpha)}_s)_{s \in [0,1]}$ together with a random vector $(X_1,\ldots,X_m)$ of non-negative random variables, where $m$ is the (random) number of splitting events of $\nu^{k,\alpha}$.

Further, the joint law of $((\nu^{(k,\alpha)}_s)_{s \in [0,1]}, X_1,\ldots,X_m)$ is given by 
\begin{align} \label{eq:vigil01}
&\mathbb{P}^{(k, \alpha)}\left( \mathcal{T}(\nu) = \overline{\beta} , \tau_1 \in \mathrm{d}t_1,\ldots,\tau_m \in \mathrm{d}t_m, X_1 \in \mathrm{d}x_1,\ldots, X_m \in \mathrm{d}x_m \right) \nonumber \\
&=  \frac{1}{ (\alpha-1) (k-1)!}  \prod_{i=1}^m  \frac{\alpha \Gamma(g_i-\alpha)}{\Gamma(2-\alpha)}   \int_0^1 (1-w)^{k-1} w^{m + \frac{2-\alpha}{\alpha-1} } \prod_{i=1}^m  (1 - w t_i )^{-g_i} \mathrm{d}t_i ~  \Delta_{g_i,t_i}^w(x_i) \mathrm{d}x_i  ~  \mathrm{d}w.
\end{align}
where  $\overline{\beta} = (\beta_0,\ldots,\beta_m)$  is any splitting sequence for $\{1,\ldots,k\}$ with $m \geq 1$ splits having split sizes $g_1,\ldots,g_m$, $0 < t_1 < \ldots < t_m < 1$ are split times, $x_1,\dots,x_m\geq 0$ are split offspring sizes, and where 
\begin{align} \label{eq:Delta4}
\Delta^w_{g,t}(x) := \frac{ x^{g-\alpha-1} }{ \Gamma(g-\alpha) (1/w -t )^{\frac{g-\alpha}{\alpha-1}}}  \exp \left\{  - \frac{x}{(1/w-t)^{\frac{1}{\alpha-1}} } \right\}
\end{align}
is the probability density function of a standard Gamma random variable with shape parameter $g-\alpha$ and rate parameter $(w/(1-wt))^{\frac{1}{\alpha-1}}$. 
\end{theorem}

Let us take a moment to unpack Theorem \ref{thm:C}. First we note that \eqref{eq:vigil01} is an immediate generalisation of \eqref{eq:vigil00}; indeed, integrating through the variables $x_1,\ldots,x_m$ of the former we immediately obtain the latter. Let us comment further that the variable $w$ is like an (unobservable) mixture random variable that acts as a proxy for the entire population size, and after fixing the value of $w$, the topology and split times have a joint law proportional to $\prod_{i=1}^m \frac{ \Gamma(g_i-\alpha)}{\Gamma(2-\alpha)} (1 - w t_i )^{-g_i} $. After choosing $w$, the topology and the split sizes $g_1,\ldots,g_m$, the offspring events are conditionally gamma distributed with shape parameter $g-\alpha$ and scale factor $(1/w-t)^{\frac{1}{\alpha-1}}$ (equivalently, rate parameter 
$w^{\frac{1}{\alpha-1}}/(1-wt)^{\frac{1}{\alpha-1}}$).

Let us touch on the interpretation of Theorem \ref{thm:C} as $\alpha \uparrow 2$. We recall from our discussion following the statement of Theorem \ref{thm:B} that $\nu^{(k,2)}$ only has binary splitting events, i.e.\ when $\alpha=2$, each $g_i = 2$. In the case $g = 2$, we can interpret the probability measure $\Delta^w_{g,t}(x)\mathrm{d}x$ as approximating the Dirac mass at zero as $\alpha \uparrow 2$. This captures the vanishing of giant birth events as $\alpha \uparrow 2$.

Finally, let us note that using the Laplace transform of the gamma distribution, we have 
\[
\int_0^\infty e^{ - \gamma x} \Delta_{g,t}^w(x) \mathrm{d}x 
= \left(\frac{1}{1 - (1/w-t)^{\frac{1}{\alpha-1}}\gamma}\right)^{g-\alpha}
\qquad (\gamma \geq 0).
\] 
Consequently, for non-negative parameters $\gamma_1,\ldots,\gamma_m$ we may alternatively write 
\begin{align} \label{eq:vigil_laplace}
\mathbb{E}^{(k,\alpha)}&\left[ e^{ - \sum_{j =1}^m \gamma_m X_m };  \, \mathcal{T}(\nu) = \overline{\beta} , \tau_1 \in \mathrm{d}t_1,\ldots,\tau_m \in \mathrm{d}t_m \right] \nonumber \\
&=  \frac{1}{ (\alpha-1) (k-1)!}  \prod_{i=1}^m  \frac{\alpha \Gamma(g_i-\alpha)}{\Gamma(2-\alpha)}   \int_0^1 (1-w)^{k-1} w^{m + \frac{2-\alpha}{\alpha-1} } \prod_{i=1}^m  \frac{ (1 - w t_i )^{-g_i} }{\left( 1 - (1/w-t_i)^{\frac{1}{\alpha-1}}\gamma_i \right)^{g_i-\alpha} }  ~  \mathrm{d}w.
\end{align}

\subsection{Outline of proofs}
Let us outline briefly on our approach to proving Theorems \ref{thm:A}, \ref{thm:B} and \ref{thm:C}. We begin by introducing a collection of \emph{spines}, which are distinguished lines of descent that flow through a continuous-time Galton-Watson tree forward in time. We adapt and generalise the techniques in Harris et al.\ \cite{HJR20}, by constructing a change of measure $\mathbb{Q}_{\theta,T}^{(k)}$ that encourages the spines to flow through the tree in such a way that at time $T$ they represent a uniform sample of $k$ distinct particles of the tree in such a way that the overall population is size biased by the function $n \mapsto n(n-1)\ldots(n-k+1)e^{-\theta n}$. (Harris et al.\ \cite{HJR20} have $\theta=0$ in their set-up.) The parameter $\theta$, which is a discounting parameter controlling the size of the tree, furnishes a simple interpretation of sampling from a $k$ times size biased tree in the absence of a second (let alone $k^{\text{th}}$) moment. Let us mention here that this exponential discounting technique was used by first and third author in concurrent work with Sandra Palau \cite{HPP}.

A further innovation in our approach here is in accounting for the sizes of the offspring events at spine split times. This tool ultimately leads to the more descriptive limit in Theorem \ref{thm:C} of not just the ancestral process of the spines but also the magnitudes of the birth events.

After establishing basic properties of change of measure $\mathbb{Q}_{\theta,T}^{(k)}$, in Section 4 we undertake a careful analysis of the generating functions associated with processes in the universality class \eqref{eq:hyp1}, which allows us in Section 5 to ultimately tie our work together to prove our main results.

The behaviour of the spines under $\mathbb{Q}_{\theta,T}^{(k)}$ may be understood as a proxy for the ancestral behaviour of uniformly chosen particles under the original measure governing the Galton-Watson tree, though under $\mathbb{Q}_{\theta,T}^{(k)}$ the spines have a tractable distribution which is fairly easily described.

Moreover, under the measure $\mathbb{Q}_{\theta,T}^{(k)}$ we are able to describe in full detail the joint distribution of the entire population size in conjunction with the ancestral behaviour of the spines.

Finally, we  mention that Lager\r{a}s and Sagitov \cite{LS08} study the so-called \emph{reduced process} associated with continuous-time Galton-Watson processes in the heavy-tailed limit $\alpha=1$; the reduced process up until a time $t$ is the associated random process of a Galton-Watson process consisting of particles living at times $s \in [0,t]$ who have a descendent alive at time $t$. Lager\r{a}s and Sagitov mention that it is possible to show that the reduced process associated with a Galton-Watson process with heavy tails with   $\alpha \in (1,2]$  in our universality class \eqref{eq:hyp1} may be described in terms of a certain deterministic time change of a supercritical Galton-Watson process with a particular offspring distribution. (This offspring distribution also appears in Proposition 19 in Berestycki et al. \cite{BBS}  in the context of Beta coalescents.) In principle this paves a potential alternative avenue to our intermediate result in Theorem \ref{thm:B}: in order to obtain the coalescent structure of Galton-Watson trees in our university class \eqref{eq:hyp1} one may argue that the coalescent structure of a Galton-Watson tree is identical to that of its reduced tree, and thereafter study the coalescent structure of the particular supercritical tree, and ultimately undo the time change. We nonetheless favour our more direct approach, since it provides a rich probabilistic description of the coalescent structure under the size biased change of measure and, significantly, together with the understanding of  the uniform sampling at large times. The latter supplies a more intricate understanding of the relationship between the coalescent structure of the uniform sample and the population size, and in particular gives us the full description of this relationship afforded by Theorem \ref{thm:C}.

\subsection{Further discussion of related work} \label{sec:literature} 
As mentioned above, countless authors have studied various special cases of the partition process $(\pi^{(k,T)}_t)_{t \in [0,T]}$ associated with the ancestry of $k$ uniformly sampled particles from a continuous-time Galton-Watson tree. The case $k =2$, which amounts to sampling two particles from the population and studying their time to most recent common ancestor, is particularly well trodden \cite{Ath12a, Ath12b, Ath16, buhler:super, Dur78, OConnell95}. Work for $k \geq 3$ has appeared only more recently. In the setting of general continuous-time Galton-Watson trees, the second author \cite{J19} found an explicit integral formula for the finite dimensional distributions of $\pi^{(k,T)}$ in terms of the generation functions $F_t(s) := \mathbb{E}[s^{Z_t}]$ of the process, and thereafter considered the large-$T$ asymptotics in the supercritical and subcritical cases; see also Grosjean and Huillet \cite{grosjean_huillet:coalescence} and Le \cite{Le14}. 
Zukhov \cite{Zubkov76} studied a certain aspect of the $k=\infty$ case for critical Galton-Watson trees, showing that asymptotically the most recent common ancestor of the entire population at time $T$ is uniformly distributed on $[0,T]$. 

Several related models have also been considered. The first and last author have worked recently,   on the problem of sampling $k$ particles from a critical branching process in a varying environment (see \cite{HPP}). A similar result was obtained independently by Boenkost et al. \cite{BFRS}. The second author has worked recently with David Cheek \cite{CJ} on the problem of sampling a single particle uniformly from the population of a continuous-time Galton-Watson tree and studying the point process of reproduction times along the particles ancestral lineage, and with Amaury Lambert \cite{JL20} on the ancestry of continuous-state branching processes. The three authors of this manuscript have also considered the problem of sampling $k$ particles from a critical branching process with infinite mean (see \cite{HJP23}). Let us also mention work by Amaury Lambert and coauthors \cite{L1,L2} on coalescent point processes associated with trees, as well Aldous and Popovic \cite{AP} and Gernhard \cite{gernhard}.

\subsection{Overview}
We now give the structure of the remainder of the article.
In Section \ref{sec:spines}, we introduce multiple spines and our changes of measure that will underpin our approach.
 In Section \ref{sec:Q}, we prove properties of the multiple spines under the change of measure $\mathbb{Q}_{\theta,T}^{(k)}$.
Section \ref{sec:gen} is dedicated to studying basic properties of trees in the universality class \eqref{eq:hyp1}, and the large-$T$ asymptotic behaviour of spines under $\mathbb{Q}_{\theta,T}^{(k)}$ for trees in this universality class.
In the final section, Section \ref{sec:inversion}, we study the joint law of the population size and the spines under $\mathbb{Q}_{\theta,T}^{(k)}$, thereby ultimately inverting the change of measure.

\section{Spines and changes of measures} \label{sec:spines}

In this section we introduce various probability measures that will serve as essential tools for our understanding of the genealogies of samples of $k$ individuals drawn from the population without replacement at time $T$. 
We start by giving a more precise description of the continuous-time Galton-Watson (GW) population model. 
Then we will extend the original GW model by identifying $k$ distinguished lines of descent, known as \emph{spines}. 
Via a change of measure, we will then define our key size-biased and discounted GW process with $k$-spines under a probability measure $\mathbb{Q}_{\theta,T}^{(k)}$, describing its properties and how it modifies the behaviour of both the population process and the spines.
\subsection{Notation}
Throughout, we use the convention that $\mathbb{Z}^+:=\{0,1,2,3,\dots\}$ and $\mathbb{N}:=\{1,2,3,\dots\}$.
We will make use of the standard Ulam-Harris labelling system to keep track of genealogical information of particles: 
when an individual labelled $u$ dies and is replaced by $L$ offspring, these are labelled by concatenating the parent label with the number of each child, yielding offspring labels $u1, u2, \dots, uL$, and so on. 
If there is only one initial particle, for convenience, the single root is usually labelled $\emptyset$. 
Then, for example, $2.3.1$ would represent the first child of the third child of the second child of the initial ancestor.
With this notation, it is easy to refer to subtrees within, or join subtrees onto, existing trees.

Denote by $\mathcal{Z}_t$ the set of labels of all particles alive at time $t>0$, and let $Z_t=|\mathcal{Z}_t|$ be the number of particles alive at time $t>0$.
If $u$ is an ancestor of $v$ or equal to $v$ then we write $u\preceq v$, and if $u$ is a strict ancestor of $v$ then we write $u\prec v$. 
An initial labelling $\mathcal{Z}_0$ is said to be \emph{permissible} as long as no initial individual is an ancestor of, or the same as, any other one, that is, $u\not\preceq v$ for any distinct pair $u,v\in\mathcal{Z}_0$.

Throughout, we will use a common sample space and $\sigma$-algebra $(\Omega, \mathcal{F})$ which is  sufficiently enriched to describe the randomness associated with all the various processes encountered below. 
In the following subsections, we will introduce various probability measures, $\P,\mathbb{P}_{\theta, T}, \mathbb{P}^{(k)}, \mathbb{Q}^{(k)}_{\theta, T}$, and  $\mathbb{P}^{(k)}_{{\rm unif},T }$, which will all defined on this same common space $(\Omega, \mathcal{F})$. Considering the same process under different probability laws will be foundational to our approach.

We define $(\mathcal{F}_t, t\ge 0)$ to be the natural filtration of the population process $\mathcal{Z}:=(\mathcal{Z}_s)_{s\geq0}$, that is, $\mathcal{F}_t:=\sigma(\mathcal{Z}_s : s\leq t)$. (We will introduce other filtrations in the sequel;  see Subsection \ref{filts}.)

\subsection{The continuous-time Galton-Watson process under $\P$}
Let $(\Omega, \mathcal{F}, \P)$ be a probability space.
Suppose $\overline{p}=(p_i)_{i\in \mathbb{Z}^+}$ is a probability distribution on $\Z^+$ 
with finite mean $m:=\sum_{i\in\Z^+} i p_i<\infty$.

\begin{defi}[Galton-Watson process under $\mathbb{P}$]
Under the probability measure $\mathbb{P}$, we say that $\mathcal{Z}=(\mathcal{Z}_s)_{s\geq0}$ is a continuous-time Galton-Watson 
process with branching rate $r$ and offspring distribution $\overline{p}$,
started with one individual, if:
\begin{enumerate}
\item The process initially starts with one particle alive (i.e.\ $Z_0=1$ with label $\mathcal{Z}_0=\{\emptyset\}$).
\item The branching property: Any particles currently alive evolve forward in time independently of one another and of the history of the process.
\item Any particle currently alive at time $t$ undergoes branching at rate $r$ (i.e.\ the time until the particle branches is exponentially distributed at rate $r$).
\item Given that particle $v$ branches at time $t$, it immediately dies and is simultaneously replaced by $L_v$ offspring, where $L_v$ is an independent realisation of the offspring random variable $L$ with $\mathbb{P}(L=i)= p_i$, for $i\in\Z^+$.
\end{enumerate}
\end{defi}


Similarly, a Galton-Watson process initially started  with $j\geq2$ particles alive is defined analogously simply by modifying (1), typically with the labelling convention that $\mathcal{Z}_0=\{1,2,\dots,j\}$, although any permissible labelling $\mathcal{Z}_0$ can be specified.

\subsection{The Galton-Watson discounted by population size under $\mathbb{P}_{\theta, T}$}

Consider any fixed $T>0$ and $\theta \geq 0$ and define a new probability measure $\mathbb{P}_{\theta, T}$ 
on $\mathcal{F}_T$ via the Radon-Nikodym derivative
\begin{equation}\label{PthetaT}
\frac{{\ud  \mathbb{P}_{\theta, T}}}{\ud \P}\Bigg|_{\mathcal{F}_T}
:=
\frac{ e^{-\theta Z_T}}{\mathbb{E}\Big[e^{-\theta Z_T}\Big]}.
\end{equation}
In other words, the law of the process $\mathcal{Z}$ under $\mathbb{P}_{\theta, T}$ is like that of the original law $\mathbb{P}$ except with exponential discounting at rate $\theta$ according to the size of population at time $T$. 
In fact, the process under $\mathbb{P}_{\theta, T}$ remains a branching process although the behaviour becomes \emph{time-dependent}, as below (Lemma \ref{PthetaConstruction}).

Introducing discounting by the final population size turns out to be quite natural and indeed will be crucial in allowing our methods to work without any additional moment assumptions on the offspring distribution $L$, in particular, to encompass the heavy-tailed laws of interest.

The following lemma describes the behaviour of the particles under $\mathbb{P}_{\theta,T}$.
\begin{lemma}[Galton-Watson discounted by population size under $\mathbb{P}_{\theta, T}$]\label{PthetaConstruction}
Under the probability law $\mathbb{P}_{\theta, T}$, the process $\mathcal{Z}=(\mathcal{Z}_s)_{s\geq0}$ evolves over time period $[0,T]$ as a time inhomogeneous Markov branching process, where:
\begin{enumerate}
\item The process initially starts  with a single particle alive.
\item Any particles currently alive evolve forward in time independently of one another and of the history of the process \emph{(branching property)}.
\item Any particle currently alive at time $t$ undergoes branching at rate
\[
r \E\left[\left(\E[e^{-\theta Z_{T-t}}]\right)^{L-1}\right].
\]
\item Given that particle $v$ branches at time $t$, it immediately dies and is simultaneously replaced by $L_v$ offspring,
where $L_v$ is an independent realisation of the offspring random variable at time $t$, $L(t)$, where
\[
\mathbb{P}_{\theta,T}(L(t)=\ell)= p_\ell \,\frac{\left(\E[e^{-\theta Z_{T-t}}]\right)^\ell}{\E\left[\left(\E[e^{-\theta Z_{T-t}}]\right)^L\right]}.
\]
\end{enumerate}
\end{lemma}
Lemma \ref{PthetaConstruction} follows from the case $k=0$ of Lemma \ref{lem:nonspine}, which is proven in Section \ref{sec:Q}.

\subsection{The Galton-Watson process with $k$-spines under $\mathbb{P}^{(k)}$}
For any fixed $k\in\N$, we now proceed to define a measure $\mathbb{P}^{(k)}$ under which the population process $\mathcal{Z}$  has $k$ distinguished lines of descent, known as \emph{spines}.
The measure $\mathbb{P}^{(k)}$ will serve as a natural and convenient \emph{reference measure} when looking at the behaviour of other population processes with $k$ distinguished particles (eg. later we will select a sample of $k$ individuals uniformly at random  at time $T$). 

We denote the $k$ spines by $\xi=(\xi^{(1)}, \xi^{(2)}, \dots, \xi^{(k)})$, where $\xi^{(i)}$ corresponds to the distinguished line of decent of the $i^{th}$ \emph{spine}. 
Each spine is represented by a sequence Ulam-Harris labels $v_0 v_1 v_2 \dots$ which start at the initial ancestor, and where the next label in the sequence is always an offspring of the previous (i.e.\ $v_0=\emptyset$ and, for each $i\in\N$, $v_{i+1}=v_i \ell$ for some $\ell\in\{1,\dots,L_{v_i}\}$). 
A spine may be an infinite line of descent, or a finite path which terminates at a leaf in the underlying genealogical tree of the population. 
If a particle $u$ has $j$ distinct spines passing though it (i.e. $\#\{i\in\{1,\dots,k\}: u\in\xi^{(i)}\}=j$), then we say particle $u$  is \emph{carrying $j$ spines}.  

The process $\mathcal{Z}$ with $k$ spines $\xi$ under measure $\mathbb{P}^{(k)}$ is constructed as a simple extension of $\mathcal{Z}$ under $\mathbb{P}$, in that, all particles behave exactly as in the original branching process but some particles are additionally identified as carrying the spines, as follows.
Firstly, the population $\mathcal{Z}$ is constructed according to $\mathbb{P}$. Then, given a realisation of the population $\mathcal{Z}$, the $k$ spine lineages are chosen independently, each spine starts by following the initial ancestor, and then at each subsequent branching event following one of the offspring chosen uniformly at random (or dying if no offspring). 
More precisely, we have the following construction:

\begin{defi}[Galton-Watson process with $k$ spines under  $\P^{(k)}$]\label{Pkdescription}
Under probability measure $\mathbb{P}^{(k)}$, the process $\mathcal{Z}=(\mathcal{Z}_s)_{s\geq0}$ with $k$ spines $\xi=(\xi^{(1)}, \xi^{(2)}, \dots, \xi^{(k)})$ is defined as follows:
\begin{enumerate}
\item The process initially starts with one particle alive (i.e.\ $\mathcal{Z}_0=\{\emptyset\}$, with $Z_0=1$ particle labelled as the root $\emptyset$) which is carrying the $k$ spines. 
\item The branching property: Any particle alive --- either carrying or not carrying spines --- evolves forwards in time independently of the other particles in the system and of the history of the process.
\item Any particle currently alive at time $t$ undergoes branching at rate $r$.
\item Given that a particle $v$ branches at time $t$, it immediately dies and is simultaneously replaced by $L_v$ offspring, where $L_v$ is an independent realisation of the offspring random variable $L$ with $\mathbb{P}(L=i)= p_i$ for $i\in\Z^+$. 
\item Conditional on particle $v$ carrying $j$ spines at the time it branches and on having $L_v=\ell$ offspring:
\begin{enumerate}
\item if $\ell\geq1$, each of the $j$ spines chooses independently and uniformly at random which of the $L_v$ offspring to continue to follow. 
\item $\ell=0$, there are no offspring for the spines to continue to follow and those $j$ spine paths terminate
and pass into the cemetery state 
$\overline{\delta}$.
\end{enumerate}
\end{enumerate}
\end{defi}

\begin{rem}
We note that a Galton-Watson processes with $k$-spines under  $\P^{(k)}$ can also be thought of as a multi-type Galton-Watson process where the type of each particle in $\{0,1,\dots,k\}$  corresponds to the number of spines passing along it. Here, a type $j$ individual can only have offspring types of the same or lower value and where the sum of offspring types must match the parent's type (the number of spines are preserved at branching events). 
As a richer alternatively to keep track of individual spine trajectories, the type of each particle could also include of the labels of any spines passing along it
(for example, taking the type space as all possible subsets of $\{1,\dots,k\}$),  
although again only the number of spines along each particle would affect its branching.
\end{rem}

We write $\xi_t=(\xi_t^{(1)}, \ldots, \xi_t^{(k)})$ to identify the $k$ spines at time $t$, where $\xi_t^{(i)}$ is the label of the particle carrying spine $i$ at time $t$. Since each spine chooses its path independently and uniformly from amongst the available offspring, we can immediately  observe that, for any $u_1,\dots,u_k\in\mathcal{Z}_t$ and $\underline{u}=(u_1,\dots,u_k)$ , 
\begin{equation}\label{mea-k-spine}
\mathbb{P}^{(k)}\Big( \xi_t=\underline{u}\Big|\mathcal{F}_t\Big)=\prod_{i=1}^k\prod_{v\prec u_i}\frac{1}{L_v}.
\end{equation} 
Also observe that spines paths can end at a leaf under $\P^{(k)}$.  
For more details, see proof of Lemma 6 in Harris et al. \cite{HJR20}.

\subsubsection{Some further terminology and notation} 
When we are interested in which particular spines pass though a given particle, it is sometimes convenient to think of a particle carrying marks.
The $k$ spines are marked (i.e.\ labelled) as $1,\dots,k$, so we can identify which spines a particle is carrying by their marks. 
Since all the spines start at the initial ancestor, particle $\emptyset$ carries all $k$ marks, $1,2,3,\dots,k$. 
A particle $v$ through which $j$ spines pass will carry $j$ marks for some $b_1<b_2<\dots<b_j$, where each $b_i\in\{1,\dots,k\}$ uniquely identifes a spines.

Further, let $n_t$ be the number of distinct spines (i.e.\ the number of distinct particles in $\mathcal{Z}_t$ carrying marks) at time $t$. For $i\ge 1$, let $\tau_i$ be the $i$-th spine split time
\[
\tau_0=0, \qquad \tau_i=\inf\{t\ge 0: n_t\notin\{1,2,\ldots, i\}\}
\] 
where $\tau_i$ is the first time we have more than $i$ spine groups, equivalently, the $i$-th spine split time when accounting for multiplicities (i.e.\ if a group of spines splits into $i$ separate groups at a branch time $t$, then there would be $i-1$ new split times at $t$). We also let $\rho_t^{(i)}$ be the total number of spines accompanying spine $i$ at time $t$, including $i$ itself ($\rho_t^{(i)}=0$ if $\xi^{(i)}=\overline\delta$ corresponding to spine $i$ already in the cemetery state).

The set of distinct spine particles at any time $t$, and the marks that are following those spine particles, induces a partition $\pi_t^{(k)}$ of $\{1,\ldots, k\}$ as follows: we declare $i$ and $j$ to be in the same block of $\pi_t^{(k)}$ if $\xi_t^{(i)}=\xi^{(j)}_t$. If we then let $\mathcal{T}_i=\pi^{(k)}_{\tau_i}$ for $i=0,\ldots, m$, where $m$ is the number of split times of the  partition process, then we have created a {splitting sequence} of partitions $\mathcal{T}_0, \mathcal{T}_1, \ldots, \mathcal{T}_m$ which describe the topological information about the spines without the information about the spine split times. 
 
For any particle $u\in \mathcal{Z}_t$, there exists a last time at which $u$ was a spine (which may be $t$). If this time equals $\tau_i$ for some $i$, then we say that $u$ is a {\it residue} particle; if it does not equal $\tau_i$ for any $i$, and $u$ is not a spine, then we say that $u$ is \emph{ordinary}. Each particle is exactly one of residue, ordinary, or a spine (i.e.\ carrying one or more marks).

\subsubsection{Filtrations}\label{filts} Whilst all our processes and probability measures are assumed to be carried on a sufficiently rich common space $(\Omega,\mathcal{F})$, it will be very convenient for us to make use of a number of different sub-$\sigma$-algebras and filtrations according to how much information we want to know about the particles and spines.
To this end, we let:
\begin{itemize}
\item $(\mathcal{F}^{(k)}_t)_{t\ge 0}$ be the filtration containing all information about the process, including the $k$ spines, up to time $t$.
\item $(\mathcal{F}_t)_{t\ge 0}$ be the filtration containing only the information about the Galton-Watson tree, but nothing about the identity of any spines. 
\item $(\widetilde{\mathcal{G}}_t^{(k)})_{t\ge 0}$ be the filtration  containing all the information about the $k$ spines up to time $t$, including the birth events and numbers of offspring along the $k$ spines, but no information about the rest of the tree. 
\item $(\mathcal{G}_t^{(k)})_{t\ge 0}$ be the filtration containing information only about the spine splitting events (including which marks follow which spine); $(\mathcal{G}_t^{(k)})_{t\ge 0}$ does not know when births of ordinary particles along the spines occur (i.e.\ any births coming off the spines when the spines all stay together). 
\end{itemize}
Note, $\mathbb{P}^{(k)}$ can be defined on $\mathcal{F}^{(k)}_\infty\subseteq\mathcal{F}$, and this is the smallest $\sigma$-algebra we might use. For further details, see Harris and Roberts \cite{HR17}

\subsection{The $k$-spine measure $\mathbb{Q}^{(k)}_{\theta, T}$.}
We will now introduce the key probability measure under which the spines will form a uniform choice (without replacement) from the population alive at time $T$, as required, but where there will also be $k$-size biasing and discounting by the population size at time $T$. In particular, the population process with spines under this measure will turn out to have sufficient structural independence properties to greatly facilitate computations, and where the discounting allow us to develop $k$-spine methods without requiring any additional moment conditions. 

We will often make use of the notation $n^{(k)}$ to represent the number of ways of choosing $k$ distinct objects from $n$ objects, more precisely, for any integers $n$ and $k$,
\[
n^{{(k)}}:=\left\{\begin{array}{ll}
n(n-1)\cdots(n-k+1), & \textrm{if } n\geq k\geq1,\\
1 & \textrm{if } n\geq1, k=0, \\
0 & \textrm{otherwise}.
\end{array}
\right.
\]

Let us fix $T>0$ and $\theta \geq 0$ and introduce a new probability measure $\mathbb{Q}^{(k)}_{\theta, T}$ on $\mathcal{F}_T^{(k)}$ by setting
\begin{equation}\label{newprob}
\frac{{\ud  \mathbb{Q}^{(k)}_{\theta, T}}}{\ud \P^{(k)}}\Bigg|_{\mathcal{F}^{(k)}_T}:=\frac{\mathbf{1}_{A_{k, T}}\left(\prod_{i=1}^k\prod_{v\prec \xi^{(i)}_T} L_v\right) e^{-\theta Z_T}}{\mathbb{E}^{(k)}\Big[Z_T^{(k)} e^{-\theta Z_T}\Big]}, 
\end{equation}
where $A_{k, T}:=\{\xi_T^{(i)}\not=\xi_T^{(j)}, \forall i\not=j, i,j=1,\dots,k \}$ denotes the event that the $k$-spines are alive and distinct at time $T$.

We now comment on some key properties of $\mathbb{Q}^{(k)}_{\theta, T}$. 
First, we note that $\mathbb{Q}^{(0)}_{\theta,T} = \mathbb{P}_{\theta,T}$.

Let $\mathcal{Z}^{(k)}_T$ denote the collection of distinct $k$-tuples of individuals alive at time $T$, so that $\# \mathcal{Z}_T^{(k)} = Z_T^{(k)}$. Using the selection of the spines under $\P^{(k)}$ given the underlying population process (Definition \ref{Pkdescription}) and then using \eqref{mea-k-spine}, 
we find that
\begin{equation}\label{expnprob}
\begin{split}
\mathbb{E}^{(k)}\left[\mathbf{1}_{A_{k, T}}e^{-\theta Z_T}\prod_{i=1}^{k}\prod_{v\prec \xi^{(i)}_T} L_v\Bigg |\mathcal{F}_T\right]&=\sum_{u_1, \ldots, u_k\in \mathcal{Z}^{(k)}_T}\P^{(k)}\left( \xi_T=\underline{u}\Big| \mathcal{F}_T\right)\,e^{-\theta Z_T}\prod_{i=1}^{k}\prod_{v\prec u_i} L_v  =Z_T^{(k)}e^{-\theta Z_T}.
\end{split}
\end{equation}

This confirms thats $\mathbb{Q}^{(k)}_{\theta, T}$  is a probability measure, and also that 
\begin{equation}\label{newprob1}
\frac{{\ud  \mathbb{Q}^{(k)}_{\theta, T}}}{\ud \P^{(k)}}\Bigg|_{\mathcal{F}_T}=\frac{Z_T^{{(k)}} e^{-\theta Z_T}}{\mathbb{E}^{(k)}\Big[Z_T^{{(k)}}e^{-\theta Z_T}\Big]}.
\end{equation}
In particular, the distribution of the random variable $Z_T$ under $\mathbb{Q}^{(k)}_{\theta, T}$ is that of $\P^{(k)}$ but $k$-size biased and $\theta$-discounted by the function $n \mapsto n^{{(k)}}e^{-\theta n }$.

Importantly, for any $u_1, \ldots, u_k \in Z_T$, we have
\[
\mathbb{Q}^{(k)}_{\theta, T}\left(\xi_T=\underline{u}\Big| \mathcal{F}_T\right)\propto \P^{(k)}\left(\xi_T=\underline{u}\Big| \mathcal{F}_T\right)\mathbf{1}_{A_{k, T}}\left(\prod_{i=1}^{k}\prod_{v\prec u_i} L_v \right) e^{-\theta Z_T}, 
\]
then from \eqref{mea-k-spine}, we deduce the crucial equation
\begin{equation}\label{qk-form}
\mathbb{Q}^{(k)}_{\theta, T}\left(\xi_T=\underline{u}\Big| \mathcal{F}_T\right)=\frac{\mathbf{1}_{A_{k, T}}}{Z_T^{{(k)}}},
\end{equation}
which states the property that, under the measure $\mathbb{Q}^{(k)}_{\theta, T}$, the $k$-spines are a uniform choice without replacement from those particles alive a time $T$.

From the previous observation and its definition in \eqref{newprob}, we can think of $\mathbb{Q}^{(k)}_{\theta, T}$ by first (i) $k$-size biasing and discounting by the population size $Z_T$ given $\mathcal{F}_T$ and then (ii) choosing $k$-spines uniformly without replacement under $\mathbb{Q}^{(k)}_{\theta, T}$ given $\mathcal{F}_T$, i.e
\begin{equation}
\frac{\ud \mathbb{Q}^{(k)}_{\theta, T}}{\ud \mathbb{P}^{(k)}}\Bigg|_{\mathcal{F}^{(k)}_T}
=\frac{Z_T^{{(k)}} e^{-\theta Z_T}}{\mathbb{E}^{(k)}\Big[Z_T^{{(k)}} e^{-\theta Z_T}\Big]}\cdot
\frac{\mathbf{1}_{A_{k, T}}}{Z_T^{{(k)}}}\cdot
\prod_{i=1}^k\prod_{v\prec \xi^{(i)}_T} L_v.
\label{ksizeunif}
\end{equation}
Also observe that $\mathbb{Q}^{(k)}_{\theta, T}(Z_T\ge k)=1$.

Our aim is to provide a complete description of the evolution of the Galton-Watson process with $k$-spines under $\mathbb{Q}^{(k)}_{\theta, T}$ - its desirable properties (including uniformly sampled spines) and tractability due to the independence within its structure (due to the size-biasing) will prove absolutely crucial to our later analysis of uniform sampling.

We already remarked that a Galton-Watson process with $k$-spines under $\mathbb{P}^{(k)}_{\theta, T}$ can be thought of as a time inhomogeneous multi-type Galton-Watson process, where the type of a particle is the number of spines it carries. 
Since the branching structure is preserved by the ``product" structure of the Radon-Nikodyn derivative in \eqref{newprob}, it  turns out that the Galton-Watson process with $k$-spines under $\mathbb{Q}^{(k)}_{\theta, T}$ can be seen as another time inhomogeneous multi-type Galton-Watson process (cf. Abraham and Debs \cite{AD20} for a discrete GW with k-spines described this way).
We have the following property:

\begin{lemma}[Branching Markov property/Symmetry Lemma] \label{branlem}Suppose that $u\in \mathcal{Z}_t$  is carrying $j$ marks at time $t$, i.e. that $u$ has $j$ spines passing through it. Then,  under $\mathbb{Q}^{(k)}_{\theta, T}$, the subtree generated by $u$ after time $t$ is independent of the rest of the process and behaves as if under $\mathbb{Q}^{(j)}_{\theta, T-t}$.
\end{lemma}

The proof of this Lemma follows from the same arguments of  Lemma 8 in Harris et al. \cite{HJR20}, where the measure $\mathbb{Q}^{(k)}_{\theta, T}$ is considered without the compensation term $\theta$. Thus, it just remains  to identify the branching rates for $k$-spines to describe $\mathbb{Q}^{(k)}_{\theta, T}$. 

The following result gives a full account of the behaviour of the spine and non-spine particles under the change of measure $\mathbb{Q}^{(k)}_{\theta,T}$.

\begin{lemma}[Size-biased and discounted Galton-Watson process with $k$ spines under  $\mathbb{Q}^{(k)}_{\theta, T}$]
\label{Qconstr}
The process $\mathcal{Z}=(\mathcal{Z}_s)_{s\in[0,T]}$ with $k$ spines $\xi=(\xi^{(1)}, \xi^{(2)}, \dots, \xi^{(k)})$  under measure $\mathbb{Q}^{(k)}_{\theta, T}$ evolves as follows:
\begin{enumerate}
\item The process starts at time $0$ with one particle carrying all $k$ spines (i.e.\ $\mathcal{Z}_0=\{\emptyset\}$ and $\xi_0=(1,2,\dots,k)$).
\item A particle carrying $j$ spines evolves a sub-tree forward in time independently of the rest of the process (branching Markov property).
\item 
A particle carrying $j\geq1$ spines at time $t$ branches into $\ell$ offspring and the $j$ spines split into $g\in\{1,\dots,j\}$ groups of sizes $k_1,\dots,k_g\geq 1$ with $\sum_{i=1}^g k_i=j$
at rate
\[ 
\frac{j!}{\prod_{m=1}^j h_m!\prod_{i=1}^g k_i!}\cdot
\frac{\prod_{i=1}^g \mathbb{E}[Z^{(k_i)}_{T-t}e^{-\theta Z_{T-t}}]}{\mathbb{E}[Z^{{(j)}}_{T-t}e^{-\theta Z_{T-t}}]} 
r\mathbb{E}\Big[L^{(g)}\left(\mathbb{E}[e^{-\theta Z_{T-t}}]\right)^{L-g}\Big]\cdot
\frac{\ell^{(g)}\left(\mathbb{E}[e^{-\theta Z_{T-t}}]\right)^{\ell-g}\,p_\ell}{\mathbb{E}\Big[L^{(g)}\left(\mathbb{E}[e^{-\theta Z_{T-t}}]\right)^{L-g}\Big]},
\]
where $h_m:=|\{i: k_i=m\}|$ is the number of spine groups of size $m$, so that $\sum_{m=1}^j m h_m =j$ and $\sum_{m = 1}^j h_m = g$.
\item
Given a particle carrying  $j\geq1$ spines branches into $\ell$ offspring 
where the $j$ spines split into $g$ groups of sizes $k_1,\dots,k_g\geq 1$, 
the spines are assigned between the offspring as follows:
\begin{enumerate}
\item choose $g$ of the $\ell$ offspring to carry the spine groups 
uniformly amongst the $\ell!/(g!(\ell-g)!$ distinct ways. 
\item assign the group sizes $k_1,\dots, k_g$ amongst the $g$ offspring chosen to carry the spine groups 
uniformly amongst the ${g!}/{\prod_{m=1}^j h_m!}$ distinct allocations.
\item partition the $j$ (labelled) spines between the $g$ chosen offspring with their given group sizes from $k_1,\dots,k_g$  
uniformly amongst the $j!/{\prod_{m=1}^g k_m!}$ distinct ways.
\end{enumerate}
\item 
Finally, any particle $v$ which is alive at time $t$ and carries no spines behaves independently of the remainder of the process and undergoes branching 
into $\ell$ offspring that carry no spines at rate
\[
r \E\left[\left(\E[e^{-\theta Z_{T-t}}]\right)^{L-1}\right]\cdot\frac{\left( \E[e^{-\theta Z_{T-t}}]\right)^\ell \, p_\ell }{\E\left[\left(\E[e^{-\theta Z_{T-t}}]\right)^L\right]}.
\]
\end{enumerate}
\end{lemma}

For the proof of Lemma \ref{Qconstr}, Part (1) follows from the definition, Part (2) follows from the Markov branching property of Lemma \ref{branlem}, Part (3) and Part (4) follow from Lemma \ref{Splitsgspines}, and, finally, Part (5) will be given by Lemma \ref{lem:nonspine}.

The number of spines following each offspring is of particular interest as whenever spines split apart going forward in time this corresponds to coalescence of family lines when viewing backwards in time starting from the individuals sampled at the end.
There are various other ways of describing the process with spines, each of which can be extracted from the above description that involves the number of spines particles carry.  
If instead we are interested in the partitions formed by individual spines and how they break up, 
or want thinking about the the process as a multi-type Galton-Watson process, 
the spine splitting rates above can be readily adjusted by the required combinatorial factors.

Observe that the rate a particle carrying $j$ spines at time $t$ branches and all the spines stay together (so $g=1$)
is given by
$$
r\mathbb{E}\Big[L \left(\mathbb{E}[e^{-\theta Z_{T-t}}]\right)^{L-1}\Big]
$$
and the offspring distribution at a branching event at time $t$ when all spine stay together is size-biased and discounted by time to go 
with the probability of getting $\ell$ offspring being
$$
\frac{\ell \left(\mathbb{E}[e^{-\theta Z_{T-t}}]\right)^{\ell-1}\,p_\ell}{\mathbb{E}\Big[L\left(\mathbb{E}[e^{-\theta Z_{T-t}}]\right)^{L-1}\Big]}.
$$
Such branching events where all the spines follow the same particle ($g=1$) are sometimes referred to as \emph{births off the spine}, as opposed to \emph{spine splitting} branching events where the spine break apart into 2 or more groups ($g\geq 2$). 
Importantly here, \emph{the births off {any} spine branch occur independently of the number of spines following that branch}
(i.e.\ independent of $j$ whenever $g=1$, as above).
This very convenient feature means that the sub-populations 
that have come off any particular spine branch only depend on time period over which that branch has run, 
but not on the number of spines along that branch or whether it changes along it.
As the total population is made up of the spine plus the subpopulations coming off each spine branch, this observation will make understanding the total population size relatively straightforward under $\mathbb{Q}^{(k)}_{\theta, T}$.

\subsection{Uniform sampling from a Galton-Watson process} 
As was noted in Harris et al.\ \cite{HJR20}, the relatively easy to work with measure $\mathbb{Q}^{(k)}_{\theta, T}$ proves invaluable in that questions about past genealogies of particles sampled uniformly at time $T$ under $\mathbb{P}^{(k)}_{{\rm unif},T}$ may be converted into more tractable questions about the forward-in-time behaviour of the spines under $\mathbb{Q}^{(k)}_{\theta,T}$.
Let us introduce the probability measure $\mathbb{P}^{(k)}_{{\rm unif},T }$ as follows. 
Let $f$ be a measurable functional on the genealogies of $k$-tuples of particles. 
That is, let $f$ be a functional of the (Ulam-Harris labelling of) the ancestors of the $k$ particles, their birth times, death times, and the number of offspring they have upon death.
Let 
$\xi_T=(\xi^{(1)}_T, \ldots, \xi^{(k)}_T)$
be a uniform sample without replacement at time $T$ taken from a Galton-Watson process conditioned on the event that $\{Z_T\ge k\}$. 

Define the probability measure $\mathbb{P}^{(k)}_{{\rm unif},T }$ on $\{Z_T\geq k\}$ as follows
\begin{equation}\label{punif}
\begin{split}
\mathbb{E}^{(k)}_{{\rm unif},T }&\Big[f(\xi_T)\Big|Z_T\ge k\Big]
=\mathbb{E}{}\left[\frac{1}{Z_T^{{(k)}}}\sum_{\underline{u} \in \mathcal{Z}^{(k)}_T}f(\underline{u})\Bigg|Z_T\ge k\right],
\end{split}
\end{equation}
where 
the first term of the right hand side of $\eqref{punif}$ is the probability for any given choice of distinct $\underline{u} \in \mathcal{Z}^{(k)}_T$. 
Equivalently, we can make the definition
\begin{equation}
\frac{\ud \mathbb{P}^{(k)}_{{\rm unif},T }(\,\cdot\, | Z_T\geq k)}{\ud \mathbb{P}^{(k)}}\Bigg|_{\mathcal{F}^{(k)}_T}
:=
\frac{\mathbf{1}_{A_{k, T}}}{Z_T^{{(k)}}}\cdot
\prod_{i=1}^k\prod_{v\prec \xi^{(i)}_T} L_v.
\label{punifdef}
\end{equation}
In particular, recalling \eqref{ksizeunif}, we directly see the size-biased relationship between $\mathbb{Q}^{(k)}_{\theta, T}$ and $\mathbb{P}^{(k)}_{{\rm unif},T }$ as
\[
\frac{\ud \mathbb{Q}^{(k)}_{\theta, T}}{\ud \mathbb{P}^{(k)}_{{\rm unif},T }(\,\cdot\, | Z_T\geq k)}\Bigg|_{\mathcal{F}^{(k)}_T}
=\frac{Z_T^{{(k)}} e^{-\theta Z_T}}{\mathbb{E}^{(k)}\Big[Z_T^{{(k)}} e^{-\theta Z_T}\Big]}.
\]

For the particular case of the splitting times  of distinct $k$-tuples of particles alive at time $T$ under the event of $\{Z_T\ge k\}$, we have the following relationship between $\mathbb{P}^{(k)}_{{\rm unif},T }$ and $\mathbb{Q}^{(k)}_{\theta, T}$ which will 
be very useful for our purposes. 

\begin{lemma} Let $f$ be a measurable functional on the genealogies of $k$-tuples of particles. Then
\begin{equation}\label{punif1}
\mathbb{E}^{(k)}_{{\rm unif},T }\left[ f(\xi_T)\Big|Z_T\ge k\right]=\mathbb{Q}^{(k)}_{\theta, T}\left[\frac{f(\xi_T)}{Z^{{(k)}}_T e^{-\theta Z_T}}\right]\mathbb{E}\left[Z^{{(k)}}_T e^{-\theta Z_T}\Big|Z_T\ge k\right].
\end{equation}
In particular, if $\tau_1, \ldots,\tau_{k-1}$ are the split times of the $k$ uniformly chosen individuals $\xi_T\in \mathcal{Z}_T^{(k)}$, then we have
\[
\begin{split}
\mathbb{P}^{(k)}_{{\rm unif},T }\Big(\tau_1\in \ud t_1, \ldots, &\tau_{k-1}\in \ud t_{k-1}\Big|Z_T\ge k\Big)\\
&=\mathbb{Q}^{(k)}_{\theta, T}\left(\frac{1}{Z_T^{{(k)}}}\mathbf{1}_{\{\tau_1\in \ud t_1, \ldots, \tau_{k-1}\in \ud t_{k-1}\}}\right)\mathbb{E}\Big[Z_T^{{(k)}}e^{-\theta Z_T}\Big | Z_{T}\ge k\Big].
\end{split}
\]
\end{lemma}
\begin{proof} We first observe from \eqref{qk-form} that for a functional along the spines $\xi_{T}=(\xi^{(1)}_T, \ldots, \xi^{(k)}_T)$, we have
\[
\begin{split}
\mathbb{Q}^{(k)}_{\theta, T}\left[f(\xi_T)|\mathcal{F}_T\right]&=\mathbb{Q}^{(k)}_{\theta, T}\left[\sum_{u\in\mathcal{Z}^{(k)}_T}\mathbf{1}_{\{\xi_T=\underline{u}\}}f(\underline{u})\bigg|\mathcal{F}_T\right]\\
&=\sum_{\underline{u}\in\mathcal{Z}^{(k)}_T}f(\underline{u})\mathbb{Q}^{(k)}_{\theta, T}(\xi_T=\underline{u}|\mathcal{F}_T)\\
&=\frac{1}{Z^{{(k)}}_T}\sum_{\underline{u}\in\mathcal{Z}^{(k)}_T}f(\underline{u}).
\end{split}
\]
Hence from \eqref{punif} and the above identity, it follows
\[
\begin{split}
\mathbb{E}^{(k)}_{{\rm unif},T }\left[ f(\xi_T)\Big|Z_T\ge k\right]&=\mathbb{E}\left[\frac{1}{Z_T^{{(k)}}}\sum_{\underline{u}\in\mathcal{Z}^{(k)}_T}f(\underline{u})\Bigg|Z_T\ge k\right]\\
&=\mathbb{E}\left[\mathbb{Q}^{(k)}_{\theta, T}\left[f(\xi_T)|\mathcal{F}_T\right]\Bigg|Z_T\ge k\right].\\
\end{split}
\]
Finally using \eqref{newprob1}, we deduce
\[
\mathbb{E}^{(k)}_{{\rm unif},T }\left[ f(\xi_T)\Big|Z_T\ge k\right]=\mathbb{Q}^{(k)}_{\theta, T}\left[\frac{f(\xi_T)}{Z^{{(k)}}_T e^{-\theta Z_T}}\right]\mathbb{E}\left[Z^{{(k)}}_T e^{-\theta Z_T}\Big|Z_T\ge k\right].
\]
This completes the proof. 
\end{proof}
As we will show in the next section,  one can describe the spine behaviour completely under $\mathbb{Q}^{(k)}_{\theta, T}$ and  hence implicitly, thanks to the above result, under $\mathbb{P}^{(k)}_{{\rm unif},T }$.  
We also observe that if we choose $\theta$ to depends on $T$ appropriately, the term on the far right-hand side of \eqref{punif1} can be computed using a relevant Yaglom theorem. (For instance, in the finite variance case, we have $Z_T/T$ converges in distribution to an exponential distribution as $T\rightarrow\infty$.) Further, this will similarly be the case for other individual terms appearing in the description of $\mathbb{Q}^{(k)}_{\theta, T}$. This scaling procedure will lead to a limiting $k$-spine construction, and ultimately determine the asymptotic genealogies from uniform samples at large times, as desired.

\section{Behaviour of Galton-Watson with $k$-spines under $\mathbb{Q}^{(k)}_{\theta, T}$.} \label{sec:Q}

\subsection{Behaviour of the spines and births off the spine}
In this section, we are interested in computing explicitly some functionals of Galton-Watson trees with $k$-spines under $\mathbb{Q}^{(k)}_{\theta, T}$ that will prove essential in the forthcoming sections.  
Before delving into the proofs, let us give a brief overview of some key results in this section:
\begin{itemize}
\item According to Lemma \ref{lem:bots}, 
if a particle is carrying $j \geq 1$ spines, then `births-off-the-spine' (that is, births along this particle after which all of the $j$ spines follow the same offspring particle) occur at rate 
\[
r\mathbb{E}\left[L\left(\mathbb{E}[e^{-\theta Z_{T-t}}]\right)^{L-1}\right]
\]
at time $t$, and when that they do occur, the number of offspring are distributed according to the size-biased probability distribution
\begin{align*}
\mathbb{P}( \text{A birth-off-the-spine at time $t$ has size $\ell$} ) = \frac{ \ell \mathbb{E}[e^{ -\theta Z_{T-t} } ]^{\ell-1} ] p_\ell}{ \mathbb{E}[L \mathbb{E}[e^{-\theta Z_t} ]^{L-1} ] } .
\end{align*}
It is important to note that both these quantities (that is, the rate and offspring distribution) for the births off the spine are independent of $j$, the number of spines being carried. 
\item Another main result of this section is Lemma \ref{lem:Qmain}, which 
reveals that
\begin{align}
 \mathbb{Q}^{(k)}_{\theta, T}&\Big(\tau_1 \in \ud t_1,\ldots,\tau_m \in \ud t_m, \mathcal{T}(\xi) = (\beta_0,\ldots,\beta_m), L_{\tau_1} = \ell_1,\ldots, L_{\tau_m} = \ell_m \Big) \nonumber \\
& \hspace{4cm}= \frac{F_T'(e^{-\theta}) }{\mathbb{E}[Z_T^{(k)}e^{ - \theta Z_T }]}  \prod_{ i = 1}^m \ell_i^{(g_i)}\left(\mathbb{E}[e^{-\theta Z_{T-t_i}}]\right)^{\ell-g}p_{\ell_i} F_{T-t_i}'(e^{-\theta})^{g_i-1},
\end{align}
where $F_t(s)=\mathbb{E}[s^{Z_t}]$, for $|s|<1$ and $t\ge 0$,    $0  < t_1 < \ldots < t_m < T$ are times, $(\beta_0,\ldots,\beta_m)$ is a splitting process of $\{1,\ldots,k\}$ with split sizes $g_1,\ldots,g_m$, and $\ell_1,\ldots,\ell_m$ are integers with $\ell_i \geq g_i$. 

\end{itemize}
As mentioned, the spine change of measure results under $\mathbb{Q}^{(k)}_{\theta, T}$ of this section generalise some earlier spine approaches for the continuous time GW found in \cite{J19} (section 4) and \cite{HJR20} (section 4), with the significant addition of exponentially $\theta$-discounting by the final population size in addition to $k$-size biasing. 
Whilst this introduces extra complexity, all key properties are preserved, and the discounting is crucial to permit the spine approach to be applied when heavy tailed offspring distributions are present.

First, we compute the event that there are no births along the spine by time $t$.  The proof of the following result follows similar arguments as those of Lemma 9 in \cite{HJR20}, we provide its proof for the sake of completeness.
\begin{lemma}\label{nobirthst} Let $\chi_1$ be the time of the first birth event in the entire population. 
Then
\[
\mathbb{Q}^{(k)}_{\theta, T}(\chi_1>t)=e^{-rt}\frac{\mathbb{E}^{(k)}[Z^{{(k)}}_{T-t}e^{-\theta Z_{T-t}}]}{\mathbb{E}^{(k)}[Z^{{(k)}}_{T}\e^{-\theta Z_{T}}]}.
\]
\end{lemma}
\begin{proof} Recall that  $A_{k,T}$ denotes the event that spines are separated and alive by time $T$. Then 
\begin{align*}
\mathbb{Q}^{(k)}_{\theta, T}(\chi_1>t)&=\frac{\mathbb{E}^{(k)}\left[\left(\prod_{i=1}^k\prod_{v\prec \xi^{(i)}_T}L_v\right) e^{-\theta Z_{T}}\mathbf{1}_{\{\chi_1>t, A_{k,T}\}}\right]}{\mathbb{E}^{(k)}[Z^{{(k)}}_{T}e^{-\theta Z_{T}}]}\\
&=\mathbb{P}(\chi_1>t)\frac{\mathbb{E}^{(k)}\left[\left(\prod_{i=1}^k\prod_{v\prec\xi^{(i)}_T}L_v\right) e^{-\theta Z_{T}}\mathbf{1}_{ A_{k,T}}\Bigg| Z_t=1\right]}{\mathbb{E}^{(k)}[Z^{{(k)}}_{T}e^{-\theta Z_{T}}]}\\
&=e^{-rt}\frac{\mathbb{E}^{(k)}[Z^{{(k)}}_{T-t}e^{-\theta Z_{T-t}}]}{\mathbb{E}^{(k)}[Z^{{(k)}}_{T}e^{-\theta Z_{T}}]},
\end{align*}
where the third equality follows from the Markov property and identity \eqref{expnprob}.
\end{proof}
Recall that we call birth events that occur along the spines but which do not occur at spine splitting events, {\it births off the spine}. The following Lemma tell us the distribution of the number of children at the first birth off spines at time $t$.
\begin{lemma}\label{firstbe} Let $B_{\chi_1}$ be the event that spines stay together at time $\chi_1$ and $L_{\chi_1}$ be the number of offspring at the first birth event, then
\begin{equation*}
\mathbb{Q}^{(k)}_{\theta, T}(\chi_1\in {\rm d}t, L_{\chi_1}=\ell; B_{\chi_1} )=\ell\left(\mathbb{E}[e^{-\theta Z_{T-t}}]\right)^{\ell-1}p_\ell  re^{-rt}\frac{\mathbb{E}^{(k)}[Z^{{(k)}}_{T-t}e^{-\theta Z_{T-t}}]}{\mathbb{E}^{(k)}[Z^{{(k)}}_{T}e^{-\theta Z_{T}}]}{\rm d} t.
\end{equation*}
\end{lemma}
\begin{proof} We first observe that if the first  particle has $\ell$ offspring then the  probability  all $k$ spines follow the same one of these offspring is $1/\ell^{k-1}$. Thus  from the Markov property, Lemma \ref{nobirthst} and identity \eqref{expnprob}, we have 
\begin{align*}
\mathbb{Q}^{(k)}_{\theta, T}(\chi_1\in {\rm d}t, &L_{\chi_1}=\ell; B_{\chi_1} )=\frac{\mathbb{E}^{(k)}\left[\left(\prod_{i=1}^k\prod_{v\prec\xi^{(i)}_T}L_v\right) e^{-\theta Z_{T}}\mathbf{1}_{\{ A_{k,T}; \chi_1\in {\rm d}t, L_{\chi_1}=\ell; B_{\chi_1}\}} \right]}{\mathbb{E}^{(k)}[Z^{{(k)}}_{T}e^{-\theta Z_{T}}]}\\
&=re^{-rt}\ud t \,\,p_\ell\left(\frac{1}{\ell}\right)^{k-1}\ell^k \left(\mathbb{E}[e^{-\theta Z_{T-t}}]\right)^{\ell -1}\frac{\mathbb{E}^{(k)}[Z^{{(k)}}_{T-t}e^{-\theta Z_{T-t}}]}{\mathbb{E}^{(k)}[Z^{{(k)}}_{T}e^{-\theta Z_{T}}]}\\
&=\ell\left(\mathbb{E}[e^{-\theta Z_{T-t}}]\right)^{\ell-1}p_\ell  re^{-rt}\frac{\mathbb{E}^{(k)}[Z^{{(k)}}_{T-t}e^{-\theta Z_{T-t}}]}{\mathbb{E}^{(k)}[Z^{{(k)}}_{T}e^{-\theta Z_{T}}]}{\rm d} t.
\end{align*}
\end{proof}

From the previous lemma, by summing over all possible values of $\ell$, we deduce that
\begin{align*}
\mathbb{Q}^{(k)}_{\theta, T}(\chi_1\in {\rm d}t,  B_{\chi_1} )=\mathbb{E}\left[L\left(\mathbb{E}[e^{-\theta Z_{T-t}}]\right)^{L-1}\right] re^{-rt}\frac{\mathbb{E}^{(k)}[Z^{{(k)}}_{T-t}e^{-\theta Z_{T-t}}]}{\mathbb{E}^{(k)}[Z^{{(k)}}_{T}e^{-\theta Z_{T}}]}{\rm d} t.
\end{align*}
The following result computes the probability of the event that spines do not split apart by time $t$.
\begin{lemma} \label{Lspinenot}Let $\tau_1$ be the first time that the spines split apart, then
\begin{equation}\label{Spinenot}
\mathbb{Q}^{(k)}_{\theta, T}(\tau_1>t)=\exp\left\{-\int_0^t r\left(1-\mathbb{E}\left[L\left(\mathbb{E}[e^{-\theta Z_{T-s}}]\right)^{L-1}\right]\right)\ud s\right\} \frac{\mathbb{E}^{(k)}[Z^{(k)}_{T-t}e^{-\theta Z_{T-t}}]}{\mathbb{E}^{(k)}[Z^{(k)}_{T}e^{-\theta Z_{T}}]}.
\end{equation}
\end{lemma}
\begin{proof} Let   $G_t$ be the event that no splitting has occurred before time $t$, thus from the definition of $\mathbb{Q}^{(k)}_{\theta, T}$ we have
\begin{align*}
\mathbb{Q}^{(k)}_{\theta, T}(\tau_1>t)&=\frac{\mathbb{E}^{(k)}\left[\left(\prod_{i=1}^k\prod_{v\prec\xi_T^{(i)}}L_v \right)e^{-\theta Z_{T}}\mathbf{1}_{\{\tau_1>t; \,\, A_{k,T}\}}\right]}{\mathbb{E}^{(k)}[Z^{(k)}_{T}e^{-\theta Z_{T}}]}\\
&=\frac{\mathbb{E}^{(k)}\left[\left(\prod_{i=1}^k\prod_{v\prec \xi_T^{(i)}}L_v\right)\, e^{-\theta Z_{T}}\mathbf{1}_{\{\tau_1>t;  \,\, A_{k,T};  \,\, G_t\}}\right]}{\mathbb{E}^{(k)}[Z^{(k)}_{T}e^{-\theta Z_{T}}]}.
\end{align*}
Expanding the product, this further reduces to
\begin{align*}
\mathbb{Q}^{(k)}_{\theta, T}(\tau_1>t)&=\frac{1}{\mathbb{E}^{(k)}[Z^{(k)}_{T}e^{-\theta Z_{T}}]}\mathbb{E}^{(k)}\left[\displaystyle\prod_{u\preceq \xi_t^{(1)}}\left(\frac{1}{L_{u}}\right)^{k-1}\mathbf{1}_{\{\tau_1>t; \,\, A_{k,T}\}}\right.\\
&\hspace{1cm}\times\left. \displaystyle\left(\prod_{i=1}^k\prod_{\substack{\xi_t^{(i)}\prec v\prec \xi_T^{(i)}}}L_v\prod_{\substack{w\preceq \xi_t^{(i)},}}L_w\right)e^{-\theta Z^{(1)}_{T-t}}\prod_{u\preceq \xi_t^{(1)}}e^{-\theta\sum_{k=1}^{L_{u}-1} Z_{T-\sigma_{u}}^{u,k}}\right]\\
&=\frac{1}{\mathbb{E}^{(k)}[Z^{(k)}_{T}e^{-\theta Z_{T}}]}\mathbb{E}^{(k)}\left[\mathbf{1}_{\{\tau_1>t; \,\, A_{k,T}\}}\displaystyle\prod_{u\preceq \xi^{(1)}_t}\left(\frac{1}{L_u}\right)^{k-1}\right.\\
&\hspace{1cm}\times\left.\displaystyle\left(\prod_{\substack{u\preceq \xi^{(1)}_t}}L^k_u\right)\left(\prod_{i=1}^k\prod_{\xi_t^{(i)}\prec w\prec\xi_T^{(i)}}L_w\right)e^{-\theta Z^{(1)}_{T-t}}\prod_{u\preceq \xi^{(1)}_t}e^{-\theta\sum_{j=1}^{L_u-1} Z_{T-\sigma_u}^{u,j}}\right],
\end{align*}
where $\sigma_u$ denotes the time of  the birth off spine of particle $u$
 and  $Z^{u,j}_{T-\sigma_u}$ represents the contribution of the $j$-child of the particle $u$ to the population alive at time $T$.  Moreover,  $Z^{(1)}_{T-t}$ denotes the contribution of  $\xi^{(1)}_t$  to the population alive at time $T$ implying that  $Z_T$ can be rewritten as follows
 \[
 Z_T=\sum_{u\preceq\xi_t^{(1)}}\sum_{j=1}^{L_u-1} Z_{T-\sigma_u}^{u,j} + Z^{(1)}_{T-t}.
 \]
 It is important to note that for each node $u$, the random variables $(Z^{u,j}_{T-\sigma_u}, j\ge 1)$ are i.i.d., conditionally on $\sigma_u$, and that for $u\prec v\preceq \xi^{(1)}_t$ the families $(Z^{u,j}_{T-\sigma_u}, j\ge 1)$ and $(Z^{v,j}_{T-\sigma_v}, j\ge 1)$ are independent, conditionally on $(\sigma_u, \sigma_v)$. Besides, we simplify the above identity  by observing  that
\[
\mathbf{1}_{\{\tau_1>t; \,\, A_{k,T}\}}\prod_{u\preceq \xi^{(1)}_t}\left(\frac{1}{L_u}\right)^{k-1}=\mathbf{1}_{\{A_{k;t,T}\}}\prod_{u\preceq \xi^{(1)}_t}\left(\frac{1}{L_u}\right)^{k-1},
\]
where $A_{k;t,T}$ denotes the event that spines split over $(t,T]$ and use the Markov branching property at time $t$ to obtain
\begin{align*}
\mathbb{Q}^{(k)}_{\theta, T}(\psi_1>t)&=\frac{\mathbb{E}^{(k)}\left[\displaystyle\mathbf{1}_{A_{k;t,T}}\prod_{u\preceq \xi_t^{(1)}}L_u\,e^{-\theta\sum_{i=1}^{L_u-1} Z_{T-\sigma_u}^{u,i}}\left(\prod_{i=1}^k\prod_{\xi_t^{(i)}\prec v\prec \xi_T^{(i)}}L_v\right)e^{-\theta Z^{(1)}_{T-t}}\right]}{\mathbb{E}^{(k)}[Z^{(k)}_{T}e^{-\theta Z_{T}}]}\\
&=\mathbb{E}^{(k)}\left[\prod_{u\preceq \xi_t^{(1)}}L_u\,e^{-\theta\sum_{i=1}^{L_u-1} Z_{T-\sigma_u}^{u,i}}\right]\frac{\mathbb{E}^{(k)}\Big[Z^{(k)}_{T-t}e^{-\theta Z_{T-t}}\Big]}{\mathbb{E}^{(k)}\Big[Z^{(k)}_{T}e^{-\theta Z_{T}}\Big]}\\
&=\exp\left\{-\int_0^t r\left(1-\mathbb{E}\left[L\left(\mathbb{E}[e^{-\theta Z_{T-s}}]\right)^{L-1}\right]\right)\ud s\right\} \frac{\mathbb{E}^{(k)}\Big[Z^{(k)}_{T-t}e^{-\theta Z_{T-t}}\Big]}{\mathbb{E}^{(k)}\Big[Z^{(k)}_{T}e^{-\theta Z_{T}}\Big]},
\end{align*}
where in the last identity we have used Campbell's formula, i.e.
\begin{equation}\label{Campbell}
\begin{split}
\mathbb{E}^{(k)}\left[\prod_{u\preceq \xi_t^{(1)}}L_u\,e^{-\theta\sum_{i=1}^{L_u-1} Z_{T-\sigma_u}^{u,i}}\right]&=\mathbb{E}^{(k)}\left[\exp\left\{\sum_{i=1}^{\eta_t}\left(\ln (L_i)-\theta\sum_{j=1}^{L_i-1} Z_{T-\sigma_i}^{i,j}\right)\right\}\right]\\
&=\exp\left\{-\int_0^t r\left(1-\mathbb{E}\left[L\left(\mathbb{E}[e^{-\theta Z_{T-s}}]\right)^{L-1}\right]\right)\ud s\right\}.
\end{split}
\end{equation}
More precisely, let $\eta_t$ represents the number of births off spine up to time $t$ and $L_i$ the number of offspring of the $i$-th birth off spine. The process $(\eta_t,t\ge 0)$ is clearly  a Poisson  process with intensity $\lambda(t):=\int_0^t r \ud s$, under $\mathbb{P}^{(k)}$. Recall that in this case Campbell's formula is stated as follows: for  $Y_t=\sum_{i=1}^{\eta_t} X_i(\sigma_i)$, where $(X_i (\cdot), i\ge 1)$ are independent and $\sigma_i$ represents the time of the $i$-th arrival, we have
\[
\mathbb{E}\left[e^{-\theta Y_t}\right]=\exp\left\{\int_0^t\lambda(s) \left(1-\mathbb{E}\Big[e^{-\theta X(s)}\Big]\right)\ud s\right\}.
\]
This completes the proof.
\end{proof}

Whilst the above proof provides some clear probabilistic insight, note that an alternative approach to compute the event that all spines are togther at time $t$ was given in  \cite{HJR20}. 
More precisely in the current context, the probability $\mathbb{Q}^{(k)}_{\theta, T}(\psi_1>t)$  can be computed without using Campbell's formula and equivalently written in terms of the derivatives of $F_t(s)=\E(s^{Z_t})$, as follows.

\begin{lemma} \label{Lspinenotbis}Let $\tau_1$ be the first time that spines splits apart, then
\begin{equation}\label{Spinenotalt}
\begin{split}
\mathbb{Q}^{(k)}_{\theta, T}(\tau_1>t)&= \frac{F_{T-t}^{(k)}(e^{-\theta})}{F_T^{(k)}(e^{-\theta})}  \frac{F_T'(e^{-\theta})}{F_{T-t}'(e^{-\theta})}, 
\end{split}
\end{equation}
where  $F_t^{(j)}(s) := \frac{ \partial^j}{\partial s^j} F_t(s)$.
\end{lemma}
\begin{proof}
Using the same notation as in the previous Lemma, we observe that under  $G_t$, we may decompose the population $Z_T$ as 
\begin{align*}
Z_T = Z_T' + Z_T'',
\end{align*}
where $Z_T'$ are the descendents of the unique particle the $k$ spines are following at time $t$, and $Z_T''$ counts the rest of the particles at time $T$. Observe that $Z_T'' = Z_T-Z_T'$ and that, conditionally on $G_t$, $Z_T'$ and $Z_T''$ are  independent. Hence
\[
\begin{split}
&\mathbb{E}^{(k)}\left[\left(\prod_{i=1}^k\prod_{v\prec\xi_T^{(i)}}L_v \right)e^{-\theta Z_{T}}\mathbf{1}_{\{ G_t; \,\, A_{k,T}\}}\right]\\
 &= \mathbb{E}^{(k)}\left[\mathbb{E}^{(k)}\left[\left(\prod_{i=1}^k\prod_{v\prec\xi_T^{(i)}}L_v \right)e^{-\theta Z_{T}}\mathbf{1}_{\{ G_t; \,\, A_{k,T}\}} \Bigg| \mathcal{F}^{(k)}_t \right] \right] \\
&= \mathbb{E}^{(k)}\left[\left(\prod_{i=1}^k\prod_{ v\prec\xi_t^{(i)}}L_v \right) \mathbb{E}^{(k)}\left[\left(\prod_{i=1}^k\prod_{\xi_t^{(i)} \preceq v\prec\xi_T^{(i)}}L_v \right)e^{-\theta Z'_{T}}\mathbf{1}_{\{ A_{k,T}\}} \Bigg| \mathcal{F}^{(k)}_t \right] \mathbb{E}[ e^{ - \theta Z_T'' } | \mathcal{F}_t] \mathbf{1}_{\{G_t \}} \right].
\end{split}
\]
On the one hand, we have $\mathbb{E}[ e^{ - \theta Z_T'' } | \mathcal{F}_t]  = F_{T-t}(e^{-\theta})^{Z_t -1}.$
On the other, we have
\[
\mathbb{E}^{(k)}\left[\left(\prod_{i=1}^k\prod_{\xi_t^{(i)} \preceq v\prec\xi_T^{(i)}}L_v \right)e^{-\theta Z'_{T}}\mathbf{1}_{\{ A_{k,T}\}} \Bigg| \mathcal{F}^{(k)}_t \right] = \mathbb{E}^{(k)} [ Z_{T-t}^{(k)} e^{ - \theta Z_{T-t}} ],
\]
independently of $\mathcal{F}^{(k)}_t$ given $G_t$. 
Putting all pieces together, we obtain 
\begin{align} \label{eq:timet2}
\mathbb{E}^{(k)}&\left[\left(\prod_{i=1}^k\prod_{v\prec\xi_T^{(i)}}L_v \right)e^{-\theta Z_{T}}\mathbf{1}_{\{ G_t; \,\, A_{k,T}\}}\right] \nonumber\\
&= \mathbb{E}^{(k)} [ Z_{T-t}^{(k)} e^{ - \theta Z_{T-t}} ] \mathbb{E}^{(k)}\left[\left(\prod_{i=1}^k\prod_{ v\prec\xi_t^{(i)}}L_v \right) F_{T-t}(e^{-\theta})^{Z_t -1} \mathbf{1}_{\{G_t \}} \right].
\end{align}
Finally, by summing over the possible elements of the time $t$ population we deduce
\begin{align} \label{eq:Y}
\mathbb{E} \left[ \left(\prod_{i=1}^k\prod_{ v\prec\xi_t^{(i)}}L_v \right) \mathbf{1}_{\{G_t \}}  \Bigg| \mathcal{F}^0_t \right] = Z_t.
\end{align}
Taking $\mathcal{F}_t$-conditional expectations in \eqref{eq:timet2} and using \eqref{eq:Y}, we deduce
\begin{equation} \label{eq:timet3}
\begin{split}
\mathbb{E}^{(k)}\left[\left(\prod_{i=1}^k\prod_{v\prec\xi_T^{(i)}}L_v \right)e^{-\theta Z_{T}}\mathbf{1}_{\{ G_t; \,\, A_{k,T}\}}\right]
&= \mathbb{E}^{(k)} [ Z_{T-t}^{(k)} e^{ - \theta Z_{T-t}} ] \mathbb{E}^{(k)}\left[ Z_t F_{T-t}(e^{-\theta})^{Z_t -1} \right]\\ &= F_{T-t}^{(k)}( e^{ - \theta}) F_t'(F_{T-t}(\theta)).
\end{split}
\end{equation}
Let us note that from the semigroup identity $F_T(s) = F_t(F_{T-t}(s))$ we have
\begin{align} \label{eq:mahler}
F_t'(F_{T-t}(s)) = F_T'(s)/F_{T-t}'(s).
\end{align}
Substituting \eqref{eq:timet3} into \eqref{eq:timet2} and using \eqref{eq:mahler} allows us to deduce \eqref{Spinenotalt}.
\end{proof}

Noting that $\mathbb{E}^{(k)}[Z^{(k)}_{t}e^{-\theta Z_{t}}] = F^{(k)}_t(e^{-\theta})$, the agreement of the two representations of $\mathbb{Q}_{\theta,T}^{(k)}(\psi_1 > t)$ in Lemmas \ref{Lspinenot} and \ref{Lspinenotbis} amounts to the identity
\begin{align*}
\exp\left\{-\int_0^t r\left(1-\mathbb{E}\left[L\left(\mathbb{E}[e^{-\theta Z_{T-s}}]\right)^{L-1}\right]\right)\ud s\right\}  =  F_t^{\prime}(F_{T-t}(e^{-\theta})),
\end{align*}
the proof of which is a standard calculation in the theory of continuous-time Galton-Watson processes; see e.g.\ Lemma 2.5 of \cite{CJ}.

We continue our exposition by computing the rate of births off the spine.
\begin{lemma}\label{lem:bots} 
Let $\sigma_1$ represent the first time there is a birth off the spine event, that is, a birth event where all the spines stay together.
The rate of births occurring off the $k$ spines is given by
\begin{equation}\label{firstbirthos}
\begin{split}
\mathbb{Q}^{(k)}_{\theta, T}&(\sigma_1\in \ud t; L_{\sigma_1}=\ell |\tau_1>t)=\ell\left(\mathbb{E}[e^{-\theta Z_{T-t}}]\right)^{\ell-1}p_\ell  r e^{-rt} \frac{F_{T-t}'(e^{-\theta})}{F_T'(e^{-\theta})}.\\
\end{split}
\end{equation}
\end{lemma}
\begin{proof}
In order to deduce this identity, we will use Lemmas \ref{firstbe} and \ref{Lspinenotbis}. Observe that under the event $\{\tau_1>t\}$, the spines stay together by time $t$ and that, by time  $t$, the first birth event  and the first birth off spines are   the same.  More precisely, we have
\[
\begin{split}
\mathbb{Q}^{(k)}_{\theta, T}(\sigma_1\in \ud t; &L_{\sigma_1}=\ell |\tau_1>t)=\frac{\mathbb{Q}^{(k)}_{\theta, T}(\sigma_1\in \ud t; L_{\sigma_1}=\ell , B_{\sigma_1}, \tau_1>t)}{\mathbb{Q}^{(k)}_{\theta, T}(\tau_1>t)}\\
&=\frac{\mathbb{Q}^{(k)}_{\theta, T}(\chi_1\in \ud t; L_{\chi_1}=\ell , B_{\chi_1})}{\mathbb{Q}^{(k)}_{\theta, T}(\tau_1>t)}\\
&=\frac{\ell\left(\mathbb{E}[e^{-\theta Z_{T-t}}]\right)^{\ell-1}p_\ell  re^{-rt}\frac{\mathbb{E}[Z^{{(k)}}_{T-t}e^{-\theta Z_{T-t}}]}{\mathbb{E}[Z^{{(k)}}_{T}e^{-\theta Z_{T}}]}{\rm d} t}{\frac{F_{T-t}^{(k)}(e^{-\theta})}{F_T^{(k)}(e^{-\theta})}  \frac{F_T'(e^{-\theta})}{F_{T-t}'(e^{-\theta})} }\\
&=\ell\left(\mathbb{E}[e^{-\theta Z_{T-t}}]\right)^{\ell-1}p_\ell  r e^{-rt} \frac{F_{T-t}'(e^{-\theta})}{F_T'(e^{-\theta})},
\end{split}
\]
as required.
\end{proof}
Let us make a brief remark interpreting \eqref{firstbirthos}. Note that we may alternatively write 
\begin{equation}\label{firstbirthos2}
\begin{split}
\mathbb{Q}^{(k)}_{\theta, T}&(\sigma_1\in \ud t; L_{\sigma_1}=\ell |\tau_1>t)=\frac{\ell (\mathbb{E}[e^{-\theta Z_{T-t}}])^{\ell-1}p_\ell}{\mathbb{E}\left[L\left(\mathbb{E}[e^{-\theta Z_{T-t}}]\right)^{L-1}\right]} \times \mathbb{E}\left[L\left(\mathbb{E}[e^{-\theta Z_{T-t}}]\right)^{L-1}\right] r e^{-rt} \frac{F_{T-t}'(e^{-\theta})}{F_T'(e^{-\theta})}\ud t.
\end{split}
\end{equation}
As mentioned in the opening of this section, the first term 
\[
\frac{\ell (\mathbb{E}[e^{-\theta Z_{T-t}}])^{\ell-1}p_\ell}{\mathbb{E}\left[L\left(\mathbb{E}[e^{-\theta Z_{T-t}}]\right)^{L-1}\right]}
\]
 in the above identity is a probability mass function in the variable $\ell$. The latter term represents the total rate of births off the spine. In other words, the rate of births off a spine branch and number of births off a spine branch are size biased and discounted by time to go, and are independent of the number of spines following that branch $k$.

Now, we compute the probability of the event that spines split at time $t$ with $\ell$ offspring.
Suppose $k$ spines are split into $g\in \{2, \ldots, k\}$ groups of sizes $k_1, k_2, \ldots k_g\ge 1$. Let $h_i$ be the number of groups of size $i$. We note that
\[
\sum_{i=1}^g k_i=k\qquad \textrm{and}\qquad \sum_{j=1}^kjh_j=k.
\]
\begin{lemma}\label{Splitsgspines}
Let $C_{g; k_1, \ldots, k_g}$ be the event that at the first splitting event, the spines split into $g$ groups of sizes $k_1, \ldots, k_g$. Then
\[
\begin{split}
\mathbb{Q}^{(k)}_{\theta, T}&\Big(\tau_1 \in \ud t; C_{g; k_1, \ldots, k_g}; L_{\tau_1}=\ell\Big)\\
&\hspace{2cm}=\ell^{(g)}\left(\mathbb{E}[e^{-\theta Z_{T-t}}]\right)^{\ell-g}p_\ell \frac{k!}{\prod_{j=1}^k h_j!\prod_{i=1}^g k_i!}\frac{\prod_{i=1}^g \mathbb{E}[Z^{(k_i)}_{T-t}e^{-\theta Z_{T-t}}]}{\mathbb{E}[Z^{(k)}_{T}e^{-\theta Z_{T}}]}\frac{F_T'(e^{-\theta})}{F_{T-t}'(e^{-\theta})}  r\ud t.
\end{split}
\]
\end{lemma}
Observe that the term $\ell^{(g)}\left(\mathbb{E}[e^{-\theta Z_{T-t}}]\right)^{\ell-g}p_\ell$  represents the $g$-size biased discounted offspring.
\begin{proof} Recall that  $A_{k_j,T}$ represents the event that $k_j$ spines are separated and alive by time $T$, then
\[
\begin{split}
\mathbb{Q}^{(k)}_{\theta, T}&\Big(\tau_1 \in \ud t; C_{g; k_1, \ldots, k_g}; L_{\tau_1}=\ell\Big)\\
&=\frac{\mathbb{E}^{(k)}\left[\left(\displaystyle\prod_{i=1}^k\prod_{v\prec \xi^{(i)}_T}L_v \right) e^{-\theta Z_{T}};\tau_1 \in \ud t; G_t; C_{g; k_1, \ldots, k_g}; L_{\tau_1}=\ell; A_{k,T}\right]}{\mathbb{E}[Z^{(k)}_{T}e^{-\theta Z_{T}}]}\\
&=\frac{1}{\mathbb{E}[Z^{(k)}_{T}e^{-\theta Z_{T}}]} \mathbb{E}^{(k)}\left[\displaystyle\prod_{u\preceq \xi^{(1)}_t} \left(\frac{1}{L_u}\right)^{k-1}\binom{\ell}{g}\frac{g!}{\prod_{i=1}^{k-1} h_i!}\frac{k!}{\prod_{j=1}^{g} k_j!}\left(\frac{1}{\ell}\right)^k\mathbf{1}_{\{\tau_1 \in \ud t;L_{\tau_1}=\ell\}} \right.\\
&\hspace{3cm}\times\displaystyle\prod_{i=1}^{\ell-g}e^{-\theta Z^{(i)}_{T-t}}\prod_{u\preceq \xi^{(1)}_t}e^{-\theta\sum_{m=1}^{L_u-1} Z^{u,m}_{T-\sigma_u}}\left(\prod_{j=1}^g \mathbf{1}_{A_{k_j,T}} \prod_{i=1}^{k_j}\prod_{\xi^{(j)}_t\prec v\prec \xi^{(j,i)}_T}L_v \right)\\
&\hspace{10cm}\left. \times\prod_{u\preceq \xi^{(1)}_t}L_u^k \left(\prod_{n=1}^g e^{-\theta Z^{(n)}_{T-t}}\right)\right],
\end{split}
\]
where $\xi^{(j,i)}_T$ denotes the $i$-th spine at time $T$ whose ancestor is $\xi^{(j)}_t$ at time $t$ and, similarly in the proof of Lemma \ref{Lspinenot}, $\sigma_u$ denotes the time of  the birth off spine of particle $u$
 and  $Z^{u,j}_{T-\sigma_u}$ represents the contribution of the $j$-child of the particle $u$ to the population alive at time $T$.  Moreover,  $Z^{(i)}_{T-t}$ denotes the contribution of the $i$-th offspring    at first spine splitting to the population alive at time $T$. In other words,   $Z_T$ can be rewritten as follows
 \[
 Z_T=\sum_{u\preceq\xi_t^{(1)}}\sum_{j=1}^{L_u-1} Z_{T-\sigma_u}^{u,j} + \sum_{i=1}^{\ell}Z^{(i)}_{T-t}.
 \]
 It is important to note that for each node $u$, the random variables $(Z^{u,j}_{T-\sigma_u}, j\ge 1)$ are i.i.d., conditionally on $\sigma_u$, and that for $u\prec v\preceq \xi^{(1)}_t$ the families $(Z^{u,j}_{T-\sigma_u}, j\ge 1)$ and $(Z^{v,j}_{T-\sigma_v}, j\ge 1)$ are independent, conditionally on $(\sigma_u, \sigma_v)$.

Before we continue with the simplification of the terms inside the expectation, let us provide some explanations about each term. The first term, inside the expectation, in the second identity represents that all spines stay together at births before $t$. The second term inside the expectation (combinatorics) can be interpreted as follows: the  term $\binom{\ell}{g}$ represents the way we choose offspring which get spine groups, the next term is the way we choose which group gets which group size and the  third term is the number of ways dividing $k$ spines into groups with sizes $k_1,k_2,\ldots, k_g$. All three terms together represents the number of ways to get $g$ groups of sizes $k_1,k_2,\ldots, k_g$ with $\ell$ offsprings from $k$ spines. Finally the fourth term is the probability of particular spine choice when $\ell$ offspring.  The first term in the second row (just after the product sign) represents the contribution at time $T$ of the non-spine individuals. The next term are the contributions  to $Z_T$ from births off spines before time $t$. The three terms inside the brackets  provides the contributions from  $g$ groups with $k_1, \ldots, k_g$ spines after time $t$.  The last two terms are nothing but  $k$ spines follow the same path before time $t$ and the contribution  to $Z_T$ of the spines. 

Therefore, from the Markov branching property we deduce 
\[
\begin{split}
\mathbb{Q}^{(k)}_{\theta, T}&\Big(\tau_1 \in \ud t; C_{g; k_1, \ldots, k_g}; L_{\tau_1}=\ell\Big)=r\ud t\mathbb{E}\left[\prod_{u\prec \xi^{(1)}_{t}} L_ue^{-\theta\sum_{j=1}^{L_u-1} Z^{u,j}_{T-\sigma_u}}\right]\\
&\hspace{2cm}\times\frac{\ell!}{(\ell -g)!}\left(\mathbb{E}\left[e^{-\theta Z_{T-t}}\right]\right)^{\ell -g}p_\ell\frac{k!}{\prod_{i=1}^{k-1} h_i!\prod_{j=1}^{g} k_j!}\frac{\prod_{i=1}^g\mathbb{E}\Big[Z^{(k_i)}_{T-t}e^{-\theta Z_{T-t}}\Big]}{\mathbb{E}\Big[Z^{(k)}_{T}e^{-\theta Z_{T}}\Big]},
\end{split}
\]
where the desired identity follows after applying Campbell's formula, see \eqref{Campbell}.
\end{proof}
Combining Lemmas \ref{Lspinenot} and \ref{Splitsgspines}, we can also obtain the rate of $k$-spines splitting, that is,
\begin{equation}\label{spinesplit}
\begin{split}
\mathbb{Q}^{(k)}_{\theta, T}&\Big(\tau_1 \in \ud t; C_{g; k_1, \ldots, k_g}; L_{\tau_1}=\ell\Big| \tau_1>t\Big)=\frac{\ell^{(g)}\left(\mathbb{E}[e^{-\theta Z_{T-t}}]\right)^{\ell-g}}{\mathbb{E}\Big[L^{(g)}\left(\mathbb{E}[e^{-\theta Z_{T-t}}]\right)^{L-g}\Big]}p_\ell \\
&\hspace{3cm}\times\frac{k!}{\prod_{j=1}^k h_j!\prod_{i=1}^g k_i!}\frac{\prod_{i=1}^g \mathbb{E}[Z^{(k_i)}_{T-t}e^{-\theta Z_{T-t}}]}{\mathbb{E}[Z^{(k)}_{T-t}e^{-\theta Z_{T-t}}]} r\mathbb{E}\Big[L^{(g)}\left(\mathbb{E}[e^{-\theta Z_{T-t}}]\right)^{L-g}\Big]\ud t.
\end{split}
\end{equation}
It is important to note that  the  first term in the right hand side of \eqref{spinesplit}  represents the $g$-size biased and discounted offspring when split into $g$-groups and the remaining terms represent the rate of $k$-spines splitting into $g$-groups of sizes $k_1,\ldots,k_g$.

Importantly, there is another way of reinterpreting the above result  in terms of partitions. Recall that we write $(\mathcal{T}_0,\mathcal{T}_1,\ldots,\mathcal{T}_m)$ for the topology of the spine process - i.e. $\mathcal{T}_i$ is the spine partition after the $i^{\text{th}}$ split time. Hence  the joint law of the first spine split time $\tau_1$, the offspring size $L_{\tau_1}$ at the split, as well as $\mathcal{T}_1$, the partition describing the grouping of the spines after this first split satisfies.
\begin{lemma} \label{lem:firstsplit}
Let $\beta_1$ be a partition of $\{1,\ldots,k\}$ into $g$ blocks of sizes $k_1,\ldots,k_g$. Then
\begin{align} \label{eq:firstsplit}
\mathbb{Q}^{(k)}_{\theta, T}\Big(\tau_1 \in \ud t; \mathcal{T}_1 = \beta_1 ; L_{\tau_1}=\ell\Big) =\ell^{(g)}\left(\mathbb{E}[e^{-\theta Z_{T-t}}]\right)^{\ell-g}p_\ell \frac{\prod_{i=1}^g \mathbb{E}[Z^{(k_i)}_{T-t}e^{-\theta Z_{T-t}}]}{\mathbb{E}[Z^{{(k)}}_{T}e^{-\theta Z_{T}}]}\frac{F_T'(e^{-\theta})}{F_{T-t}'(e^{-\theta})}   r \ud t.
\end{align}
\end{lemma}
\begin{proof}
Let $k_1 \geq \ldots \geq k_g$ be integers with $k_1 + \ldots + k_g = k$. Let $h_j := \# \{ i : k_i = j \}$. Then there are
\begin{align*}
\frac{k!}{ k_1! \ldots k_g! h_1! \ldots h_k! } 
\end{align*}
different set partitions of $\{1,\ldots,k\}$ such that the block sizes, listed in decreasing order, are given by $k_1,k_2,\ldots,k_g$. Since the spines are exchangeable, it follows that given the event in Lemma \ref{Splitsgspines}, it is equally likely that any of these partitions occur. The result therefore follows by dividing through the formula in Lemma \ref{Splitsgspines} by the combinatorial factor ${k!}/{(k_1! \ldots k_g! h_1! \ldots h_k! )} $.
\end{proof}

We are now ready to state the main result of this section, providing an explicit description of the joint law of the spine ancestry process $(\pi_t^k)_{t \in [0,T]}$ (through its split times $\tau_1,\ldots,\tau_m$ and topology $(\beta_0,\ldots,\beta_m)$) and the birth sizes $L_{\tau_1},\ldots,L_{\tau_m}$ at these split times.

\begin{lemma} \label{lem:Qmain}
Let $0  < t_1 < \ldots < t_m < T$. Let $(\beta_0,\ldots,\beta_m)$ be a splitting process of $\{1,\ldots,k\}$ and let $g_1,\ldots,g_m$ be the associated split sizes. Let $\ell_1,\ldots,\ell_m$ be integers with $\ell_i \geq g_i$. Then 
\begin{align}
& \mathbb{Q}^{(k)}_{\theta, T}\Big(\tau_1 \in \ud t_1,\ldots,\tau_m \in \ud t_m, \mathcal{T}(\xi) = (\beta_0,\ldots,\beta_m), L_{\tau_1} = \ell_1,\ldots, L_{\tau_m} = \ell_m \Big) \nonumber \\
& = \frac{F_T'(e^{-\theta}) }{\mathbb{E}[Z_T^{(k)}e^{ - \theta Z_T }]}  \prod_{ i = 1}^m \ell_i^{(g_i)}\left(\mathbb{E}[e^{-\theta Z_{T-t_i}}]\right)^{\ell_i-g_i}p_{\ell_i} F_{T-t_i}'(e^{-\theta})^{g_i-1}\ud t_i.
\end{align}
\end{lemma}

\begin{proof}
We induct on $m$. The case $m = 1$ follows immediately from the previous lemma.

We now assume the result holds for $m' = 1,\ldots,m-1$ splits, and prove the result holds for $m' =m$ splits. Let $p_1 := \{ \Gamma_1,\ldots,\Gamma_{g_1} \}$ be the partition after the first split, so that $\Gamma_i$ contains $k_i$ elements. According to the symmetry lemma, after the first split time $\tau_1 = t_1$, the spines in each of the $g_1$ different groups behave independently from one another, and as if under $\mathbb{Q}^{(k_i)}_{\theta,T-t_1}$.

Consider now each of the subsequent splitting events $\tau_2,\ldots,\tau_m$. Each of these splitting events corresponds to some subblock of some $\Gamma_j$ ($1 \leq j \leq g_1$) breaking into smaller blocks. As such, we can reindex the $(\tau_i)_{i = 2,\ldots,m}$ as $(\tau_{i,j})_{1 \leq i \leq m_j, 1 \leq j \leq g_1}$, where $m_j\geq 0$ is the number of subsequent splitting event involving a subblock of $\Gamma_j$. Specifically, for $i \geq 1$, $\tau_{i,j}$ is the $i^{\text{th}}$ time, after $\tau_1$, some subblock of $\Gamma_j$ (possibly equal to $\Gamma_j$ itself) breaks into smaller sub-blocks. Under the same reindexing we can reindex $(t_i)_{i =2,\ldots,m}$ to $(t_{i,j})_{1 \leq i \leq m_j, 1 \leq j \leq g_1}$, $(\ell_i)_{i =2,\ldots,m}$ to $(\ell_{i,j})_{1 \leq i \leq m_j, 1 \leq j \leq g_1}$, and $(g_i)_{i=2,\ldots,m}$  to $(g_{i,j})_{1 \leq i \leq m_j, 1 \leq j \leq g_1}$.

Let $\mathcal{T}(\xi) := (\mathcal{T}_0,\ldots,\mathcal{T}_m)$ be the splitting sequence of the partition $\{1,\ldots,k\}$ associated with the spine splitting times. Given a block $\Gamma_j$ of $\mathcal{T}_1$, we define $\mathcal{T}^{\Gamma_j}(\xi) = (\mathcal{T}^{\Gamma_j}_0,\ldots,\mathcal{T}^{\Gamma_j}_{m_j}) $ to be the splitting sequence of $\Gamma_j$ associated with any spine splitting event involving a spine in $\Gamma_j$ after time $t_1$. $\mathcal{T}^{\Gamma_j}(\xi) = (\mathcal{T}^{\Gamma_j}_0,\ldots,\mathcal{T}^{\Gamma_j}_{m_j})$ is a sequence of partitions of $\Gamma_j$ with the property that $\mathcal{T}^{\Gamma_j}_0 = \{ \Gamma_j\}$, $\mathcal{T}^{\Gamma_j}_{m_j}$ is the singletons, and each $\mathcal{T}^{\Gamma_j}_{\ell+1}$ is obtained from $\mathcal{T}^{\Gamma_j}_\ell$ by breaking precisely one block of $\mathcal{T}^{\Gamma_j}_\ell$ into two or more sub-blocks. 

\begin{figure}[h!]
\centering
\begin{tikzpicture}[xscale=1,yscale=1]
\draw[gray] (0,-4.7) -- (0,3.3);
\draw[gray] (10,-4.7) -- (10,3.3);
\node at (0,-5)   (a) {Time $0$};
\node at (10,-5)   (a) {Time $T$};

\draw[very thick] (0,0) -- (3,0);
\node at (0.7,0.3)   (a) {\footnotesize{$\{1,2,3,4,5,6,7,8,9\}$}};
\draw[very thick] (3,-2) -- (3,2);
\node at (3.7,-1.7)   (a) {\footnotesize{$\{3,4,6,7,8,9\}$}};
\draw[very thick] (3,-2) -- (5.9,-2);
\draw[very thick] (5.9,-3.6) -- (5.9,-0.5);
\draw[very thick] (5.9,-0.5) -- (10,-0.5);
\node at (10.3,-0.5)   (a) {$3$};

\draw[thick, dotted] (3,-2) -- (3, -4.7);
\node at (3,-5)   (a) {$\tau_1$};

\draw[thick, dotted] (5.9,-3.6) -- (5.9, -4.7);
\node at (5.9,-5)   (a) {$\tau_2$};

\draw[thick, dotted] (7.6,1) -- (7.6, -4.7);
\node at (7.6,-5)   (a) {$\tau_3$};

\draw[thick, dotted] (8.3,-4.0) -- (8.3, -4.7);
\node at (8.3,-5)   (a) {$\tau_4$};

\draw[very thick] (5.9,-1.4) -- (10,-1.4);
\node at (10.3,-1.4)   (a) {$8$};

\draw[very thick] (5.9,-3.6) -- (8.3,-3.6);

\draw[very thick] (8.3,-2.2) -- (8.3,-4.0);
\draw[very thick] (8.3,-2.2) -- (10,-2.2);
\node at (10.3,-2.2)   (a) {$4$};
\node at (10.3,-2.8)   (a) {$6$};
\node at (10.3,-3.4)   (a) {$7$};
\node at (10.3,-4.0)   (a) {$9$};

\draw[very thick] (8.3,-4.0) -- (10,-4.0);
\draw[very thick] (8.3,-3.4) -- (10,-3.4);
\draw[very thick] (8.3,-2.8) -- (10,-2.8);

\draw[very thick] (3,0.3) -- (10,0.3);
\node at (10.3,0.3)   (a) {$1$};

\node at (3.7,0.6)   (a) {\footnotesize{$\{1\}$}};

\draw[very thick] (3,2) -- (7.6,2);
\node at (3.7,2.3)   (a) {\footnotesize{$\{2,5\}$}};
\draw[very thick] (7.6,1) -- (7.6,2.5);
\draw[very thick] (7.6,1) -- (10,1);
\node at (10.3,1)   (a) {$5$};
\draw[very thick] (7.6,2.5) -- (10,2.5);
\node at (10.3,2.5)   (a) {$2$};
\end{tikzpicture}
\caption{The ancestral tree of $k=9$ spines. After the initial split at time $\tau_1$ into blocks three blocks, the remaining split times and the splitting process can be divided across these three blocks.}
\label{fig:mahl}
\end{figure}
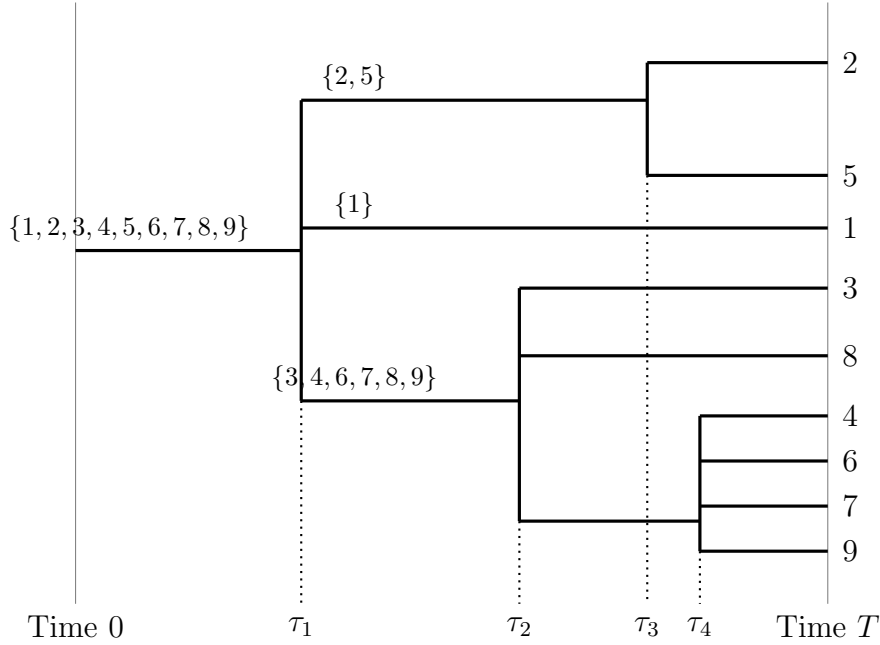

Let us work through some aspects of the example in Figure \ref{fig:mahl}. Here after the first split we break into three blocks, $\Gamma_1 = \{1\}, \Gamma_2 = \{2,5\}$ and $\Gamma_3 = \{3,4,6,7,8,9\}$. Each block $\Gamma_i$ has $m_i$ subsequent splits; here, $m_1 = 0, m_2 = 1, m_3 = 2$. We can reindex the subsequent split times $\tau_2,\tau_3,\tau_4$, so that for instance $\tau_{1,3} = \tau_2$ and $\tau_{2,3} = \tau_4$. Let us consider the sub- splitting process associated with the block $\Gamma_3 = \{3,4,6,7,8,9\}$. This is given by $\mathcal{T}^{\Gamma_3} = (\mathcal{T}^{\Gamma_3}_0,\mathcal{T}^{\Gamma_3}_1,\mathcal{T}^{\Gamma_3}_2)$, where 
$\mathcal{T}^{\Gamma_3}_0 = \{ \Gamma_3 \}$ is the partition of $\Gamma_3$ into a single block, $\mathcal{T}^{\Gamma_3}_1 = \{ \{3\}, \{4,6,7,9\}, \{8\} \}$, and finally $\mathcal{T}^{\Gamma_3}_2 = \{ \{ x\} : x \in \Gamma_3 \}$ is the singletons.

Write $\Gamma := \{ \tau_1 \in \ud t_1,\ldots,\tau_m \in \ud t_m, \mathcal{T}(\xi) = (\beta_0,\ldots,\beta_m), L_{\tau_1} = \ell_1,\ldots, L_{\tau_m} = \ell_m \}$. We can then decompose the event $\Gamma$ as
\begin{align} \label{eq:basil}
\Gamma = \{ \tau_1 \in \ud t_1, \mathcal{T}_1 = \beta_1, L_{\tau_1 } = \ell_1 \} \cap \bigcap_{j=1}^{g_1} \Gamma_j
\end{align}
where 
\begin{align*}
\Gamma_j = \{ \tau_{i,j} \in \ud t_{i,j} ~ \forall 1 \leq i \leq m_j , \mathcal{T}^{\Gamma_j} = (\beta^{\Gamma_j}_0,\ldots,\beta^{\Gamma_j}_{m_j}) \}.
\end{align*}
Now according to the symmetry lemma, conditionally on the event $ \{ \tau_1 \in \ud t_1, \mathcal{T}_1 = \beta_1, L_{\tau_1}=\ell_1 \}$, the spines in group $\Gamma_j$ behave as if under $\mathbb{Q}_{\theta,T-t_1}^{(k_1)}$. Using the inductive hypothesis (which we may use since $m_j < m$), we have
\begin{align} \label{eq:basil2}
&\mathbb{Q}_{\theta,T}^{(k)}( \Gamma_j | \tau_1 = t_1,\mathcal{T}_1 = \beta_1, L_{\tau_1} = \ell_1 ) \nonumber \\
&=  \frac{F_{T-t_1}'(e^{-\theta}) }{\mathbb{E}[Z_{T-t_1}^{(k_j)}e^{ - \theta Z_{T-t_1} }]}  \prod_{ i = 1}^{m_j} \ell_{i,j}^{(g_{i,j})} \left(\mathbb{E}[e^{-\theta Z_{(T-t_1)-(t_{i,j}-t_1)}}]\right)^{\ell_{i,j}-g_{i,j}}p_{\ell_{i,j}} F_{(T-t_1)-(t_{i,j}-t_1)}'(e^{-\theta})^{g_{i,j}-1}\ud t_{i, j}.
\end{align}
Additionally using the fact that the spines in groups $\Gamma_1,\ldots,\Gamma_{g_1}$ are independent after time $t_1$, using \eqref{eq:basil2} in \eqref{eq:basil} (and using the simplification $(T-t_1)-(t_{i,j}-t_1) = T-t_{i,j}$) we have 
\begin{align} \label{eq:basil3}
&\mathbb{Q}_{\theta,T}^{(k)}( \Gamma | \tau_1 = t_1,\mathcal{T}_1 = \beta_1, L_{\tau_1} = \ell_1 ) \nonumber\\
&= \prod_{j=1}^{g_1} \frac{F_{T-t_1}'(e^{-\theta}) }{\mathbb{E}[Z_{T-t_1}^{(k_j)}e^{ - \theta Z_{T-t_1} }]}  \prod_{ i = 1}^{m_j} \ell_{i,j}^{(g_{i,j})} \left(\mathbb{E}[e^{-\theta Z_{T-t_{i,j}}}]\right)^{\ell_{i,j}-g_{i,j}}p_{\ell_{i,j}} F_{T-t_{i,j}}'(e^{-\theta})^{g_{i,j}-1}\ud t_{i, j}.
\end{align}
Reaggregating all of the indices, \eqref{eq:basil3} simplifies to 
\begin{align} \label{eq:wroclaw}
&\mathbb{Q}^{(k)}_{\theta, T}\Big( \Gamma  \Big|  \tau_1 = t_1, \mathcal{T}_1 = \beta_1, L_{\tau_1} = \ell_1    \Big) \nonumber \\
&= \frac{ F^\prime_{T-t_1}(e^{-\theta})^{g_1} }{ \prod_{j=1}^{g_1} \mathbb{E}[Z_{T-t}^{(k_j)} e^{ - \theta Z_{T-t_1} } ] }  \prod_{ i = 2}^m \ell_i^{(g_i)}\left(\mathbb{E}[e^{-\theta Z_{(T-t_1)-(t_i-t_1)}}]\right)^{\ell_i-g_i}p_{\ell_i} F_{(T-t_1)-(t_i-t_1)}'(e^{-\theta})^{g_i-1} \mathrm{d}t_i.
\end{align}
Now according to Lemma \ref{lem:firstsplit} we have 
\begin{align} \label{eq:wroclaw2}
\mathbb{Q}^{(k)}_{\theta, T}\Big(\tau_1 \in \ud t_1; \mathcal{T}_1 = \beta_1 ; L_{\tau_1}=\ell_1 \Big) =\ell_1^{(g_1)}\left(\mathbb{E}[e^{-\theta Z_{T-t_1}}]\right)^{\ell_1-g_1}p_\ell \frac{\prod_{i=1}^{g_1} \mathbb{E}[Z^{(k_i)}_{T-t_1}e^{-\theta Z_{T-t_1}}]}{\mathbb{E}[Z^{{(k)}}_{T}e^{-\theta Z_{T}}]}\frac{F_T'(e^{-\theta})}{F_{T-t_1}'(e^{-\theta})}   r \ud t_1.
\end{align}
Multiplying \eqref{eq:wroclaw} and \eqref{eq:wroclaw2} and using the definition of $\Gamma$ we obtain
\begin{align}
& \mathbb{Q}^{(k)}_{\theta, T}\Big(\tau_1 \in \ud t_1,\ldots,\tau_m \in \ud t_m, \mathcal{T}(\xi) = (\beta_0,\ldots,\beta_m), L_{\tau_1} = \ell_1,\ldots, L_{\tau_m} = \ell_m \Big) \nonumber \\
& = \frac{F_T'(e^{-\theta}) }{\mathbb{E}[Z_T^{(k)}e^{ - \theta Z_T }]}  \prod_{ i = 1}^m \ell_i^{(g_i)}\left(\mathbb{E}[e^{-\theta Z_{T-t_i}}]\right)^{\ell_i-g_i}p_{\ell_i} F_{T-t_i}'(e^{-\theta})^{g_i-1}\ud t_{i},
\end{align}
as required.

\end{proof}

\begin{rem} \label{rem:cris}
We remark that one consequence of Lemma \ref{lem:Qmain} is that given that a split of size $g_i$ occurs at time $t_i$, the size of the birth at this time is distributed according to the probability mass function $\ell \mapsto  \ell^{(g)}\left(\mathbb{E}[e^{-\theta Z_{T-t}}]\right)^{\ell-g}p_{\ell}/\mathbb{E}[\mathbb{E}[L^{(g)}e^{-\theta Z_{T-t}}]^{L-g}]$. In particular, in the case $g=1$, we see that births off the spine are (i.e.\ births after which every spine follows the same particle) are distributed according to the distribution $\ell \mapsto  \ell \left(\mathbb{E}[e^{-\theta Z_{T-t_i}}]\right)^{\ell-1}p_{\ell}/\mathbb{E}[L\mathbb{E}[e^{-\theta Z_{T-t}}]^{L-1}]$.
\end{rem}

\begin{rem} \label{rem:cris2}
We also note it is possible by summing over $\ell_i$ to give the following projected version of Lemma \ref{lem:Qmain}:
\begin{align} \label{eq:Qmainproj}
& \mathbb{Q}^{(k)}_{\theta, T}\Big(\tau_1 \in \ud t_1,\ldots,\tau_m \in \ud t_m, \mathcal{T}(\xi) = (\beta_0,\ldots,\beta_m) \Big) \nonumber \\
& = \frac{F_T'(e^{-\theta}) }{\mathbb{E}[Z_T^{(k)}e^{ - \theta Z_T }]}  \prod_{ i = 1}^m \mathbb{E}[ L^{(g_i)} \left(\mathbb{E}[e^{-\theta Z_{T-t_i}}]\right)^{L-g_i} ] F_{T-t_i}'(e^{-\theta})^{g_i-1}\ud t_{i}.
\end{align}
\end{rem}

Finally, our results so far have given a detailed description of the spine particles under $\mathbb{Q}_{\theta,T}^{(k)}$. In our final lemma of this section, we describe the behaviour of the \emph{non-spine} particles under $\mathbb{Q}_{\theta,T}^{(k)}$.
\begin{lemma} \label{lem:nonspine}
Under $\mathbb{Q}_{\theta,T}^{(k)}$, non-spine particles behave independently of the other particles, and independently of the history of the process. Moreover, at time $t$, any non-spine particles behave like the initial ancestor under a copy of the measure $\mathbb{P}_{\theta,T-t}$ defined in \eqref{PthetaT}.

In particular, under $\mathbb{Q}_{\theta,T}^{(k)}$, at time $t$ non-spine particles undergo branching at rate
\[
r \frac{\E\left[\left(\E[e^{-\theta Z_{T-t}}]\right)^L\right]}{\E[e^{-\theta Z_{T-t}}]},
\]
and, given that they branch at time $t$, their offspring distribution is given by 
\[
\mathbb{P}_{\theta,T}(L(t)=\ell)= p_\ell \,\frac{\left(\E[e^{-\theta Z_{T-t}}]\right)^\ell}{\E\left[\left(\E[e^{-\theta Z_{T-t}}]\right)^L\right]}.
\]

\end{lemma}

\begin{proof}
Suppose $v$ is a non-spine particle alive at time $t$. Then the Radon-Nikodym derivative of $\mathbb{Q}_{\theta,T}^{(k)}$ against $\mathbb{P}$ defined in \eqref{newprob} may be written
\begin{align*}
\frac{\mathbf{1}_{A_{k, T}}\left(\prod_{i=1}^k\prod_{v\prec \xi^{(i)}_T} L_v\right) e^{-\theta Z_T}}{\mathbb{E}^{(k)}\Big[Z_T^{(k)} e^{-\theta Z_T}\Big]} = e^{ - \theta Z_T^v } \frac{\mathbf{1}_{A_{k, T}}\left(\prod_{i=1}^k\prod_{v\prec \xi^{(i)}_T} L_v\right) e^{-\theta (Z_T-Z_T^v)}}{\mathbb{E}^{(k)}\Big[Z_T^{(k)} e^{-\theta Z_T }\Big]} ,
\end{align*}
where $Z_T^v$ is the number of particles alive at time $t$ descended from $v$. Now under $\mathbb{P}$, since $v$ is not carrying any marks at time $t$, the random variables $(Z_T - Z_T^v)$ and $\mathbf{1}_{A_{k, T}}\left(\prod_{i=1}^k\prod_{v\prec \xi^{(i)}_T} L_v\right)$ are independent of the particle $v$ and of the behaviour of $v$ over the course of $[t,T]$. It follows that we can write 
\begin{align*}
\frac{\mathbf{1}_{A_{k, T}}\left(\prod_{i=1}^k\prod_{v\prec \xi^{(i)}_T} L_v\right) e^{-\theta Z_T }}{\mathbb{E}^{(k)}\Big[Z_T^{(k)} e^{-\theta Z_T}\Big]} = e^{ - \theta Z_T^v } Q_T^{(v)},
\end{align*}
where $Q_T^{(v)}$ is a random variable independent of $v$ and the descendents of $v$. By the branching property, under $\mathbb{P}$, $Z_T^v$ is distributed like $Z_{T-t}$. It follows that under $\mathbb{Q}_{\theta,T}^{(k)}$, the non-spine particle $v$ behaves independently of the other particles in the system at time $t$, and has its behaviour tilted by the exponential factor $e^{ - \theta Z_T^v}$. In short, $v$ behaves like the initial ancestor under a copy of the measure $\mathbb{P}_{\theta,T-t}$. 

In order to now describe the branching of $v$ at time $t$, we therefore need only describe the branching of the initial ancestor under a copy of $\mathbb{P}_{\theta,T}$. (Thereafter we can replace $T$ by $T-t$ as necessary.)

We now note that if $\Gamma(h,\ell)$ is the indicator function of the event that the initial ancestor dies in the time interval $[0,h)$ and has $\ell$ offspring, then using the definition \eqref{PthetaT} of the change of measure $\mathbb{P}_{\theta,T}$ we have
\begin{align} \label{eq:crr}
\mathbb{P}_{\theta,T}( \Gamma(h,\ell) ) := \frac{\mathbb{P} [ \Gamma(h,\ell) e^{ - \theta Z_T } ]  }{\mathbb{P}[e^{ - \theta Z_T} ] } 
\end{align}
For small $h$, we then have
\begin{align*}
\mathbb{P} [ \Gamma(h,\ell) e^{ - \theta Z_T } ]  &= r h p_\ell ( h + o(h)) \mathbb{P} [ e^{ - \theta Z_T } | \Gamma(h,\ell)   ]\\
&=  r h p_\ell ( h + o(h)) \mathbb{P}[e^{ - \theta Z_{T-h}}]^\ell\\
&=  r h p_\ell \mathbb{P}[e^{ - \theta Z_{T}}]^\ell  + o(h).
\end{align*}
Plugging this into \eqref{eq:crr}, we see that $\mathbb{P}_{\theta,T}( \Gamma(h,\ell) ) = r h p_\ell \mathbb{P}[e^{ - \theta Z_{T}}]^\ell / \mathbb{P}[e^{- \theta Z_T}] + o(h) $. 

Since under $\mathbb{P}_{\theta,T}$ the particles at time $t$ behave independently, and as if under an independent copy of $\mathbb{P}_{\theta,T-t}$, it follows that the time $t$ rate of splitting into $\ell$ particles  is given by $ r p_\ell \mathbb{P}[e^{ - \theta Z_{T-t}}]^\ell / \mathbb{P}[e^{- \theta Z_{T-t}}]$. Rearranging so that this quantity is proportional to a probability mass function in $\ell$, we obtain the stated branching rates and probabilities.
\end{proof}

\subsection{Subpopulation sizes}
It will be important later to understand the sizes of subpopulations coming off the spine under $\mathbb{Q}_{\theta,T}^{(k)}$.

Suppose a non-spine particle $v$ is alive at time $t$ and $Z_T^v$ represents the number of descendants of $v$ alive at time $T$, then, by Lemma \ref{lem:nonspine} and Definition \eqref{PthetaT}, we have 
\begin{equation}\label{nonspinepop}
\mathbb{Q}_{\theta,T}^{(k)} [ e^{ - \phi Z_T^v } | \mathcal{F}_t] = \mathbb{P}_{\theta,T-t}[e^{ - \phi Z_{T-t}}]=
\frac{ \mathbb{E} [ e^{ - (\phi + \theta) Z_{T-t} } ]}{ \mathbb{E} [ e^{ - \theta Z_{T-t}} ] } .
\end{equation}

Similarly, if a particle $v$ alive at time $t$ is carrying $j$ spines and  $Z_T^v$ represents the number of descendants of $v$ alive at time $T$, then recalling Lemma \ref{branlem} and \eqref{newprob1}, we find
\begin{align*}
\mathbb{Q}_{\theta,T}^{(k)} [ e^{ - \phi Z_T^v } | \mathcal{F}_t^{(k)}] 
=\mathbb{Q}_{\theta,T-t}^{(j)} [ e^{ - \phi Z_{T-t} } ]
= \frac{ \mathbb{E} [ Z_{T-t}^{(j)} \,e^{ - (\phi + \theta) Z_{T-t} } ]}{ \mathbb{E} [ Z_{T-t}^{(j)} \,e^{ - \theta Z_{T-t}} ] } .
\end{align*}
In particular, the Laplace transform for number of births coming off a single spine branch ($k=1$) started at time $t$ (plus the spine itself) is given by
\begin{equation}\label{spinepop}
\frac{ \mathbb{E} [ Z_{T-t} \,e^{ - (\phi + \theta) Z_{T-t} } ]}{ \mathbb{E} [ Z_{T-t} \,e^{ - \theta Z_{T-t}} ] }.
\end{equation}
Moreover, since we observed that the rate of births off any spine particle and the corresponding offspring distribution are \emph{independent} of the actual number of spines following it,  the above expression remains unchanged for the number of descendants coming off any single spine branch  over time period $[T-t,T]$.

\subsection{Lineage and population decompositions for the ancestral tree} \label{sec:lineage}
We now define the subpopulations off spines. We begin by decomposing the spine tree into lineages. Recall that the split times $\tau_1< \ldots< \tau_m$ associated with the $k$ spines are the times at which a particle carrying some spines dies and these spines do not all follow the same child. With these in mind, we now define the \emph{separation times} of the spines. Define $\varsigma_1 := 0$, and for $1 \leq j \leq k-1$ define the $(j+1)^{\text{th}}$ separation time $\varsigma_{j+1}$ to be the first time that the $(j+1)^{\text{th}}$ spine is separate from spines  $1,\ldots,j$, i.e.
\begin{align} \label{eq:vardef}
\varsigma_{j+1} := \inf \{ t \geq 0: \xi_t^{(j+1)} \neq \xi_t^{(i)} ~ \text{for all $1 \leq i \leq j$} \}.
\end{align}
We note that for all $j \geq 1$, $\varsigma_{j+1}$ is equal to some $\tau_i$. In fact, if $g_i$ is the size of the split at time $\tau_i$, then there are exactly $g_i-1$ elements $j \in \{1,\ldots,k\}$ for which $\varsigma_j = \tau_i$. 
\begin{figure}[h!]\label{fig:sample}
\centering
\begin{tikzpicture}[xscale=1,yscale=1]
\draw[gray] (0,-3.7) -- (0,3.3);
\draw[gray] (10,-3.7) -- (10,3.3);
\node at (0,-4)   (a) {Time $0$};
\node at (10,-4)   (a) {Time $T$};

\draw[very thick] (0,0) -- (3,0);
\draw[very thick] (3,-2) -- (3,2);
\draw[very thick] (3,-2) -- (5.5,-2);
\draw[very thick] (5.5,-3) -- (5.5,-1.5);
\draw[very thick] (5.5,-3) -- (10,-3);
\node at (10.3,-3)   (a) {$2$};
\draw[very thick] (5.5,-1.5) -- (10,-1.5);
\node at (10.3,-1.5)   (a) {$4$};
\draw[very thick] (3,-0.5) -- (10,-0.5);
\node at (10.3,-0.5)   (a) {$1$};
\draw[very thick] (3,2) -- (6,2);
\draw[very thick] (6,1) -- (6,2.5);
\draw[very thick] (6,1) -- (10,1);
\node at (10.3,1)   (a) {$5$};
\draw[very thick] (6,2.5) -- (10,2.5);
\node at (10.3,2.5)   (a) {$3$};

\draw[thick, dotted] (3,-2) -- (3, -3.7);
\node at (3,-4)   (a) {$\tau_1$};
\draw[thick, dotted] (5.5,-2) -- (5.5, -3.7);
\node at (5.5,-4)   (a) {$\tau_2$};
\draw[thick, dotted] (6,1) -- (6, -3.7);
\node at (6,-4)   (a) {$\tau_3$};

\draw[very thick, OliveGreen] (0,0.1) -- (3.1,0.1);
\draw[very thick, OliveGreen] (3.1,-0.4) -- (3.1,0.1);
\draw[very thick, OliveGreen] (3.1,-0.4) -- (10,-0.4);

\draw[very thick, OliveGreen] (3,-1.9) -- (5.6,-1.9);
\draw[very thick, OliveGreen] (5.6,-2.9) -- (5.6,-1.9);
\draw[very thick, OliveGreen] (5.6,-2.9) -- (10,-2.9);

\draw[very thick, OliveGreen] (3,2.1) -- (5.9,2.1);
\draw[very thick, OliveGreen] (5.9,2.1) -- (5.9, 2.6);
\draw[very thick, OliveGreen] (5.9,2.6) -- (10,2.6);

\draw[very thick, OliveGreen] (5.5,-1.4) -- (10,-1.4);

\draw[very thick, OliveGreen] (6,1.1) -- (10,1.1);
\end{tikzpicture}
\caption{
A decomposition of the spine tree into lineages. Using \eqref{eq:vardef}, we see that here the separation times are given by $\varsigma_1 = 0, \varsigma_2 = \tau_1, \varsigma_3 = \tau_1, \varsigma_4 = \tau_2, \varsigma_5 = \tau_3)$.}

\end{figure}
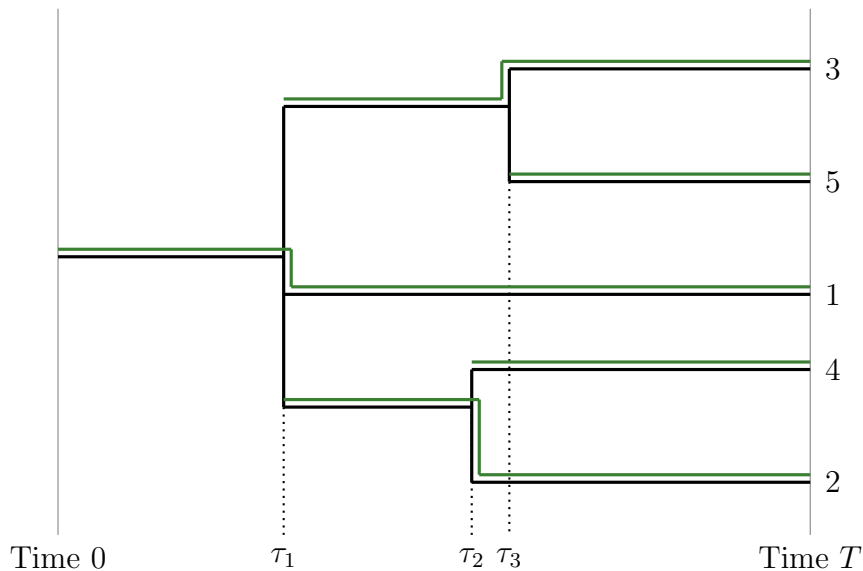

With the separation times at hand we can decompose the entire ancestral tree of the spine into $j$ lineages, where the length of lineage $j$ is given by $T - \varsigma_j$. Of course then the total length of the tree is given by $kT - \sum_{j=1}^k \varsigma_j$. This lineage decomposition affords us a decomposition of the entire population at time $T$. Giving first an informal definition, let us define
\begin{align*}
\mathcal{Z}_{T,j} := \{ u \in \mathcal{Z}_T : \text{$u$ is a descendant of a birth off the spine of the $j^{\text{th}}$ lineage} \}.
\end{align*}
Now let us a give a more precise definition. We say a particle carrying spines is \emph{non-splitting} if its death is a birth off the spine, i.e. at the death of this particle, all of the spines it was carrying follow the same child of this particle. Now define, for $1 \leq j \leq k-1$, $\mathcal{Z}_{T,j+1}$ to be the set of particles at time $T$ (excluding $\xi_T^{(j+1)}$) who are descended from any non-splitting ancestor of $\xi_T^{(j+1)}$ who is not an ancestor of any $\xi_T^{(1)},\ldots,\xi_T^{(j)}$, i.e.
\begin{align*}
\mathcal{Z}_{T,j+1} := \{ u \in \mathcal{Z}_T - \{ \xi_T^{(j+1)}\} : \exists t, \exists v \in \mathcal{Z}_t \text{ non-splitting} : v \prec u, v\prec \xi_T^{(j+1)}, v \nprec \xi_T^{(i)} ~\forall 1 \leq i \leq j  \}.
\end{align*}
We now define $\tilde{\mathcal{Z}}_{T,1},\ldots,\tilde{\mathcal{Z}}_{T,m}$ by letting
\begin{align*}
\tilde{\mathcal{Z}}_{T,j} := \{ u \in \mathcal{Z}_T : \text{$u$ is a descendent of a non-spine particle born at splitting event $\tau_j$} \}.
\end{align*} 
We now that every particle alive at time $T$ is either a spine, a descendent of a birth off the spine, or a descendent of a non-spine particle born at a splitting event. This creates a decomposition of the population at time $T$ via
\begin{align*}
\mathcal{Z}_T := \{\xi_T^{(1)},\ldots,\xi_T^{(k)}\} \cup \bigcup_{j=1}^k \mathcal{Z}_{T,j} \cup \bigcup_{j=1}^m \tilde{\mathcal{Z}}_{T,j}.
\end{align*}
Accordingly, if $Z_{T,j} := \# \mathcal{Z}_{T,j}$ and $\tilde{Z}_{T,j} := \# \tilde{\mathcal{Z}}_{T,j}$ denote cardinalities then we have
\begin{align*}
Z_T = k + \sum_{j=1}^k Z_{T,j} + \sum_{ j = 1}^m \tilde{Z}_{T,j}.
\end{align*}
Further, conditional on the separation times, the Laplace transforms for each subpopulation appearing in this decomposition can readily be written down using \eqref{nonspinepop} and \eqref{spinepop}.

\section{Uniform sampling for  critical Galton-Watson trees in the regularly varying regime} \label{sec:gen}
Throughout this section, whenever $A(\lambda)$ and $B(\lambda)$ are functions depending on a real or integer-valued parameter, we use the notation
\begin{align*}
A(\lambda) \sim B(\lambda) \qquad\textrm{as} \quad \lambda \to \infty
\end{align*}
to denote $\lim_{\lambda \to \infty} B(\lambda)/A(\lambda) = 1$. 
Let us consider a random variable $L$ taking values in $\mathbb{Z}_+=\{0,1,\ldots\}$ with the same distribution as $L_{\emptyset}$, and recall that $Z$ denotes a critical Galton-Watson process whose offspring distribution is given by $L$.   In other words,  
\[
p_n=\mathbb{P}(L=n)\quad \textrm{for } n\ge 0,\qquad  \mathbb{E}[L]=\sum_{j\ge 0}jp_j=1\qquad \textrm{and}\qquad f(s)=\mathbb{E}[s^L].
\]
Recall that we are assuming that  \eqref{eq:hyp1} holds, that is  $p_0>0$ and that for $\alpha\in(1,2]$, 
\[
f(s)=s+(1-s)^\alpha\ell\left(\frac{1}{1-s}\right),
\]
where $\ell$ is a slowly varying function at $\infty$. The asymptotic behaviour of $f$ provides   enough information about the behaviour of the offspring probabilities $p_k$, for $k$ large.

Indeed, first let us observe from using the geometric sum and interchanging the order of summation 
\[
1-f(s)=\sum_{j\ge 1}p_j(1-s^{j})=(1-s)\sum_{j\ge 1}p_j\sum_{k=0}^{j-1}s^{k} =(1-s)\sum_{k\ge 0}\overline{p}_{k}s^{k},
\]
where $\overline{p}_{j}:=\sum_{i>j}p_i$. 

Using a similar argument to obtain the  final equality below we have 
\[
\begin{split}
\frac{f(s)-s}{1-s} &= 1 - \frac{1-f(s)}{1-s} = 1-\sum_{k\ge 0}\overline{p}_{k}s^k= (1-s)\sum_{k\ge0}\overline{\overline{p}}_{k}s^k,
\end{split}
\]
where $\overline{\overline{p}}_k:=\sum_{j>k}\overline{p}_j$. 

In other words, we have
\begin{align*}
(1-s)^{\alpha - 2} \ell \left( \frac{1}{1-s} \right)  = \frac{f(s)-s}{(1-s)^2} = \sum_{k\ge0}\overline{\overline{p}}_{k}s^k.
\end{align*}
Hence Karamata's Tauberian Theorem for power series (see for instance Corollary 1.7.3 in Bingham, Goldie and Teugels \cite{BGT87}) allows us to deduce in the case $\alpha \in (1,2)$ that 
\begin{align} \label{eq:orange}
\overline{\overline{p}}_{k}\sim \frac{1}{\Gamma(2-\alpha)}k^{1-\alpha}\ell(k), \qquad\textrm{as} \quad k\to \infty,
\end{align}
and in the case $\alpha = 2$ that
\[
\sum_{ j = 0}^k  \overline{\overline{p}}_j \sim  \ell(k) \qquad\textrm{as} \quad k\to \infty.
\]
Now since the sequence $\{\overline{\overline{p}}_{k}\}_{k\ge 0}$ is monotone, from the Monotone Density Theorem (see for instance Theorem 1.7.2 in \cite{BGT87}) we obtain in the case $\alpha = 2$ that 
\begin{align} \label{eq:apple}
\overline{\overline{p}}_k \sim k^{-1} \ell(k) \qquad\textrm{as} \quad k\to \infty.
\end{align}
Consolidating \eqref{eq:orange} and \eqref{eq:apple} we have 
\begin{align} \label{eq:pear}
\overline{\overline{p}}_k \sim \tilde{Q}_\alpha k^{1-\alpha} \ell(k) \qquad\textrm{as} \quad k\to \infty,
\end{align}
where $\tilde{Q}_\alpha := \frac{1}{\Gamma(2-\alpha)}$ in the case $\alpha \in (1,2)$, and $Q_2 := 1$. 

Since the sequences $\{\overline{p}_{k}\}_{k\ge 0}$ is monotone, it follows from the Monotone density Theorem (see for instance Theorem 1.7.2 in \cite{BGT87}) we obtain in the case $\alpha \in (1,2]$ that 
\[
\overline{p}_{k}\sim Q_\alpha k^{-\alpha}\ell(k), \qquad\textrm{as} \quad k\to \infty.
\]
where $Q_\alpha = (\alpha-1)\tilde{Q}_\alpha$, so that $Q_\alpha = \frac{\alpha-1}{\Gamma(2-\alpha)}$ when $\alpha \in (1,2)$ and $Q_\alpha = 1$ when $\alpha=2$.

Our assumption \eqref{eq:hyp1} also implies that the probability of survival is regularly varying at infinity with a precise index that we  deduce below. Recall that  the moment generating function  of $Z_t$, for $t>0$, is denoted by $F_t(s)$. Thus from the backward Kolmogorov equation it satisfies
\[
\int_s^{F_t(s)}\frac{\ud y}{f(y)-y}=r t.
\]
When $s=0$, it is clear that 
\[
F(t):=F_t(0)=\mathbb{P}_1(Z_t=0)=\mathbb{P}_1(T_0\le t),
\] 
where $T_0=\inf\{s: Z_{s}=0\}$ and therefore
\[
\int_0^{F(t)}\frac{\ud y}{f(y)-y}=r t.
\]
We also introduce
\[
V(x):=\int_0^{1-1/x} \frac{\ud y}{f(y)-y}, \qquad \textrm{for}\quad x\ge1,
\]
which is increasing and concave, and denote by $R$ its inverse which is convex and increasing (see Lemma 2.1 in Pakes \cite{Pakes10}). Under our assumptions, the function $V$ can be rewritten as follows
\[
V(x)=\int_1^{x} \frac{\ud u}{u^{2-\alpha}\ell (u)}.
\]
Following the same ideas as in the proof of Proposition 1.5.8 in \cite{BGT87}, we see that
\[
(\alpha-1) V(x)\sim \frac{x^{\alpha-1}}{\ell (x)}=x^{\alpha-1}\ell_1(x), \qquad \textrm{as} \quad x\to \infty,
\]
where 
\begin{equation} \label{eq:ell1}
\ell_1(x) :=\frac{1}{ \ell (x)}
\end{equation}
is also a slowly varying function at $\infty$.  From Theorem 1.5.12 in \cite{BGT87}, we have that $R$ is regularly varying at infinity with index $1/(\alpha-1)$, i.e.
\[
R(x)\sim x^{\frac{1}{\alpha-1}}\ell_2(x), \qquad \textrm{as} \quad x\to \infty,
\]
where $\ell_2$ is another slowly varying function at $\infty$. Thus under  assumption \eqref{eq:hyp1}, we deduce  that  the survival probability $\overline{F}(t):=\mathbb{P}(Z_t>0)$ satisfies the following identity
\[
(\alpha-1)V\left(\frac{1}{\overline{F}(t)}\right)=r(\alpha-1) t=V(R((\alpha-1)rt)),
\]
implying that 

\begin{align} \label{eq:tchaik}
\overline{F}(t)\sim\frac{1}{R((\alpha-1)rt)}=c_{\alpha,r}\frac{t^{-1/(\alpha-1)}}{\ell_2(t)}, \qquad \textrm{as} \quad t\to\infty, 
\end{align}
where $c_{\alpha,r}=((\alpha-1)r)^{-1/(\alpha-1)}$ and moreover
\begin{equation}\label{poplarge}
\mathbb{E}\left[e^{-\theta{\overline{F}(t)Z_t}}\Bigg| Z_t>0\right]\xrightarrow[t\to \infty]{}1-\Big(1+\theta^{1-\alpha}\Big)^{-1/(\alpha-1)}.
\end{equation}
Let us denote by $W_{\alpha-1}$ for the r.v. whose Laplace transform is such that
\[
\mathbb{E}\left[e^{-\theta W_{\alpha-1}}\right]=1-\Big(1+\theta^{1-\alpha}\Big)^{-1/(\alpha-1)}.
\]
We note that when $\alpha=2$, the limiting random variable $W_1$ is an  exponential random variable with parameter equals 1 regardless if $\sigma^2$ is finite or infinite.

There is an interesting relationship between $\ell_1$ and $\ell_2$ which will be very useful in the sequel. To describe this relationship, we 
 introduce De Bruijn's conjugate of $\ell^*(x):=(\ell_1(x))^{\frac{1}{\alpha-1}}$, that is to say  a slowly varying function $\ell^\sharp$ such that 
\begin{equation}\label{inversel}
\ell^*(x)\ell^\sharp(x\ell^*(x))\to 1\qquad \textrm{and}\qquad\ell^\sharp(x)\ell^*(x\ell^\sharp(x))\to 1 \qquad \textrm{as} \quad x\to \infty.
\end{equation}
Moreover $\ell^{\sharp\sharp}(x)\sim \ell^*(x)$ as $x\to \infty$ (see Theorem 1.5.13 in \cite{BGT87}). According to Proposition 1.5.15 in \cite{BGT87}, we have 
\begin{equation} \label{inversel2}
\ell_2(x)\sim \ell^\sharp(x^{1/(\alpha-1)})\qquad \textrm{as}\quad x\to \infty.
\end{equation}
For instance if $\ell(x)\to q$, as $x\to \infty$, we observe that $\ell^*(x)$ goes to $q^{-1/(\alpha-1)}$ as $x\to \infty$; and moreover $\ell_2(x)$ goes to $q^{1/(\alpha-1)}$ as $x\to \infty$.

\subsection{Properties of $W_{\alpha-1}$}
When $\alpha=2$, it is not so difficult to verify that $W_1$ possesses all positive moments. More precisely,   for  $k\ge 1$, we have
\begin{equation}\label{explicitk}
\mathbb{E}\Big[W_1^k\Big]=k! \qquad\textrm{and}\qquad \mathbb{E}\Big[W_1^k e^{-\theta W_1}\Big]=\frac{ k!}{(1+\theta)^{k+1}}.
\end{equation}
The compensated moments of $W_{\alpha-1}$ in the setting $\alpha \in (1,2)$ are more involved. To compute these, write
\[
f_\alpha(\theta):= \mathbb{E}[e^{-\theta W_{\alpha-1}}]= g_\alpha(h_\alpha(\theta))
\]
where $g_\alpha(\theta) =  1 - (1 + \theta)^{-\frac{1}{\alpha-1}}$ and $h_\alpha(\theta) = \theta^{1-\alpha}$.
We can compute the derivatives of $f_\alpha(\theta)$ using Fa\`a di Bruno formula's (equation (2) of \cite{JP22}), which states that given sufficiently differentiable functions $f$ and $g$ we have 
\begin{align*}
\frac{\mathrm{d}^k}{\mathrm{d}\theta^k} g(h(\theta)) = \sum_{ \pi \in \mathcal{P}_k } g^{(\# \pi)}(h(\theta)) \prod_{\Gamma \in \pi} h^{(\# \Gamma)}(\theta),
\end{align*}
where the sum is taken over all set partitions $\pi$ of $\{1,\ldots,k\}$, and given a partition $\pi$, the product is taken over all blocks $\Gamma$ of $\pi$. 

Indeed, note the $j^{\text{th}}$ derivatives of $g_\alpha$ and $h_\alpha$ are given by 
\begin{align*}
g_\alpha^{(j)}(\theta) = (-1)^{j-1} (1+\theta)^{-\frac{1}{\alpha-1} - j } \prod_{i=1}^j \left( \frac{1}{\alpha-1} + i-1\right) ,
\end{align*}
and
\begin{align*}
h_\alpha^{(j)}(\theta) = (-1)^{j}\theta^{1-\alpha-j} \prod_{i=1}^j (\alpha+i-2).
\end{align*}
It follows from Fa\`a di Bruno's formula that
\begin{align*}
\frac{\mathrm{d}^k}{\mathrm{d}\theta^k} f_\alpha(\theta) = \sum_{ \pi \in \mathcal{P}_k } (-1)^{\# \pi -1} (1 + \theta^{1-\alpha})^{ - \frac{1}{\alpha-1} - \# \pi } \prod_{i=1}^{\# \pi} \left( \frac{1}{\alpha-1}+i-1 \right)  \prod_{\Gamma \in \pi} \left\{ (-1)^{\# \Gamma} \theta^{1-\alpha- \# \Gamma } \prod_{i=1}^{\# \Gamma} (\alpha+i-2) \right\},
\end{align*}
which, using $\sum_{\Gamma \in \pi} \# \Gamma = k$, simplifies to 
\begin{align} \label{eq:gsp}
\frac{\mathrm{d}^k}{\mathrm{d}\theta^k} f_\alpha(\theta) = (-1)^{k-1} \theta^{-k} \sum_{ \pi \in \mathcal{P}_k } (-1)^{\# \pi} (1 + \theta^{1-\alpha})^{ - \frac{1}{\alpha-1} - \# \pi } \theta^{(1-\alpha)\# \pi} \prod_{i=1}^{\# \pi} \left( \frac{1}{\alpha-1}+i-1 \right) \prod_{\Gamma \in \pi} \prod_{i=1}^{\# \Gamma} (\alpha+i-2).
\end{align}
It follows that the first compensated moment is
\begin{equation} \label{eq:firstmoment}
\mathbb{E}\left[W_{\alpha-1}e^{-\theta W_{\alpha-1}}\right]=-f^\prime_\alpha(\theta)=\frac{1}{(1+\theta^{\alpha-1})^{1+1/(\alpha-1)}},
\end{equation}
and in particular, $\mathbb{E}[W_{\alpha-1}]=1$.
For the second compensated moment, we get
\[
\mathbb{E}\Big[W^2_{\alpha-1}e^{-\theta W_{\alpha-1}}\Big]=f^{\prime\prime}_\alpha(\theta)=\frac{\alpha}{(1+\theta^{\alpha-1})^{2+1/(\alpha-1)}\theta^{2-\alpha}}.
\]
Finally, the third compensated moment satisfies 
\[
\begin{split}
\mathbb{E}\Big[W^3_{\alpha-1}e^{-\theta W_{\alpha-1}}\Big]&=-f^{(3)}_\alpha(\theta)=\frac{\alpha}{(1+\theta^{\alpha-1})^{2+1/(\alpha-1)}}\left(\frac{2\alpha-1}{(1+\theta^{\alpha-1})\theta^{2(2-\alpha)}}+\frac{2-\alpha}{\theta^{3-\alpha}}\right)\\
&=\frac{\alpha}{\theta^2(1+\theta^{\alpha-1})^{3+1/(\alpha-1)}}\left((\alpha+1)\theta^{2(\alpha-1)}+(2-\alpha)\theta^{\alpha-1}\right).
\end{split}
\]
We note at this stage that $\mathbb{E}[W_{\alpha-1}^2]$ is infinite whenever $\alpha \in (1,2)$, and hence every higher moment $\mathbb{E}[W_{\alpha-1}^k]$ is also infinite for $k \geq 2$. 
We can, however, study the asymptotics of $\mathbb{E}[W_{\alpha-1}^ke^{-\theta W_{\alpha-1}}]$ as $\theta \to 0$. Indeed, from \eqref{eq:gsp}, we see that these asymptotics concentrate on the partition minimising the power $(1-\alpha)\# \pi$, i.e. the partition of $\{1,\ldots,k\}$ into $k$ singletons. In particular, one can show that 
\begin{equation}\label{kmomW}
\mathbb{E}\Big[W^k_{\alpha-1}e^{-\theta W_{\alpha-1}}\Big]\sim \alpha(k-1-\alpha)\cdots(2-\alpha)\theta^{\alpha-k}=\alpha\frac{\Gamma(k-\alpha)}{\Gamma(2-\alpha)}\theta^{\alpha-k}\qquad \textrm{as}\quad \theta\to 0.
\end{equation}
\subsection{Critical Galton-Watson processes in the regularly varying regime.}
In the sequel we will require an understanding of the factorial compensation moments of our Galton-Watson processes at large times. To begin computing these in this section, we start by noting from \eqref{eq:tchaik} to obtain the second line below, and \eqref{poplarge} to obtain the third, for $\rho \in [0,1)$ we have 
\begin{align} \label{eq:plantain}
1 - \mathbb{E}\left[e^{-\theta \overline{F}(T)Z_{T(1-\rho)}}\right]&= \overline{F}(T(1-\rho)) \mathbb{E}\left[ 1 - e^{-\theta \frac{\overline{F}(T)}{\overline{F}(T(1-\rho)} \overline{F}(T(1-\rho)) Z_{T(1-\rho)}} \Big| Z_{T(1-\rho)} > 0 \right] \nonumber \\
&\sim  (1-\rho)^{-\frac{1}{\alpha-1}} \overline{F}(T) \mathbb{E}\left[ 1 - e^{-\theta (1-\rho)^{\frac{1}{\alpha-1}} \overline{F}(T(1-\rho)) Z_{T(1-\rho)}} \Big| Z_{T(1-\rho)} > 0 \right] \nonumber \\
&\sim \frac{\overline{F}(T)}{ (1 - \rho + \theta^{-(\alpha-1)} )^{\frac{1}{\alpha-1}} }
\end{align}
as $T \to \infty$. We remark that in each case above, the convergence in question is a consequence of the monotone convergence theorem.

Recall that $n^{{(k)}}=n(n-1)\cdots(n-k+1)$ for $n\ge k$. Then for $k \geq 1$ and $\rho \in [0,1)$, as a consequence of \eqref{poplarge} we have 

\begin{equation}\label{kmomentsrho}
\begin{split}
\mathbb{E}\left[Z_{T(1-\rho)}^{{(k)}}e^{-\theta \overline{F}(T)Z_{T(1-\rho)}}\right]\sim
\overline{F}(T(1-\rho))^{-(k-1)} \mathbb{E}\left[W_{\alpha-1}^k e^{-\theta (1-\rho)^{1/(\alpha-1)}W_{\alpha-1}}\right], 
\end{split}
\end{equation}
as $T$ goes to infinity.

When $\alpha=2$, the previous asymptotic is much simpler to write.  Indeed, for $T$ large enough, we have 
\begin{equation*} 
\mathbb{E}\left[Z^{(k)}_{T(1-\rho)}e^{-\theta \overline{F}(T)Z_{T(1-\rho)}}\right]\sim r^{k-1} (1-\rho)^{k-1}T^{k-1} \ell^{k-1}_3(T)\frac{ k!}{(1+\theta(1-\rho))^{k+1}},
\end{equation*}
and
\[
\mathbb{E}\left[Z^{{(k)}}_{(1-p)T}e^{-\theta \overline{F}(T) Z_{(1-p)T}} \Bigg| Z_{(1-p)T}\ge k\right]\sim r^{k} (1-\rho)^{k}T^{k} \ell^{k}_3(T)\frac{ k!}{(1+\theta(1-p))^{k+1}},
\]
where  $\ell_3(T)=\frac{1}{\ell_2(T)}$.
Recalling the definition \eqref{newprob} of $\mathbb{Q}^{(k)}_{\theta, T}$ we study the asymptotics of $Z_T$ under the rescaling $\theta \to \theta \overline{F}(T)$. From the asymptotic in \eqref{kmomentsrho} with $\rho=0$, we have
\begin{equation}\label{qklaplace}
\begin{split}
\mathbb{Q}^{(k)}_{\theta \overline{F}(T), T}\Big[e^{ \sam{-}\varphi \overline{F}(T)Z_T}\Big]&=\frac{\mathbb{E}\Big[Z^{{(k)}}_Te^{\sam{-}(\theta+\varphi )\overline{F}(T)Z_T}\Big]}{\mathbb{E}\Big[Z^{{(k)}}_T e^{\sam{-}\theta \overline{F}(T)Z_T}\Big]}\xrightarrow[T\to\infty]{}\frac{\mathbb{E}\left[W_{\alpha-1}^k e^{-(\theta +\varphi)W_{\alpha-1}}\right]}{\mathbb{E}\left[W_{\alpha-1}^k e^{-\theta W_{\alpha-1}}\right]}.
\end{split}
\end{equation}
In particular, the population off a single spine satisfies
\[
\begin{split}
\mathbb{Q}^{(1)}_{\theta  \overline{F}(T), T}\left[e^{-\varphi \overline{F}(T) Z_T}\right]&\xrightarrow[T\to\infty]{}\frac{\mathbb{E}[W_{\alpha-1} e^{-(\theta+\varphi) W_{\alpha-1}}]}{\mathbb{E}[W_{\alpha-1} e^{-\theta W_{\alpha-1}}]}\\
&=\left(\frac{1+\theta^{\alpha-1}}{1+(\theta+\varphi)^{\alpha-1}}\right)^{1+1/(\alpha-1)}.
\end{split}
\]
Moreover subpopulations along spine branch of length $(1-\rho)T$ with $\rho\in[0,1)$ satisfy
\[
\begin{split}
\mathbb{Q}^{(1)}_{\theta F(T), T}\left[e^{-\varphi \overline{F}(T) Z_{(1-\rho)T}}\right]&\xrightarrow[T\to \infty]{}\left(\frac{1+\theta^{\alpha-1}(1-\rho)}{1+(\theta+\varphi)^{\alpha-1}(1-\rho)}\right)^{1+1/(\alpha-1)}.\\
\end{split}
\]
When $\alpha=2$ we have an explicit expression,  for $k\ge 1$, that is 
\[
\begin{split}
\mathbb{Q}^{(k)}_{\theta \overline{F}(T), T}\Big[e^{\sam{-} \varphi \overline{F}(T)Z_T}\Big]&\xrightarrow[T\to\infty]{}\frac{\mathbb{E}\Big[W_1^k e^{\sam{-} (\theta+\varphi )W_1}\Big]}{\mathbb{E}\Big[W_1^k e^{\sam{-}\theta W_1}\Big]}=\left(\frac{1 +\theta}{1+\theta+\varphi}\right)^{k+1},
\end{split}
\]
where the right-hand side of the previous asymptotic is nothing but the Laplace exponent of a Gamma random variable with parameters $k+1$ and $1+\theta$, here denoted by $\Gamma_{k+1, 1+\theta}$.
In other words when $\alpha=2$, $\overline{F}(T)Z_T$, under $\mathbb{Q}^{(k)}_{\theta\overline{F}(T), T}$, tends to   $\Gamma_{k+1, 1+\theta}$, under a new probability measure that we denote as $\mathbb{Q}^{(k)}_{\theta, \infty}$. That is to say, the sum of $k+1$ independent $\mathbf{e}_{1+\theta}$, exponential random variables with parameter $1+\theta$. In particular, the mass coming off a single  spine branch is distributed as  $\Gamma_{2, 1+\theta}$.

\subsection{ Large offspring.}

\begin{lemma} \label{lem:palm}
If $L$ is an offspring variable satisfying \textbf{(H1)}. Then as $\lambda \to 0$ we have
\begin{equation} \label{tomate1}
\mathbb{E}[Le^{-\lambda L} ] \sim 1.
\end{equation} 
More generally, for either $g = 2,\alpha=2$, or for any $g \geq 2$ and $\alpha \in (1,2]$, we have 
\begin{equation}\label{tomate2}
\mathbb{E}[L^{(g)} e^{-\lambda L}] \sim \frac{\Gamma(g-\alpha)}{\Gamma(-\alpha)} \lambda^{\alpha-g} \ell\left( \frac{1}{\lambda} \right),
\end{equation}
as $\lambda \to 0$.
\end{lemma}
\begin{proof}
The equation \eqref{tomate1} is a consequence of the fact that $\mathbb{E}[L] =1$ and the monotone convergence theorem.

As for \eqref{tomate2}, the assumption \textbf{(H1)} states that $f(s) = \mathbb{E}[s^L]$ takes the form $f(s) = s + (1-s)^{\alpha} \ell \left( \frac{1}{1-s} \right)$ as $s \to 1$, where $\ell(\cdot)$ is slowly varying at $\infty$. We note that every derivative of $f$ is monotone decreasing, and accordingly, by the monotone density theorem, whenever either $g = 2,\alpha=2$ or $g \geq 2, \alpha \in (1,2)$ we have 
\begin{align*}
f^{(g)}(s) := \mathbb{E}[L^{(g)}s^{L-g}] \sim (-1)^g \alpha(\alpha-1)\ldots(\alpha-g+1)(1-s)^{\alpha-g}\ell \left( \frac{1}{1-s} \right),
\end{align*}
as $s \to 1$. Now set $s = e^{-\lambda}$ and use the definition of the Gamma function.
\end{proof}

The previous results implies the following useful Lemma.
\begin{lemma} \label{lem:cebolla}
For any $\alpha \in (1,2]$, as $T \to \infty$ we have
\begin{align} \label{cebolla1}
\mathbb{E}&\left[L \left(\mathbb{E}\left[e^{-\theta \overline{F}(T)Z_{T(1-\rho)}}\right]\right)^{L-1}\right] \sim 1 ,
\end{align}
and for either $g = 2,\alpha=2$, or for any $g \geq 2$ and $\alpha \in (1,2)$,
we have 
\begin{equation}\label{cebolla}
\begin{split}
\mathbb{E}&\left[L^{(g)}\left(\mathbb{E}\left[e^{-\theta \overline{F}(T)Z_{T(1-\rho)}}\right]\right)^{L-g}\right]\sim\frac{\Gamma(g-\alpha)}{\Gamma(-\alpha)}\left( \frac{\overline{F}(T)}{ (1 - \rho + \theta^{-(\alpha-1)} )^{\frac{1}{\alpha-1}} } \right)^{\alpha-g}\ell_2(T)^{(\alpha-1)},
\end{split}
\end{equation}
as $T \to \infty$.
\end{lemma}
\begin{proof} 
Recall that by \eqref{eq:plantain} $1- \mathbb{E}\left[e^{-\theta \overline{F}(T)Z_{T(1-\rho)}}\right] \sim \lambda_T$ where
\begin{align*}
\lambda_T =\frac{\overline{F}(T)}{ (1 - \rho + \theta^{-(\alpha-1)} )^{\frac{1}{\alpha-1}} } \to 0
\end{align*}
as $T \to \infty$.

The first part \eqref{cebolla1} now follows from \eqref{tomate1}.

As for the latter equation, \eqref{cebolla}, we note first that by \eqref{tomate2} we have  
\begin{align*}
\mathbb{E}&\left[L^{(g)}\left(\mathbb{E}\left[e^{-\theta \overline{F}(T)Z_{T(1-\rho)}}\right]\right)^{L-1}\right]\sim \frac{\Gamma(g-\alpha)}{\Gamma(-\alpha)} \lambda_T^{\alpha-g} \ell\left( \frac{1}{\lambda_T} \right).
\end{align*} 
A brief calculation using \eqref{eq:ell1}, \eqref{eq:tchaik}, \eqref{inversel} and \eqref{inversel2} tells us that $\ell(1/\lambda_T) \sim \ell_2(T)^{(\alpha-1)}$ as $T \to \infty$. The result follows. 
\end{proof}

The next result tells us that the asymptotic law of the offspring size at a size $g$ splitting event under $\mathbb{Q}_{\theta \overline{F}(T),T}^{(k)}$. 

\begin{lemma} \label{lem:splitsize}
Let $\alpha \in (1,2)$ and $g \geq 2$. Let $J$ denote the number of offspring at a splitting event of size $g$ at time $\rho T$. Then 
\begin{align*}
&\mathbb{Q}_{\theta \overline{F}(T),T}^{(k)} [ e^{ - \varphi \overline{F}(T) J  } | \mathcal{G}_T^{(k)}, \text{split of size $g$ at time $\rho T$}] \sim \left( 1 + \varphi (1 - \rho + \theta^{1-\alpha})^{\frac{1}{\alpha-1}} \right)^{-(g-\alpha)}.
\end{align*}
\end{lemma}
\begin{proof}
According to remark \ref{rem:cris}, given that a splitting event of size $g$ occurs at time $t$, under $\mathbb{Q}_{\theta,T}^{(k)}$, the conditional distribution of the size of offspring event is given by $\ell^{(g)} \mathbb{E}[ e^{ - \theta Z_{T-t} } ]^{\ell-g}p_\ell / \mathbb{E}[ L^{(g)} \mathbb{E}[ e^{ - \theta Z_{T-t} } ]^{L-g} ]$. 

In particular, setting $t = \rho T$ we have
\begin{align*}
\mathbb{Q}_{\theta \overline{F}(T),T}^{(k)} [ e^{ - \varphi \overline{F}(T) J } | \mathcal{G}_T^{(k)}, \text{split of size $g$ at time $\rho T$}] &\sim \frac{  \mathbb{E}[ L^{(g)} \mathbb{E}[ e^{ -  \theta \overline{F}(T) Z_{T(1-\rho)} } ]^{L-g} e^{-\varphi \overline{F}(T) Z_{T(1-\rho)} } ] }{  \mathbb{E}[ L^{(g)} \mathbb{E}[ e^{ - \theta \overline{F}(T)  Z_{T(1-\rho)} } ]^{L-g} ]}.
\end{align*}
Now with $\lambda_T$ as in the proof of Lemma \ref{lem:cebolla}, we have $\mathbb{E}[ e^{ -  \theta \overline{F}(T) Z_{T(1-\rho)} }] = e^{ - (1 + o(1)) \lambda_T }$, so that 
\begin{align*}
\mathbb{Q}_{\theta \overline{F}(T),T}^{(k)} [ e^{ - \varphi \overline{F}(T) J } | \mathcal{G}_T^{(k)}, \text{split of size $g$ at time $\rho T$}]\sim \frac{  \mathbb{E}[ L^{(g)} e^{ - (  \lambda_T L + \varphi \overline{F}(T) ) L } ] }{  \mathbb{E}[ L^{(g)}e^{ - \lambda_T L }  ]}.
\end{align*}
Now using \eqref{tomate2} we have
\begin{align*}
\mathbb{Q}_{\theta \overline{F}(T),T}^{(k)} [ e^{ - \varphi \overline{F}(T) J } | \mathcal{G}_T^{(k)}, \text{split of size $g$ at time $\rho T$}]&\sim \left( \frac{\lambda_T}{\lambda_T+ \varphi \overline{F}(T) } \right)^{g-\alpha}\\
 &\sim \left( 1 + \varphi (1 - \rho + \theta^{1-\alpha})^{\frac{1}{\alpha-1}} \right)^{-(g-\alpha)}.
\end{align*}
where the final display follows from using the definition of $\lambda_T$ in the proof of the previous lemma. That completes the proof.
\end{proof}

We note that the previous lemma states that $\overline{F}(T)J$ is asymptotically distributed like a Gamma random variable with shape parameter $g-\alpha$ and scale parameter $(1-\rho+\theta^{1-\alpha})^{\frac{1}{\alpha-1}}$. More explicitly, we have the following immediate corollary.

\begin{corollary} \label{cor:splitsize}
Let $\alpha \in (1,2)$ and $g \geq 2$. Let $J_j$ denote the number of offspring at a splitting event of size $g_j$ at time $\rho_j T$. Then  
\begin{align*}
&\mathbb{Q}_{\theta \overline{F}(T),T}^{(k)} ( \overline{F}_T J_j \in \mathrm{d}x  | \mathcal{G}_T^{(k)}, \text{split of size $g_j$ at time $\rho_j T$}) = \Delta^\theta_{g_j,\rho_j} (x) \mathrm{d}x,
\end{align*}
where
\begin{align} \label{eq:Delta}
\Delta^\theta_{g,\rho}(x) := \Delta^\theta_{g,\rho,\alpha,\theta}(x) := \frac{ x^{g-\alpha-1} }{ \Gamma(g-\alpha) (1 - \rho+\theta^{1-\alpha})^{\frac{g-\alpha}{\alpha-1}}}  \exp \left\{  - \frac{x}{(1-\rho+\theta^{1-\alpha})^{\frac{1}{\alpha-1}} } \right\}
\end{align}
is the density function of $(1-\rho+\theta^{1-\alpha})^{\frac{1}{\alpha-1}} $ times a Gamma random variable with parameter $g-\alpha$.
\end{corollary}

Our final result tells us about the asymptotic distribution of the number of children descended from non-spine particles born at a splitting event. Recall from Section \ref{sec:lineage} that $\tilde{Z}_{T,j}$ counts the number of particles who are descended from non-spine particles born at the $j^{\text{th}}$ splitting event $\tau_j$.

\begin{lemma} \label{lem:spliteffect}
Let $\alpha \in (1,2]$. Then under $\mathbb{Q}_{\theta \overline{F}(T),T}^{(k)}$, the asymptotic distribution of $\tilde{Z}_{T,j}$ conditional on $\overline{F}(T)J_j = x_j$ is given by  
\begin{align*}
&\mathbb{Q}_{\theta \overline{F}(T),T}^{(k)} [ e^{ - \varphi \overline{F}(T) \tilde{Z}_{T,j} } | \mathcal{G}_T^{(k)}, \text{split of size $g$ at time $\rho T$}, \overline{F}(T)J_j = x_j ]\\
&\sim \exp \left\{ - \left( \frac{1}{(1 - \rho_j + (\varphi+\theta)^{1-\alpha} )^{\frac{1}{\alpha-1}} } -   \frac{1}{(1 - \rho_j + \theta^{1-\alpha} )^{\frac{1}{\alpha-1}} }  \right) x_j \right\},
\end{align*}
regardless of the size $g$ of the split. 
\end{lemma}
\begin{proof}
Recall that particles not carrying spines behave at time $t$ as if under the original measure $\mathbb{P}$ but with discounting by $e^{-\theta \overline{F}(T)Z_{T-t}}$ under $\mathbb{Q}^{(k)}_{\theta\overline{F}(T), T}$. In other words, if $Y_T$ denotes the number of descendants at time $T$ of a non-spine particle $u$ living at time $t$, we have
\begin{align*}
\mathbb{Q}_{\theta \overline{F}(T),T}^{(k)} [e^{- \varphi \overline{F}(T) Y_T } | \mathcal{F}_t^{(k)} ]  = F_T(e^{ - (\theta+ \varphi )\overline{F}(T) })/F_T(e^{-\theta \overline{F}(T)}).
\end{align*}
Moreover $Y_T$ is independent of the remainder of the population. 
In particular, since $\frac{x_j}{\overline{F}(T)} - g$ non-spine particles are born at time $\tau_j$, we have
\begin{align*}
&\mathbb{Q}_{\theta \overline{F}(T),T}^{(k)} [ e^{ - \varphi \overline{F}(T) \tilde{Z}_{T,j} } | \mathcal{G}_T^{(k)}, \text{split of size $g$ at time $\rho T$}, \overline{F}(T)J_j = x_j ] = \left( \frac{  F_T(e^{ - (\theta+ \varphi )\overline{F}(T) })}{ F_T(e^{-\theta \overline{F}(T)})} \right)^{\frac{x_j}{\overline{F}(T)} - g }.
\end{align*}
The result now follows from \eqref{eq:plantain}.
\end{proof}

We remark that in Lemma \ref{lem:spliteffect}, the linear dependence of the Laplace transform on $x$ is due to a branching property.

\section{Inverting the change of measure} \label{sec:inversion}

\subsection{The ancestral tree probability under $\mathbb{Q}_{\theta \overline{F}(T)}^{(k),T}$}

\begin{lemma}
Let $(\beta_0,\ldots,\beta_m)$ be a splitting process of $\{1,\ldots,k\}$. Let $0 < t_1 < \ldots < t_m < 1$. Then 
\begin{align}
& \lim_{T \to \infty} \mathbb{Q}^{(k)}_{\theta \overline{F}(T) , T}\Big(\tau_1/T \in \ud  t_1,\ldots,\tau_m/T \in \ud t_m, \mathcal{T}(\xi) = (\beta_0,\ldots,\beta_m) \Big) \nonumber \\
& =\frac{ \prod_{i=1}^m \ud t_i  }{ \mathbb{E}[W_{\alpha-1}^k e^{ - \theta W_{\alpha-1} } ] } \prod_{i=1}^m \frac{\alpha \Gamma(g_i-\alpha)}{\Gamma(2-\alpha)} (1 - t_i + \theta^{1-\alpha} )^{\frac{g_i-\alpha}{\alpha-1}}  \prod_{i=0}^m \theta^{-\alpha(g_i-1)} (\theta^{1-\alpha} + 1 - t_i )^{- \frac{\alpha}{\alpha-1}(g_i-1)} .
\end{align}
\end{lemma}

\begin{proof}
According to \eqref{eq:Qmainproj} we have 
\begin{align}
& \lim_{T \to \infty} \mathbb{Q}^{(k)}_{\theta \overline{F}(T) , T}\Big(\tau_1/T \in \ud  t_1,\ldots,\tau_m/T \in \ud T t_m, \mathcal{T}(\xi) = (\beta_0,\ldots,\beta_m) \Big) \nonumber \\
& = \lim_{T \to \infty} \frac{F_T'(e^{-\theta \overline{F}(T) }) }{\mathbb{E}[Z_T^{(k)}e^{ - \theta  \overline{F}(T)   Z_T }]}  \prod_{ i = 1}^m \mathbb{E}\left[ L^{(g_i)} \left(\mathbb{E}[e^{-\theta   \overline{F}(T)  Z_{T(1-t_i)}}]\right)^{L-g_i} \right] p_{\ell_i} F_{T(1-t_i)}'(e^{-\theta})^{g_i-1}. 
\end{align}
Now according to \eqref{kmomentsrho} and then \eqref{eq:firstmoment} we have $F'_{T(1-\rho)}(e^{-\theta \overline{F}(T) }) \sim (1+\theta^{\alpha-1}(1-\rho))^{- \frac{\alpha}{\alpha-1}}$. Using this fact in conjunction with the asymptotics in \eqref{cebolla} and \eqref{kmomentsrho} we obtain
\begin{align} \label{eq:bronze}
& \lim_{T \to \infty} \mathbb{Q}^{(k)}_{\theta \overline{F}(T) , T}\Big(\tau_1/T \in \ud  t_1,\ldots,\tau_m/T \in \ud T t_m, \mathcal{T}(\xi) = (\beta_0,\ldots,\beta_m) \Big) \nonumber \\
& = \lim_{T \to \infty} \prod_{i=1}^m \Big(rT\ud t_i\Big) \frac{  \overline{F}(T)^{(k-1)} }{ \mathbb{E}[W_{\alpha-1}^k e^{ - \theta W_{\alpha-1} } ] } \nonumber  \\
\times &\prod_{i=1}^m \frac{\Gamma(g_i-\alpha)}{\Gamma(-\alpha)} \left( \frac{ \overline{F}(T) }{ (1 - t_i + \theta^{1-\alpha} )^{\frac{1}{\alpha-1}} } \right)^{\alpha-g_i} \ell_2(T)^{(\alpha-1)}  \prod_{i=0}^m (1+\theta^{\alpha-1}(1-\rho))^{- \frac{\alpha}{\alpha-1}(g_i-1)} .
\end{align}
We now gather the asymptotic terms in $T$. Recall from \eqref{eq:tchaik} that $\overline{F}(T) \sim ((\alpha-1)r)^{- \frac{1}{\alpha-1} } T^{-\frac{1}{\alpha-1}} \ell_3(T)$. Using this fact in conjunction with the simple identity $\sum_{i=1}^m (g_i-1)=k-1$ we have 
\begin{align} \label{eq:copper}
T^m \overline{F}(T)^{k-1} \prod_{ i = 1}^m \overline{F}(T)^{\alpha-g_i} \ell_2(T)^{(\alpha-1)} \sim \frac{1}{((\alpha-1)r)^m}
\end{align}
as $T \to \infty$. Plugging \eqref{eq:copper} into \eqref{eq:bronze} and using the identity $(\alpha-1)\Gamma(-\alpha) = \Gamma(2-\alpha)/\alpha$, we obtain 
\begin{align} \label{eq:bronze2}
& \lim_{T \to \infty} \mathbb{Q}^{(k)}_{\theta \overline{F}(T) , T}\Big(\tau_1/T \in \ud  t_1,\ldots,\tau_m/T \in \ud T t_m, \mathcal{T}(\xi) = (\beta_0,\ldots,\beta_m) \Big) \nonumber \\
& =\frac{ \prod_{i=1}^m \ud t_i  }{ \mathbb{E}[W_{\alpha-1}^k e^{ - \theta W_{\alpha-1} } ] } \prod_{i=1}^m \frac{\alpha \Gamma(g_i-\alpha)}{\Gamma(2-\alpha)} (1 - t_i + \theta^{1-\alpha} )^{\frac{g_i-\alpha}{\alpha-1}}  \prod_{i=0}^m \theta^{-\alpha(g_i-1)} (\theta^{1-\alpha} + 1 - t_i )^{- \frac{\alpha}{\alpha-1}(g_i-1)} .
\end{align}
The result follows by aggregating the terms in the product and using the identity $\sum_{i=1}^m (g_i-1) = k-1$. 
\end{proof}

Combining the previous result and Corollary \ref{cor:splitsize}, we obtain the following result. 

\begin{corollary} \label{cor:grapefruit}
We have 
\begin{align}\label{comptree2}
& \lim_{T \to \infty} \mathbb{Q}^{(k)}_{\theta \overline{F}(T) , T}\left(\frac{\tau_1}{T} \in \ud  t_1,\ldots,\frac{\tau_m}{T }\in \ud  t_m, \mathcal{T}(\xi) = (\beta_0,\ldots,\beta_m) , \overline{F}(T) L_{\tau_1} \in \mathrm{d}x_1,\ldots,  \overline{F}(T) L_{\tau_m} \in \mathrm{d}x_m \right) \nonumber \\
&=  \frac{ \theta^{ - \alpha k} (\theta^{1-\alpha} + 1 )^{-\frac{\alpha}{\alpha-1}} }{ \mathbb{E}[W_{\alpha-1}^ke^{-\theta W_{\alpha-1}} ] } \prod_{ i=1}^m \frac{\alpha \Gamma(g_i-\alpha)}{\Gamma(2-\alpha)} (1 - t_i + \theta^{1-\alpha} )^{ - g_i}  \mathrm{d}t_i \prod_{i=1}^m \Delta^\theta_{g_i,t_i}(\mathrm{d}x_i) .
\end{align}
\end{corollary}

\subsection{The conditional distribution of $Z_T$ given ancestral tree under $\mathbb{Q}_{\theta \overline{F}(T)}^{(k),T}$}
In the previous section we computed the $\mathbb{Q}_{\theta,\infty}^{(k)}$ probabilities associated with the limiting ancestral tree. In this section, we compute the conditional Laplace transform for the entire population $Z_T$ at time $T$ \emph{given} a certain ancestral tree with certain offspring sizes at splitting events. 

Let us begin by recalling from Section \ref{sec:lineage} the decomposition 
\begin{align} \label{eq:decomp2}
Z_T = k + \sum_{ j = 1}^k Z_T^{(j)} + \sum_{ j = 1}^m \tilde{Z}_{T,j} 
\end{align}
of the entire population at time $T$ into conditionally independent consistuents. Here $Z_T^{(j)}$ counts the number of individuals born off the spine of lineage $j$, and $\tilde{Z}_{T,j}$ counts the number of descendants of non-spine particles born at the time of the $j^{\text{th}}$ splitting event. Our next two lemmas characterise the conditional laws of $Z_T^{(j)}$ and $\tilde{Z}_{T,j}$ given the ancestral tree, split times, and offspring sizes.

It transpires that asymptotically under $\mathbb{Q}^{(k)}_{\theta\overline{F}(T), T}$, the spine subpopulations have a highly tractable form:

\begin{lemma}
Conditional on $\mathcal{G}_T^{(k)}$ given $\varsigma_j = \rho T$, we have 
\begin{align} \label{eq:lineage}
\lim_{T \to \infty} \mathbb{Q}^{(k),T}_{\theta\overline{F}(T)}  [ e^{ - \varphi \overline{F}(T) Z_T^{(j)} } | \mathcal{G}_T^{(k)}, \varsigma_j = \rho T ] = \left( \frac{ 1 + (1-\rho)\theta^{\alpha-1} }{ 1 + (1-\rho)(\theta+\varphi)^{\alpha-1} }  \right)^{\frac{\alpha}{\alpha-1}},
\end{align}
as $T \to \infty$. 
\end{lemma}

\begin{proof}
Note that according to Lemma \ref{lem:bots}, the rate of births occuring off the spine is independent of the number of spines follows a particle. 
In particular, it follows that the contribution to births off the spine along a lineage $[\rho T,T]$ is identical to as under the $1$-spine measure run across $[0,(1-\rho)T]$. It follows that 
\begin{align*}
\mathbb{Q}^{(k),T}_{\theta\overline{F}(T)}  [ e^{ - \varphi \overline{F}(T) Z_T^{(j)} } | \mathcal{G}_T^{(k)}, \varsigma_j = \rho T ]  =  \mathbb{Q}^{(1),T(1-\rho)}_{\theta\overline{F}(T)}  [ e^{ - \varphi \overline{F}(T) Z_{T(1-\rho)}  }  ] .
\end{align*}
Now appealing to the case $k=1$ of \eqref{newprob1} we have
\begin{align*}
\mathbb{Q}^{(1),T(1-\rho)}_{\theta\overline{F}(T)}  [ e^{ - \varphi \overline{F}(T) Z_{T(1-\rho)}  }  ]  = \frac{ \mathbb{E}[Z_{T(1-\rho)} e^{ - (\varphi+\theta) \overline{F}(T) Z_{T(1-\rho)} } ] }{  \mathbb{E}[Z_{T(1-\rho)} e^{ - \theta \overline{F}(T) Z_{T(1-\rho)} } ] }  .
\end{align*}
Using \eqref{eq:tchaik} to replace $\overline{F}(T)$ with $\overline{F}(T(1-\rho))$ and using the monotone convergence theorem we have 
\begin{align*}
\lim_{T \to \infty} \mathbb{Q}^{(1),T(1-\rho)}_{\theta\overline{F}(T)}  [ e^{ - \varphi \overline{F}(T) Z_{T(1-\rho)}  }  ] = \lim_{T \to \infty}  \frac{ \mathbb{E}[Z_{T(1-\rho)} e^{ - (\varphi+\theta) (1-\rho)^{\frac{1}{\alpha-1}} \overline{F}(T(1-\rho)) Z_{T(1-\rho)} } ] }{  \mathbb{E}[Z_{T(1-\rho)} e^{ - \theta (1-\rho)^{\frac{1}{\alpha-1}} \overline{F}(T(1-\rho)) Z_{T(1-\rho)} } ] }  .
\end{align*}
Now use \eqref{eq:firstmoment}.
\end{proof}

Let us mention here that a simple calculation using \eqref{eq:lineage} establishes the following corollary in the setting $\alpha=2$:

\begin{corollary}
Let $\alpha = 2$. Then the distribution of $Z_T^{(j)}$ converges in distribution as $T \to \infty$ to the sum of two independent exponential random variables with parameter $\frac{1}{1-\rho} + \theta$. 
\end{corollary}

We also recall now from Lemma \ref{lem:spliteffect} that
\begin{align} \label{eq:spliteffect2}
\lim_{T \to \infty} \mathbb{Q}_{\theta \overline{F}(T),T}^{(k)} [ e^{ - \varphi \overline{F}(T) \tilde{Z}_{T,j} } | \mathcal{G}_T^{(k)}, \text{split of size $g$ at time $\rho T$}, \overline{F}(T)J_j = x_j ]= \exp \left\{ - G(\rho_j,\varphi) x_j \right\},
\end{align}
where $G(\rho_j,\varphi) := \left( \frac{1}{(1 - \rho_j + (\varphi+\theta)^{1-\alpha} )^{\frac{1}{\alpha-1}} } -   \frac{1}{(1 - \rho_j + \theta^{1-\alpha} )^{\frac{1}{\alpha-1}} }  \right)  $. The main result of this section is a simple consequence of what we have seen in this section so far, and gives a characterisation of the asymptotic conditional law of $Z_T$ under $\mathbb{Q}_{\theta \overline{F}(T),T}^{(k)}$ after conditioning on the spine ancestral tree and the offspring sizes at split times. 

\begin{lemma} \label{lem:kiwi}
For short write $\Gamma_T := \{ \tau_1/T \in \ud  t_1,\ldots,\tau_m/T \in \ud  t_m, \mathcal{T}(\xi) = (\beta_0,\ldots,\beta_m) , \overline{F}(T) L_{\tau_1} \in \mathrm{d}x_1,\ldots,  \overline{F}(T) L_{\tau_m} \in \mathrm{d}x_m \}$.
Then
\begin{align} \label{eq:kiwi}
\lim_{T \to \infty} \mathbb{Q}_{\theta \overline{F}(T),T}^{(k)} [ e^{ - \varphi\overline{F}(T) Z_T} | \Gamma_T ] = \prod_{j=0}^m  \left( \frac{ 1 + (1-\rho_j)\theta^{\alpha-1} }{ 1 + (1-\rho_j)(\theta+\varphi)^{\alpha-1} }  \right)^{(g_j-1)\frac{\alpha}{\alpha-1}} ~ \exp \left( - \sum_{ j = 1}^m x_j G(\rho_j,\varphi) \right),
\end{align}
where we are using the convention $\rho_0 = 0$, $g_0 = 2$. 
\end{lemma}
\begin{proof}
Using the decomposition \eqref{eq:decomp2} we have 
\begin{align*}
\mathbb{Q}_{\theta \overline{F}(T),T}^{(k)} [ e^{ - \varphi \overline{F}(T) Z_T} | \Gamma ] &= e^{ - \theta \overline{F}(T) k } \prod_{ j = 1}^k 
\mathbb{Q}^{(k),T}_{\theta\overline{F}(T)}  [ e^{ - \varphi \overline{F}(T) Z_{T,j} } | \Gamma_T ] \prod_{j=1}^m \mathbb{Q}_{\theta \overline{F}(T),T}^{(k)} [ e^{ - \varphi \overline{F}(T) \tilde{Z}_{T,j} } | \Gamma_T ].
\end{align*}
Recall the definition of $\varsigma_j$ given in Section \ref{sec:lineage}. Each $\varsigma_j$ is the initial time of a lineage, and is equal to some $\tau_i$. Moreover, for each $i =0,\ldots,m$ there are exactly $g_i - 1$ different $j$ such that $\varsigma_j = \tau_i$. It follows using \eqref{eq:lineage} that
\begin{align*}
\lim_{T \to \infty}  \prod_{ j = 1}^k \mathbb{Q}^{(k),T}_{\theta\overline{F}(T)}  [ e^{ - \varphi \overline{F}(T) Z_T^{(j)} } | \Gamma_T ] = \prod_{j=0}^m  \left( \frac{ 1 + (1-\rho)\theta^{\alpha-1} }{ 1 + (1-\rho)(\theta+\varphi)^{\alpha-1} }  \right)^{(g_j-1)\frac{\alpha}{\alpha-1}} .
\end{align*}
As for the other term, by \eqref{eq:spliteffect2} we have immediately
\begin{align*}
\lim_{T \to \infty}  \prod_{j=1}^m \mathbb{Q}_{\theta \overline{F}(T),T}^{(k)} [ e^{ - \varphi \overline{F}(T) \tilde{Z}_{T,j} } | \Gamma_T ] = \exp \left( - \sum_{ j = 1}^m x_j G(\rho_j,\varphi) \right). 
\end{align*}
Combining the two equations completes the proof.
\end{proof}

\subsection{The joint law of the ancestral tree, split offspring sizes, and entire population}
With Corollary \ref{cor:grapefruit} and Lemma \ref{lem:kiwi} at hand, we now prove the following. 
\begin{lemma} \label{lem:joint}
With $\Gamma_T$ as in Lemma \ref{lem:kiwi} we have 
\begin{align*}
\lim_{T \to \infty}   \mathbb{Q}_{\theta \overline{F}(T),T}^{(k)} [ e^{ - (\varphi-\theta) \overline{F}(T) Z_T} ; \Gamma_T ] &= \frac{ 1 }{ \mathbb{E}[W^k_{\alpha-1}e^{ - \theta W_{\alpha-1}} ]} \varphi^{ - \alpha k } ( 1 + \varphi^{1-\alpha})^{ - \frac{\alpha}{\alpha-1}}\\
&\times \prod_{i=1}^m 
 \frac{\alpha \Gamma(g_i-\alpha)}{\Gamma(2-\alpha)}  (1 - t_i + \varphi^{1-\alpha} )^{-g_i} \Delta_{g_i,t_i}^\varphi(x_i).
\end{align*}

\end{lemma}
\begin{proof}
Simply multiplying the main equations of Corollary \ref{cor:grapefruit} and Lemma \ref{lem:kiwi} (and carefully separating the $j=0$ term out in the product in \eqref{eq:kiwi}), we have immediately
\begin{align} \label{eq:elanga1}
\mathbb{Q}_{\theta \overline{F}(T),T}^{(k)} [ e^{ - (\varphi - \theta)\overline{F}(T) Z_T} ; \Gamma_T ] = \frac{\theta^{-\alpha k} (\theta^{1-\alpha}+1)^{-\frac{\alpha}{\alpha-1}}  }{ \mathbb{E}[W_{\alpha-1}e^{ - \theta W_{\alpha-1}} ]} \left( \frac{1+\theta^{\alpha-1}}{1 + \varphi^{\alpha-1}} \right)^{\frac{\alpha}{\alpha-1}} \prod_{ i =1}^m R_i \mathrm{d}x_i \mathrm{d}t_i,
\end{align}
\begin{align*}
R_i = \frac{\alpha \Gamma(g_i-\alpha)}{\Gamma(2-\alpha)}(1 - t_i + \theta^{1-\alpha} )^{-g_i} \left( \frac{1 + (1-t_i) \theta^{\alpha-1}}{ 1+ (1-t_i) \varphi^{\alpha-1} } \right)^{(g_i-1)\frac{\alpha}{\alpha-1}} \Delta^\theta_{g_i,t_i}(x_i).
\end{align*}
Using the definition \eqref{eq:Delta} of $\Delta^\theta_{g,\rho}(x)$, a calculation verifies that
\begin{align} \label{eq:elanga2}
R_i =  \frac{\alpha \Gamma(g_i-\alpha)}{\Gamma(2-\alpha)} (\theta/\varphi)^{(g_i-1)\alpha} (1 - t_i + \varphi^{1-\alpha} )^{-g_i} \Delta_{g_i,t_i}^\varphi(x_i),
\end{align}
where we note the latter expression above involves $\Delta_{g_i,t_i}^\varphi$ as opposed to $\Delta_{g_i,t_i}^\theta$.

Plugging \eqref{eq:elanga1} into \eqref{eq:elanga2}, and making good use of the identity $\sum_{i=1}^m (g_i-1) = k - 1$, we obtain the result.

\end{proof}

The reader will note from Lemma \ref{lem:joint}, we have
\begin{align*}
\mathbb{Q}_{\theta \overline{F}(T),T}^{(k)} [ e^{ - (\varphi-\theta)\overline{F}(T) Z_T} ; \Gamma_T ] = \mathbb{Q}_{\varphi \overline{F}(T),T}^{(k)} [ \Gamma_T ] \frac{ \mathbb{E}[W_{\alpha-1}e^{ - \varphi W_{\alpha-1}} ]}{\mathbb{E}[W_{\alpha-1}e^{ - \theta W_{\alpha-1}} ]}. 
\end{align*}
This identity may alternatively be derived from considering the respective changes of measure.

\subsection{Tree probabilities for uniform choice under $\mathbb{P}^{(k)}_{{\rm unif}, T}$.}

We are almost ready to prove the main results, Theorems \ref{thm:A}, Theorem \ref{thm:B} and Theorem \ref{thm:C}.
We begin with the following formulation.
\begin{theorem}
We have
\begin{align} \label{eq:vigil}
&\lim_{T \to \infty} \mathbb{P}^{(k)}_{{\rm unif}, T}\Big(  \mathcal{T}(\xi) = (\beta_0,\ldots,\beta_m) , \tau_i/T \in \ud  t_i,  \overline{F}(T) L_{\tau_i} \in \mathrm{d}x_i, i =1,\ldots,m  |Z_T\ge k\Big) \nonumber \\
&=\frac{1}{ (k-1)!} \int_0^\infty  \varphi^{ - (\alpha-1) k - 1 } ( 1 + \varphi^{1-\alpha})^{ - \frac{\alpha}{\alpha-1}} \prod_{i=1}^m  \frac{\alpha \Gamma(g_i-\alpha)}{\Gamma(2-\alpha)}  (1 - t_i + \varphi^{1-\alpha} )^{-g_i} \Delta_{g_i,t_i}^\varphi(x_i) ~
\mathrm{d}\varphi,
\end{align}
where recalling \eqref{eq:Delta} we have
\begin{align} \label{eq:Delta2}
\Delta^\varphi_{g,\rho}(x) := \frac{ x^{g-\alpha-1} }{ \Gamma(g-\alpha) (1 - \rho+\varphi^{1-\alpha})^{\frac{g-\alpha}{\alpha-1}}}  \exp \left\{  - \frac{x}{(1-\rho+\varphi^{1-\alpha})^{\frac{1}{\alpha-1}} } \right\}.
\end{align}
\end{theorem}

\begin{proof}
Taking an asymptotic version of \eqref{punif1} (simply replacing $Z_T^{(k)}$ with $Z_T^k$) with $\theta \overline{F}(T)$ in place of $\theta$ we have 
\begin{equation}\label{punif3}
\mathbb{E}^{(k)}_{{\rm unif},T }\left[ f(\underline{u})\Big|Z_T\ge k\right] \sim \mathbb{E}\left[Z^{k}_T e^{-\theta  \overline{F}(T) Z_T}\Big|Z_T\ge k\right] \mathbb{Q}^{(k)}_{\theta, T}\left[\frac{f(\xi_T)}{Z^k_T e^{-\theta \overline{F}(T) Z_T}}\right].
\end{equation}
Now according to the Gamma integral we have $\frac{1}{z^k} = \frac{1}{(k-1)!} \int_0^\infty \varphi^{k-1} e^{ - \varphi z} \mathrm{d} \varphi$, so that using Fubini's theorem we may instead write 
\begin{equation}\label{punif4}
\mathbb{E}^{(k)}_{{\rm unif},T }\left[ f(\underline{u})\Big|Z_T\ge k\right] \sim \frac{1}{(k-1)!} \mathbb{E}\left[Z^{k}_T e^{-\theta  \overline{F}(T) Z_T}\Big|Z_T\ge k\right] \int_0^\infty \varphi^{k-1}  \mathbb{Q}^{(k)}_{\theta, T}\left[ f(\xi_T) e^{Z_T } e^{ - (\varphi-\theta \overline{F}(T) ) Z_T} \right]\mathrm{d}\varphi.
\end{equation}
Changing variable from $\varphi$ to $\varphi \overline{F}(T)$ we obtain
\begin{equation}\label{punif5}
\mathbb{E}^{(k)}_{{\rm unif},T }\left[ f(\underline{u})\Big|Z_T\ge k\right] \sim \frac{\overline{F}(T)^k}{(k-1)!} \mathbb{E}\left[ Z^{k}_T e^{-\theta  \overline{F}(T) Z_T}\Big|Z_T\ge k\right] \int_0^\infty \varphi^{k-1}  \mathbb{Q}^{(k)}_{\theta, T}\left[ f(\xi_T)  e^{ - (\varphi-\theta) \overline{F}(T) ) Z_T} \right]\mathrm{d}\varphi.
\end{equation}
Now using the case $\rho = 0$ \eqref{kmomentsrho} this reduces to 
\begin{equation}\label{punif6}
\mathbb{E}^{(k)}_{{\rm unif},T }\left[ f(\underline{u})\Big|Z_T\ge k\right] \sim \frac{\mathbb{E}[W_{\alpha-1}^k e^{ - \theta W_{\alpha-1} } ] }{(k-1)!} \int_0^\infty \varphi^{k-1}  \mathbb{Q}^{(k)}_{\theta, T}\left[ f(\xi_T)  e^{ - (\varphi-\theta) \overline{F}(T) ) Z_T} \right]\mathrm{d}\varphi.
\end{equation}
Now we consider the case where $f(\underline{u})$ is the indicator function of the event $\{ \mathcal{T}(\xi) = (\beta_0,\ldots,\beta_m) , \tau_i/T \in \ud t_i, \overline{F}(T) L_{\tau_i} \in \mathrm{d}x_i, i =1,\ldots,m \}$. Using Lemma \ref{lem:joint} and the bounded convergence theorem we obtain  
\begin{align}\label{punif6}
&\lim_{T \to \infty} \mathbb{E}^{(k)}_{{\rm unif},T }\left[ f(\underline{u})\Big|Z_T\ge k\right] \nonumber \\
 &=  \frac{ 1 }{ (k-1)!} \int_0^\infty  \varphi^{k-1} \varphi^{ - \alpha k } ( 1 + \varphi^{1-\alpha})^{ - \frac{\alpha}{\alpha-1}} \prod_{i=1}^m  \frac{\alpha \Gamma(g_i-\alpha)}{\Gamma(2-\alpha)}  (1 - t_i + \varphi^{1-\alpha} )^{-g_i} \Delta_{g_i,t_i}^\varphi(x_i) ~
\mathrm{d}\varphi,
\end{align}
thereby completing the proof.

\end{proof}

Consider now taking the change of variable $w = (1 + \varphi)^{1-\alpha}$. The following equations are easily verified
\begin{align}
\varphi^{-(\alpha-1)} &= \frac{1-w}{w} \label{X1},\\
\frac{\mathrm{d}\varphi}{\varphi} &= \frac{ \mathrm{d}w }{ (\alpha-1) w (1-w) }. \label{X2}
\end{align}
Using \eqref{X1} and \eqref{X2} in \eqref{eq:vigil} we obtain
\begin{align} \label{eq:vigil3}
&\lim_{T \to \infty} \mathbb{P}^{(k)}_{{\rm unif}, T}\Big( \{\mathcal{T}(\xi) = (\beta_0,\ldots,\beta_m) , \tau_i/T \in \ud  t_i,  \overline{F}(T) L_{\tau_i} \in \mathrm{d}x_i, i =1,\ldots,m  \}  |Z_T\ge k\Big) \nonumber \\
&= \frac{ 1  }{ (\alpha-1) (k-1)!} \int_0^1 \frac{\mathrm{d}w}{w(1-w)} \left( \frac{1-w}{w} \right)^k w^{\frac{\alpha}{\alpha-1}} \prod_{i=1}^m  \frac{\alpha \Gamma(g_i-\alpha)}{\Gamma(2-\alpha)}  (\frac{1}{w} - t_i )^{-g_i} \Delta_{g_i,t_i}^w(x_i),
\end{align}
where we are abusing notation slightly and writing 
\begin{align} \label{eq:Delta3}
\Delta^w_{g,\rho}(x) := \frac{ x^{g-\alpha-1} }{ \Gamma(g-\alpha) (1/w - \rho )^{\frac{g-\alpha}{\alpha-1}}}  \exp \left\{  - \frac{x}{(1/w-\rho)^{\frac{1}{\alpha-1}} } \right\}
\end{align}
for the probability density function of $(1/w-\rho)^{\frac{1}{\alpha-1}}$ times a Gamma random variable with shape $g-\alpha$.
Tidying \eqref{eq:vigil3} and using the identity $\sum_{i=1}^m (g_i-1)= k-1$, we ultimately obtain
\begin{align} \label{eq:vigil4}
& \lim_{T \to \infty} \mathbb{P}^{(k)}_{{\rm unif}, T}\Big( \{\mathcal{T}(\xi) = (\beta_0,\ldots,\beta_m) , \tau_i/T \in \ud  t_i,  \overline{F}(T) L_{\tau_i} \in \mathrm{d}x_i, i =1,\ldots,m  \}  |Z_T\ge k\Big) \nonumber \\
&=  \frac{1}{ (\alpha-1) (k-1)!} \int_0^1 (1-w)^{k-1} w^{m + \frac{2-\alpha}{\alpha-1} }  \prod_{i=1}^m  \frac{\alpha \Gamma(g_i-\alpha)}{\Gamma(2-\alpha)}  (1 - w t_i )^{-g_i} \Delta_{g_i,t_i}^w(x_i) ~  \mathrm{d}w.
\end{align}

Let us now recapitulate, and prove Theorem \ref{thm:A}, Theorem \ref{thm:B} and Theorem \ref{thm:C} explicitly. 
\begin{proof}[Proof of Theorems \ref{thm:A}, \ref{thm:B} and \ref{thm:C}]
Consider a continuous-time Galton-Watson tree with offspring distribution in the universality class \eqref{eq:hyp1}. Under a probability measure $ \mathbb{P}^{(k)}_{{\rm unif}, T}( \cdot | Z_T \geq k )$, condition on the event $\{Z_T \geq k \}$ that there are at least $k$ particles alive at time $T$, and sample $k$ particles uniformly from the population at time $T$. Let $(\pi^{(k,T)}_t)_{t \in [0,T]}$ denote the joint ancestral process of these particles. Let $\mathcal{T}(\pi^{(k,T)})$ denote the splitting process associated with this ancestry, and let $\tau_1,\ldots,\tau_m$ denote the split times. Finally, let $L_1,\ldots,L_m$ denote the offspring sizes at these split times.

Asymptotically in $T$, the equation \eqref{eq:vigil4} holds, which in particular does not depend on the explicit form of the offspring distribution, but only the parameter $\alpha$ governing the universality class \eqref{eq:hyp1}. As such, that proves Theorem \ref{thm:A}.

Next we note that the convergence in distribution for such trees in this universality class proves Theorem \ref{thm:C}.

Finally, we note that Theorem \ref{thm:B} is obtained from Theorem \ref{thm:C} by integrating against $x_i$. 

\end{proof}
\subsection{The Lauricella representation} \label{sec:lauricellaproof}

In this section we derive the Lauricella representation for the joint density of the split times. 
According to Theorem \ref{thm:B} we have
\begin{align} \label{eq:vigil99}
&\mathbb{P}\left( \mathcal{T}(\nu) = \overline{\beta} , \tau_1 \in \mathrm{d}t_1,\ldots,\tau_m \in \mathrm{d}t_m\right) \nonumber \\
&=  \frac{1}{ (\alpha-1) (k-1)!}  \prod_{i=1}^m  \frac{\alpha \Gamma(g_i-\alpha)}{\Gamma(2-\alpha)}   \int_0^1 (1-w)^{k-1} w^{m + \frac{2-\alpha}{\alpha-1} } \prod_{i=1}^m  (1 - w t_i )^{-g_i}  ~  \mathrm{d}w.
\end{align}
We now derive our alternative representation for the right hand side. Let us begin by noting that 
\begin{align} \label{eq:andres}
\int_0^1 w^{m+\frac{2-\alpha}{\alpha-1}}&(1-w)^{k-1}\prod_{i=1}^m (1-wt_i)^{-g_i}\ud w  \nonumber  \\
&=\int_0^1\ud w\, w^{m+\frac{2-\alpha}{\alpha-1}}(1-w)^{k-1}\sum_{j_i,\ldots, j_m=0}^m \frac{(g_1)_{j_1}\cdots(g_m)_{j_m} }{j_1!\cdots j_m!}(wt_1)^{j_1}\cdots (wt_m)^{j_m} \nonumber \\
&=\sum_{j_i,\ldots, j_m=0}^m \frac{(g_1)_{j_1}\cdots(g_m)_{j_m} }{j_1!\cdots j_m!}t_1^{j_1}\cdots t_m^{j_m}\int_0^1 w^{m+\frac{1}{\alpha-1}+\sum_{i=1}^m j_i-1}(1-w)^{k-1}\ud w  \nonumber  \\
&=\frac{\Gamma(m+\frac{1}{\alpha-1})\Gamma(k)}{\Gamma(k+m+\frac{1}{\alpha-1})}\sum_{j_i,\ldots, j_m=0}^m \frac{(m+\frac{1}{\alpha-1})_{j_1+\cdots+j_m}(g_1)_{j_1}\cdots(g_m)_{j_m} }{(k+m+\frac{1}{\alpha-1})_{j_1+\cdots+j_m}j_1!\cdots j_m!}t_1^{j_1}\cdots t_m^{j_m}  \nonumber  \\
&=\frac{\Gamma(m+\frac{1}{\alpha-1})\Gamma(k)}{\Gamma(k+m+\frac{1}{\alpha-1})}F^{(m)}_D\left[m+\frac{1}{\alpha-1}, g_1, \ldots, g_m; k+m+\frac{1}{\alpha-1}; t_1, \ldots, t_m\right],
\end{align}
where $F^{(m)}_D$ is the Lauricella hypergeometric function in $m$ variables $t_1, \ldots, t_m$ which was introduced by Lauricella \cite{Lauri}.  We note now that plugging \eqref{eq:andres} into \eqref{eq:vigil99} we obtain \eqref{eq:vigil_laur} as stated in the introduction.

We now appeal to a probabilistic representation of Chamayou and Wesolowski \cite{CW}. To set this up, let $(X_1, \ldots, X_n)$ be a random vector with  Dirichlet distribution with parameters $a=(a_1,\ldots, a_n)$ and $b>0$,  here denoted by ${\tt Dir}(a;b)$. In other words, its distribution is absolutely continuous with respect to the Lebesgue measure on $\mathbb{R}^d$ and
\[
D_n(a;b;x)=C\left(1-\sum_{i=1}^n x_i\right)^{b-1}\prod_{i=1}^n x_{i}^{a_i-1}\mathbf{1}_{\mathbb{T}_n}(x),
\]
where $x=(x_1,\ldots, x_n)\in \mathbb{R}^n$, 
\[
\mathbb{T}_n=\left\{(x_1,\ldots, x_n): x_i>0, i=1,\ldots, n, \sum_{i=1}^n x_i<1\right\}\quad \textrm{and}\quad C=\frac{\Gamma(b+\sum_{i=1}^n a_i)}{\Gamma(b)\Gamma(\sum_{i=1}^n a_i)}.
\]
An important property of the Dirichlet distribution is that it can be represented through independent gamma distributions, that is let $U_1, \ldots, U_{n}$ be independent Gamma r.v.'s with parameters $(\sigma, a_i)$, for $i=1,\ldots, n$, then 
\[
(X_1, \ldots, X_n)\stackrel{d}{=} \frac{(U_1, \ldots, U_{n})}{\sum_{i=1}^n U_i},
\]
where the latter is independent of $\sum_{i=1}^n U_i$.

Moreover, the Laplace exponent of the Dirichlet distribution satisfies 
\[
\mathbb{E}\left[e^{\langle t, X\rangle}\right]=C\int_{\mathbb{T}_n} e^{\langle t, x\rangle} \left(1-\sum_{i=1}^n x_i\right)^{b-1}\prod_{i=1}^n x_{i}^{a_i-1}\ud x_1 \cdots \ud x_n=\Phi^{(n)}_2(a;b;t),
\]
where $t=(t_1,\ldots, t_n)$. The function $\Phi^{(n)}_2$ can be viewed as multivariate version of the hypergeometric function ${}_1F_1$, i.e.
\[
\Phi^{(1)}_2(a,b,t)={}_1F_1(a; a+b; t),
\]
where the right-hand side is the Laplace transform of a beta r.v. with parameters $(a,b)$.

Next let $Z$ be a Gamma r.v. with parameters $(1,c)$. If $X\sim {\tt Dir}(a;b)$ and $X$ and $Z$ are independent, then the random vector $Y=ZX$ satisfies
\[
\mathbb{E}\Big[\exp\{\langle t, Y\rangle\}\Big]= F^{(n)}_D (c,a;b;t).
\]
Thus conditioning with respect to $Z$, we have
\[
F^{(n)}_D (c,a;b;t)=\frac{1}{\Gamma(c)}\int_0^\infty z^{c-1}e^{-z} \Phi_2^{(n)}(a;b;tz) \ud z.
\]
Conditioning with respect to $X$, we have
\begin{equation} \label{eq:ocosto}
F^{(n)}_D (c,a;b;t)=\int_{\mathbb{T}_n}\Big(1-\langle t,x\rangle\Big)^{-c} D_n(a; b;x)\ud x.
\end{equation}
Using \eqref{eq:ocosto} in \eqref{eq:andres} we obtain
\begin{align*}
&\int_0^1 w^{m+\frac{2-\alpha}{\alpha-1}}(1-w)^{k-1}\prod_{i=1}^m (1-wt_i)^{-g_i}\ud w  \nonumber  \\
&= \frac{\Gamma(m+\frac{1}{\alpha-1})\Gamma(k)}{\Gamma(k+m+\frac{1}{\alpha-1})}\int_{\mathbb{T}_m} \left( 1 - \langle t, x \rangle \right)^{ - m - \frac{1}{\alpha-1}} D_m( g_1,\ldots,g_m,k+m+\frac{1}{\alpha-1} ; \mathrm{d}x).
\end{align*}
We finally note that if $(E_1,\ldots,E_m)$ is distributed according to $ D_m( g_1,\ldots,g_m,k+m+\frac{1}{\alpha-1} ; \mathrm{d}x)$, then we have the identity in distribution
\begin{align*}
E_i = \frac{W_i}{W_1+\ldots+W_m+Q} \qquad W_i \sim \Gamma(g_i), Q \sim \Gamma(k+m+\frac{1}{\alpha-1}).
\end{align*}
As such we may instead write 
\begin{align*}
\int_0^1 w^{m+\frac{2-\alpha}{\alpha-1}}(1-w)^{k-1}\prod_{i=1}^m (1-wt_i)^{-g_i}\ud w
= \frac{\Gamma(m+\frac{1}{\alpha-1})\Gamma(k)}{\Gamma(k+m+\frac{1}{\alpha-1})} \mathbb{E}\left[ \left( 1 - \frac{ t_1 W_1 + \ldots + t_m W_m }{ W_1 + \ldots + W_m + Q} \right)^{-m-\frac{1}{\alpha-1}} \right].
\end{align*}

Using \eqref{eq:vigil99} we obtain 
\begin{align} \label{eq:vigil990}
&\mathbb{P}\left( \mathcal{T}(\nu) = \overline{\beta} , \tau_1 \in \mathrm{d}t_1,\ldots,\tau_m \in \mathrm{d}t_m\right) \nonumber \\
&=  \frac{1}{ (\alpha-1) (k-1)!}  \prod_{i=1}^m  \frac{\alpha \Gamma(g_i-\alpha)}{\Gamma(2-\alpha)}  \frac{\Gamma(m+\frac{1}{\alpha-1})\Gamma(k)}{\Gamma(k+m+\frac{1}{\alpha-1})} \mathbb{E}\left[ \left( 1 - \frac{ t_1 W_1 + \ldots + t_m W_m }{ W_1 + \ldots + W_m + Q} \right)^{-m-\frac{1}{\alpha-1}} \right],
\end{align}
where as above, $W_1$ are Gamma distrbuted with parameter $g_i$, and $Q$ is Gamma distributed with parameter $k+m+\frac{1}{\alpha-1}$.

\end{document}